\documentclass{article}

\usepackage{amsthm,amsmath,amsfonts,amssymb}
\usepackage[authoryear]{natbib}
\usepackage[colorlinks,citecolor=blue,urlcolor=blue, linkcolor=blue]{hyperref}
\usepackage{graphicx}
\usepackage{subcaption}			     
\usepackage{bbm}					     
\usepackage{dutchcal}				   
\usepackage{enumitem}       	     
\usepackage{fontawesome5}   	
\usepackage{xcolor}         			 
\usepackage{authblk}					
\usepackage{orcidlink}					
\usepackage{geometry}
\newtheorem{theorem}{Theorem}[section]
\newtheorem{prop}{Proposition}[section]
\newtheorem{remark}{Remark}[section]
\newtheorem{lemma}{Lemma}[section]
\newtheorem{cor}{Corollary}[section]
\newtheorem*{assL}{Assumptions L}
\newtheorem*{assM}{Assumptions M}
\newtheorem*{assP}{Assumption P}

\numberwithin{equation}{section}

\newcommand{\N}{\mathbb{N}}
\newcommand{\R}{\mathbb{R}}
\newcommand{\Prob}{\mathbb{P}}
\newcommand{\abs}[1]{\left\vert #1 \right\vert}
\newcommand{\norm}[1]{\left\| #1 \right\|}

\definecolor{amarelo}{HTML}{F4D166}
\definecolor{laranxa}{HTML}{EB6349}
\definecolor{vermello}{HTML}{B71D3E}

\title{Granulometric Smoothing on Manifolds}
\author[1]{Diego Bolón}
\author[2]{Rosa M. Casais} 
\author[2]{Alberto Rodríguez-Casal}
\affil[1]{ECARES and Department of Mathematics, Université libre de Bruxelles, Brussels, 1050, Belgium} 
\affil[2]{Galician Center for Mathematical Research and Technology, CITMAga, Universidade de Santiago de Compostela, Santiago de Compostela, 15782, Spain} 
\date{}

\begin{document}
\maketitle

\begin{abstract}
	Given a random sample from a density function supported on a manifold $M$, a new method for the estimating highest density regions of the underlying population is introduced. The new proposal is based on the empirical version of the opening operator from mathematical morphology combined with a preliminary estimator of the density function. This results in an estimator that is easy-to-compute since it simply consists of a list of centers and a radius $r$ that are adequately selected from the data. The new estimator is shown to be consistent and its convergence rates in terms of the Hausdorff distance are provided. All consistency results are established uniformly on the level of the set and for any Riemannian manifold $M$ satisfying mild assumptions. The applicability of the procedure is shown by some illustrative examples.
\end{abstract}

\section{Introduction}
\label{sec:intro}


In data analysis, one of the first questions that the statistician has to answer is: ``Where are the data concentrated?''
Some of the statistical techniques that provide an answer to the previous question include density estimation \citep{Parzen1962, Cao1994, Wand1994, Loader1996}, mode hunting \citep{Silverman1981, Hartigan1985, Minnotte1993, Chaudhuri1999, Burman2009, AmeijeirasAlonso2019}, support estimation \citep{Devroye1980, Mammen1995, Cuevas1997, Klemelae2004, Biau2008} or cluster analysis \citep{Hartigan1975, Cuevas2001, Chacon2013, Wierzchon2018}.
For most of these exploratory techniques, highest density regions (HDRs in short) play a relevant role.
HDRs are the sets where the density function of the data exceeds a given level. More precisely, if $f$ is the population density function, HDRs are the sets such that
\begin{equation}
	\label{eq:HDR}
	L (\lambda) = f^{-1} \big( [\lambda, + \infty) \big), \quad \lambda > 0.
\end{equation}
Modes can be seen as limiting cases of HDRs where the value of $\lambda$ converges to the global maximum of the density function. In practice, the specific value of the level $\lambda$ is often irrelevant and the interest lies in recovering the HDR that satisfies a given probability content. That is, given any $\gamma \in (0, 1)$, the objective is to estimate the set $L (\lambda_{\gamma})$, where $\lambda_{\gamma} > 0$ is such that
\begin{equation}
	\label{eq:HDRgamma}
	\Prob \big[ L (\lambda_{\gamma}) \big] = 1 - \gamma.
\end{equation}

Values of $\gamma$ close to $0$ lead to large HDRs that approach to the population support (since their probability content is close to one), and, therefore, their corresponding level $\lambda_{\gamma}$ is close to zero. Estimation of these sets provides a natural way of detecting outliers \citep{Devroye1980, Baillo2001}. Reciprocally, if $\gamma$ is close to $1$, the corresponding level $\lambda_{\gamma}$ is large and $L (\lambda_{\gamma})$ converges to the largest modes of the population. This connection allows to investigate multimodality via HDR estimation, see \cite{Muller1991}, \cite{Cuevas2001}, and \cite{Rinaldo2010}. 

HDR estimation provides an easy-to-understand way to visualize a density in a general setting. Indeed, there are several practical problems regarding non-Euclidean data where HDR estimators may provide useful information. One example, which will be analyzed in detail, is the study of long-period comet orbits, a task previously addressed by \cite{Jupp2003}. Long-period comets (comets with periods larger than 200 years) are assumed to originate in a celestial formation called Oort cloud. The Oort cloud \citep{Weissman1990} is a theoretical cloud of planetesimals supposed to surround the Solar System at very large distances. It is though to be sphere shaped, with the asteroids roughly uniformly distributed within it. The Oort cloud is influenced by galactic tidal forces and gravity perturbations due to passing stars, bringing some of these planetesimals close enough to the sun to become long-period comets (see \citealp{Heisler1986}). This formation process explains why the orbits of these comets are nearly uniformly distributed, while short-period comets are highly concentrated around the ecliptic plane.

The uniformity of orbits can be assessed by testing the uniformity of their unit normal vectors. This way, each orbit orientation corresponds to a point on the 2-dimensional sphere $\mathbb{S}^2 = \{ x \in \R^3 \colon \norm{x} = 1\}$, and one can address this problem by testing the uniformity of a given sample in $\mathbb{S}^2$. Figure~\ref{fig:comets} shows the normal unit vectors of all the long-period comets detected up to 30th of May 2022. A more detailed explanation of these data can be found in Section~\ref{sec:comets}.
\begin{figure}[!b]
	\centering
	\includegraphics[width = 2in]{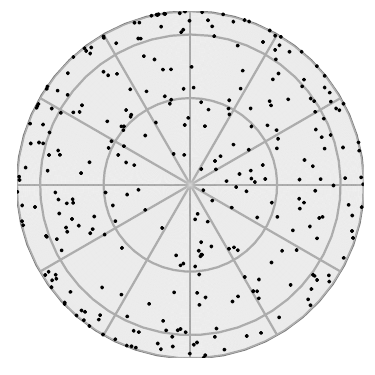}
	\hspace{0.5in}
	\includegraphics[width = 2in]{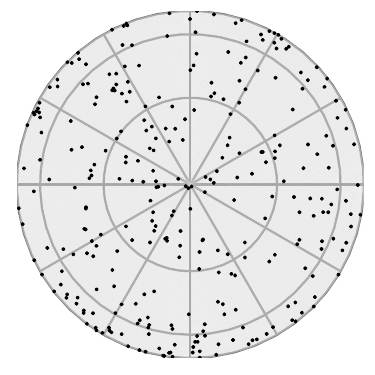}
	\caption{\label{fig:comets} Comet orbits data represented on the sphere using orthogonal projections centered on the north pole (left) and south pole (right).}
\end{figure}

The problem of testing the uniformity of the long-period comets has been a widely considered example in the statistical literature, and some authors like \cite{GarciaPortugues2023} noticed that their unit vectors do not follow a uniform spherical distribution. This is suggested to be caused by an observational bias (see \citealp{Jupp2003}). When astronomers look for objects in the space, they do not search uniformly in all the celestial sphere. Actually, they tend to focus on where it is more likely to find something, that is, the ecliptic plane. Since most of the observational effort is concentrated around it, long-period comets near the ecliptic plane are more likely to be detected, and therefore, they are overrepresented in the observed sample. 

One possible way to check this observational bias hypothesis is via HDR estimation. If the HDRs of the observed normal unit vectors are located around the north and south poles of the sphere, then this will mean that comets near the ecliptic are overrepresented in the registered data, which agrees with the observational bias theory. 

Furthermore, this analysis cannot be performed ignoring the geometry of the data support. See, for example, Figure~\ref{fig:mercomets}, where the orientations of the comet orbits are represented in spherical coordinates, that is, each unit vector $v \in \mathbb{S}^2$ is represented as the pair $(\varphi, \psi) \in [-\pi, \pi] \times [-\pi/2, \pi/2]$ such that $v = \big( \cos(\varphi) \cos(\psi), \sin(\varphi) \cos(\psi), \sin(\psi) \big)'$. This representation is known to distort distances in the sphere, especially in the regions near the north and south poles. Since these areas are precisely the areas of interest in this situation, any conclusion based on such a ``flat'' representation of the data will be misleading, and an HDR technique that takes into account the natural curvature and distance of the sphere should be employed.

Of course, this problematic is not specific to the sphere and data supported on any non-flat Riemannian manifold will show the same behavior. This kind of datasets are a current trend in modern science, with practical examples on the sphere, as previously illustrated, but also on the torus (\citealp{DiMarzio2011}, \citealp{AmeijeirasAlonso2020}), cylinder (\citealp{Jammalamadaka2006}, \citealp{AlonsoPena2023a}) and polysphere (\citealp{GarciaPortugues2023a}). Therefore, a global theory of HDR estimation is required, establishing consistency results for any Riemannian manifold $M$.

\begin{figure}[t]
	\centering
	\includegraphics[width = 3in]{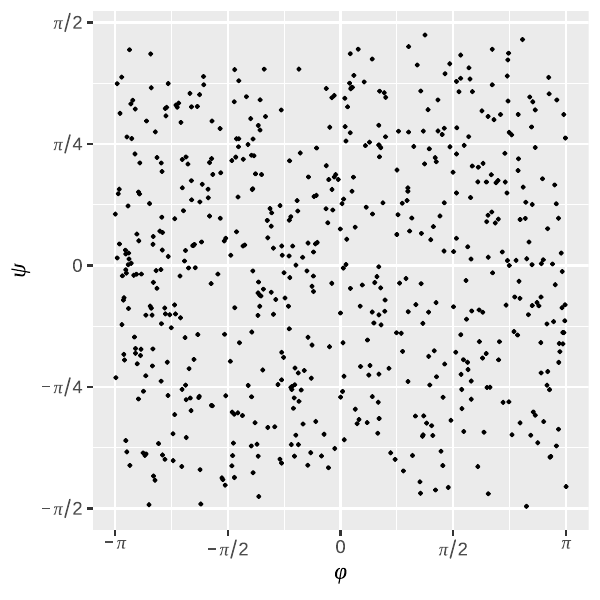}
	\caption{\label{fig:mercomets} Comet orbits data represented in spherical coordinates. $\varphi$ is the longitude, whereas $\psi$ is the latitude. This representation distort areas and distances: points with $\psi$ close to $\pm \pi/2$ appear further apart than they really are.}
\end{figure}

Due to its practical utility, HDR estimation in the Euclidean framework has been widely addressed in the literature. Some references in this topic include \cite{Polonik1995}, \cite{Tsybakov1997}, \cite{Cadre2006} and \cite{RodriguezCasal2022}. However, HDR estimation in other settings has not been considered until very recently. \cite{Cuevas2006} prove some theoretical results concerning the consistency and convergence rates of level set estimators in metric spaces (for a general function, not necessarily a density). \cite{SaavedraNieves2021} provide an estimator for HDRs on the $d$-dimensional sphere; while \cite{Jiang2017} and \cite{Cholaquidis2022} introduce two different HDR estimators for compact manifolds.  

All the contributions on HDR estimation for manifold data cited above present strengths and shortcomings. For instance, \cite{Jiang2017} proposes to recover $L (\lambda)$ with a subset of the sample that is selected via the DBSCAM algorithm \citep{Ester1996}. This results in an estimator that is easy to compute in practice, but it yields a discrete set as an estimator, which may be a handicap for some applications. Sometimes one is not interested in the HDR itself, but in some of its characteristics, like its volume, its boundary, whether it is a convex set or not, etc; and the particular form of this estimator makes it impossible to compute those directly form it. The other three proposals \citep{Cuevas2006, SaavedraNieves2021, Cholaquidis2022} follow a plug-in approach. That is, they propose an estimator for the underlying density $f$, say $f_n$, and then the HDR is reconstructed with the set $f_n^{-1} \big( [\lambda, + \infty) \big)$. The main advantage of this estimation technique is that it is theoretically simple to analyze, since its asymptotic properties can be related to those of $f_n$ as an estimator of $f$. However, this approach has some issues in practice. 
For example, visualizing this estimator requires evaluating $f_n$ in a large grid of values, which might be computationally expensive. In addition, and similarly to the proposal by \cite{Jiang2017}, if our interest lies in some specific characteristic of the HDR, like its boundary or the number of connected components, it is not straightforward to compute those for $f_n^{-1} \big( [\lambda, + \infty) \big)$. Finally, both approaches ignore the inner geometry of the problem: if one knows that the HDRs of $f$ fulfill some geometric property (e.g.~some shape condition), there is no guarantee that the plug-in or the DBSCAM estimator will satisfy it too.

In the euclidean context geometrically based/conditioned HDR estimators have been proposed in the literature. For instance, an HDR estimator of a convex HDR should intuitively be convex too. To do that, one possibility is to maximize the empirical estimator of the excess mass within a given family of sets, see \cite{Muller1991} and \cite{Polonik1995}. Of course, this approach guarantees that the HDR estimator satisfies the required geometric restriction. Nevertheless, for a general family of sets, the optimization problem is not feasible.
This fact hampers the use of this estimation technique for most practical purposes.

To overcome this issue, an easy-to-compute new estimator for HDRs on Riemannian manifolds is introduced in this paper. The new approach combines an estimator of the underlying density with some \textit{a priori} geometric information, which is appropriately included in the estimation method to simplify the final form of the estimator. More specifically, the new proposal can be viewed as an extension of the euclidean HDRs estimation technique, called granulometric smoothing, introduced by \cite{Walther1997} to Riemannian manifolds. It is assumed that the population HDR satisfies some smoothness assumptions which can be established in terms of some simple geometrical operations from mathematical morphology, see \cite{Serra1982}. Thanks to this assumption, the HDRs can be recovered with a finite union of balls of radius $r_n$, where this radius is a parameter that has to be adequately chosen.

Of course, the selection of this radius $r_n$ is a crucial topic for the consistency of this estimation technique. Surprisingly, the seminal work by \cite{Walther1997} does not provide any data-driven method to choose $r_n$. In fact, the problem of selecting $r_n$ from the data remained unaddressed until very recently, and the only contribution in this sense is given by \cite{RodriguezCasal2022}. However, their selector of $r_n$ is not suitable manifold data since it relies on a deep connection between this radius and the density function $f$ which remains unclear in the manifold setting. This paper presents a novel approach to this problem and provides an original data-driven selector of~$r_n$. The new selector is based on empirical versions of the geometrical operations from mathematical morphology introduced by \cite{Serra1982}. To the best of authors' knowledge, this idea has not been explored before (even in the Euclidean setting) and allows the construction of a selector of $r_n$ that is simple to compute and works for data supported on a general Riemannian manifold $M$.

The distribution of this paper is as follows. First, Section~\ref{sec:notation} includes the basic notation and definitions that are required throughout the manuscript. Section~\ref{sec:HDRproposal} introduces the new proposal of HDR estimator. In Section~\ref{sec:Lnconsistency}, the consistency of the new estimator is proved and its convergence rates are derived (uniformly in $\lambda$). Section~\ref{sec:HDRgamma} deals with the problem of estimating an HDR when, instead of a level $\lambda$, a probability content $\gamma$ is given. A general procedure for estimating $\lambda_\gamma$ is proposed and analyzed. In Section~\ref{sec:rn}, the choice of the parameter $r_n$ is discussed. Section~\ref{sec:realdataHDR} illustrates the proposed HDR estimator with two real-data examples. Finally, Section~\ref{sec:gurus} provides a final discussion and future work.

All proofs are collected in the two-part Supplementary Material. Supplement~A \citep{suppA} contains all the proofs related to the consistency and convergence rates of the proposed HDR estimator, along with some auxiliary results that are needed along the proofs. Supplement~B \citep{suppB} investigates the assumptions required on the manifold $M$ and the HDRs of the population to ensure the consistency of the estimator. Particularly, in Supplement~B, these assumptions are shown to hold for any manifold with non-negative and bounded sectional curvature and any differentiable density function, supporting the generality of our method.

\section{Preliminaries and notation}
\label{sec:notation}

The HDR estimator proposed by \cite{Walther1997} is based on Minkowski operations. So, extending this estimation method to Riemannian manifolds requires the generalization of these operations to the manifold setting. 
This section provides an appropriate generalization of these concepts jointly with some basic notation that will be used throughout this manuscript.

\subsection{Riemannian manifolds and density functions}
Let $M$ be a connected Riemannian manifold of dimension $d \in \N$. Given a point $x \in M$, we denote by $T_x M$ the tangent space to $M$ at $x$. Given any $v \in T_x M$, its norm (or length) is denoted by $\norm{v}_M$ (see \citealp[page 12]{Lee2018}). We also use $\norm{v}$ when there is no need to specify the manifold.

We denote by $\mathrm{Vol}_M$ the Riemann-Lebesgue volume measure of $M$ (see \citealp{Amann2009}, Ch.~XII). $\mathrm{Vol}_M$ is a Radon measure over the \hbox{$\sigma$-algebra} of Borel sets of $M$, and it allows one to extend the concept of density function to the manifold setting.

Let $(\Omega, \mathcal{A}, \Prob)$ be a probability space and $X: \Omega \rightarrow M$ a Borel-measurable function. $X$ is called a random point of $M$. $X$ has density function $f$, where $f: M \rightarrow [0, + \infty)$ is a Borel-measurable function, if
\begin{equation}
	\label{eq:densitydef}
	\Prob (X \in A) = \int_A f d \mathrm{Vol}_M 
\end{equation}
for all Borel sets $A \subset M$. Actually, the Radon--Nikodym theorem (see \citealp{Billingsley1995}, Th.~32.2) ensures that $X$ has a density function if and only if the probability measure induced by $X$ is absolutely continuous with respect to $\mathrm{Vol}_M$, and the density function $f$ coincides (almost everywhere) with the Radon--Nikodym derivative.

\subsection{Distance and topology}
Let $\mathcal{d} (\cdot, \cdot)$ be the geodesic distance of $M$ (note that the geodesic distance function is well defined since $M$ is connected, see \citealp{Lee2012}, Ch.~13). The open and closed ball with center $x \in M$ and radius $r > 0$ will be denoted as $B_r (x)$ and $B_r [x]$, respectively. That is, given a point $x \in M$ and a radius $r > 0$, define
\begin{equation*}
	B_r (x) = \left\lbrace y \in M \colon \mathcal{d} (x, y) < r \right\rbrace, \qquad
	B_r [x] = \left\lbrace y \in M \colon \mathcal{d}
	(x, y) \leq r \right\rbrace.
\end{equation*}
Given a subset $A \subset M$, the interior, closure and boundary of $A$ are denoted as $\mathring{A}$, $\overline{A}$ and $\partial A$, respectively.

\subsection{Minkowski operations}
Given a subset $A \subset M$ and a radius $r > 0$, the Minkowski sum and difference are defined as
\begin{align}
	\label{eq:Minkplus}
	A \oplus r B &= \bigcup_{x \in A \phantom{^c}} B_r [x] = \big\lbrace x \in M \colon B_r [x] \not\subset A^c \big\rbrace,
	\\
	\label{eq:Minkminus}
	A \ominus r B &= \bigcap_{x \in A^c} \big( B_r [x] \big)^c = \big\lbrace x \in M \colon B_r [x] \subset A \big\rbrace;
\end{align}
where $A^c$ is the complement of the set $A$. The notation of the Minkowski operations is directly taken from \cite{Walther1997}. Note that there is a straightforward relation between these two operations via the De Morgan's laws:
\begin{equation}
	\label{eq:DeMorgan}
	(A \ominus r B)^c = A^c \oplus r B.
\end{equation}

\subsection{Distance between two sets and between a set and a point}
\label{sec:distset}
Let $A_1 , A_2 \subset M$ be two non empty subsets of $M$. The point distance can be used to define a distance between two sets, as
\begin{equation*}
	\mathcal{d} (A_1, A_2) = \inf \left\lbrace \mathcal{d} (a_1, a_2) \colon a_1 \in A_1, a_2 \in A_2 \right\rbrace.
\end{equation*}
One can define the distance between a point $x \in M$ and a non empty set $A_1 \subset M$ using the previous definition:
\begin{equation*}
	\mathcal{d} (x, A_1) = \mathcal{d} \big( \{ x \}, A_1 \big) = \inf \{ \mathcal{d} (x, a_1) \colon a_1 \in A_1 \}.
\end{equation*}

\subsection{Hausdorff distance}
The distance between sets introduced in Section~\ref{sec:distset} is indeed useful, but it does not provide a metric space structure to the set
\begin{equation*}
	F(M) =  \big\lbrace A \subset M \colon A \textit{ is compact and non empty} \big\rbrace.
\end{equation*}
Since our aim is to make inference with the subsets of $M$, a metric space structure on $F(M)$ is required. We must then introduce a new distance between two sets. Given two compact and non empty subsets of $M$, $A_1 , A_2 \in F(M)$, the Hausdorff distance between $A_1$ and $A_2$ is given by
\begin{equation*}
	\mathcal{d}_H (A_1, A_2) = \inf \big\lbrace r > 0 \colon A_1 \subset A_2 \oplus r B \textit{ and } A_2 \subset A_1 \oplus r B \big\rbrace.
\end{equation*}

\subsection{Eventually almost sure boundedness:}
Let $(\Omega, \mathcal{A}, \Prob)$ a probability space. Given a sequence of events $\{ A_n \}_{n \in \N} \subset \mathcal{A}$, define
\begin{equation}
	\liminf_{n \in \N} A_n = \bigcup_{n \in \N} \bigcap_{k \geq n} A_k,
	\qquad
	\limsup_{n \in \N} A_n = \bigcap_{n \in \N} \bigcup_{k \geq n} A_k.
\end{equation}
$A_n$ occur eventually almost sure or eventually almost surely (in short, e.a.s.) if
\begin{equation*}
	\Prob \left( \liminf_{n \in \N} A_n \right) = 1.
\end{equation*}
Finally, given two sequences of random variables $X_n$ and $Y_n$, the notation
\begin{equation*}
	X_n = O_{\textrm{a.s.}} (Y_n)
\end{equation*}
means that there exists a constant $C > 0$ such that the events
\begin{equation*}
	A_n = \big\lbrace \abs{X_n} \leq C \abs{Y_n} \big\rbrace
\end{equation*}
occur e.a.s. Notice that $X_n = O_{\textrm{a.s.}} (Y_n)$ implies $X_n = O_{\textrm{P}} (Y_n)$ by Fatou's Lemma.

\section{The new HDR estimator}
\label{sec:HDRproposal}

Let $X$ be a random point on a manifold $M$ with density function $f$. The main objective of this section is to introduce an estimator of the HDRs of $f$. That is, given a constant $\lambda$ such that $0 < \lambda < \sup f$, the aim is to recover the set
\begin{equation}
	\label{eq:Llambda}
	L (\lambda) = f^{-1} \big( [\lambda, + \infty) \big),
\end{equation}
from a finite i.i.d.~sample of $X$.

The new HDR estimator follows from a hybrid approach, combining a pilot estimator of the density function with some shape conditions on the true HDRs of the population. These shape conditions are based on the assumptions made by \cite{Walther1997} for HDR estimation on Euclidean spaces, that are strongly related to the notion of ``granulometry'' introduced by \cite{Matheron1975}. Given a subset $A\subset M$ and a radius $r > 0$, define the functional
\begin{equation}
	\label{eq:granulometry}
	\Psi_r (A)  = (A \ominus rB) \oplus rB = \bigcup_{B_r [x] \subset A} B_r [x].
\end{equation}

The shape conditions required by the new estimation method can be expressed in terms of this functional $\Psi_r$. Precisely, the HDRs are assumed to be fixed-points of $\Psi_r$, i.e., it is assumed that there exists a positive constant $r$ such that
\begin{equation}
	\label{eq:psirfix}
	L (\lambda) = \Psi_r \big[ L(\lambda) \big] =
	\bigcup_{B_r [x] \subset L(\lambda) } B_r [x].
\end{equation}
Hence, the HDRs can be written as the union of closed balls of a positive radius $r$. However, one readily sees that this radius $r$ is not unique. If Equation~\eqref{eq:psirfix} holds for $r_0 > 0$, then it also holds for any $r \in (0, r_0]$ (see Proposition~A.1.2d in Supplementary A, \citealp{suppA}). To avoid ambiguity, we consider the supremum of all solutions $r$ of Equation~\eqref{eq:psirfix}. That is, given $\lambda \in (0, \sup f)$, define
\begin{equation}
	\label{eq:defr}
	r_0(\lambda)
	=
	\sup \Big\lbrace r > 0 \colon L (\lambda) = \Psi_r \big[ L(\lambda) \big] \Big\rbrace
	,
\end{equation}
where the supremum of the empty set is taken as $0$. If $r_0 (\lambda) > 0$, then the set $L (\lambda)$ can be expressed as a union of balls for any radius $r$ between $0$ and $r_0 (\lambda)$. The proposed estimation method takes advantage of this fact recovering $L (\lambda)$ as a union of balls for some $r \in (0, r_0 (\lambda) )$. The idea is to ``inflate'' the high density points of the sample, distinguishing high and low density points using a density function estimator. In practice, the radius $r_0 (\lambda)$ is unknown, so it is also replaced by an estimator.

Let $\mathcal{X}_n = \{ X_1, \ldots, X_n \}$ be an i.i.d.~sample from $X$. The estimation technique consists of the following steps.
\begin{enumerate}[label = \textit{Step \arabic*:}, leftmargin = 0.5in]
	\item First, consider an estimator of the underlying density, say $f_n$, and a choice for the radius, namely $r_n (\lambda)$.
	\item Divide the sample $\mathcal{X}_n$ into low and high density points using $f_n$:
	\begin{equation*}
		\mathcal{X}^{+}_n ( \lambda ) = \{ X_i \colon f_n (X_i) \geq \lambda \}, \qquad \mathcal{X}^{-}_n ( \lambda )= \{ X_i \colon f_n (X_i) < \lambda \}.
	\end{equation*}
	\item The proposed estimator of $L (\lambda)$ is
	\begin{align}
		L_n \big( \lambda \big) 
		&= \Big[ \mathcal{X}^{+}_n ( \lambda ) \cap \big( \mathcal{X}^{-}_n ( \lambda ) \oplus r_n (\lambda) B \big)^c \Big] \oplus r_n (\lambda) B
		\nonumber
		\\
		&= \bigcup_{
			\begin{scriptsize}
				\begin{array}{c}
					X_i \in \mathcal{X}^{+}_n ( \lambda ) \\
					B_{r_n (\lambda)} [X_i] \cap \mathcal{X}^{-}_n ( \lambda ) = \emptyset
				\end{array}
			\end{scriptsize}
		} B_{r_n(\lambda)} [X_i] .
		\label{eq:setest}
	\end{align}
\end{enumerate}

The definition of the Minkowski difference in~\eqref{eq:Minkminus} yields $L ( \lambda ) \ominus r B \subset L(\lambda)$ for all $r > 0$. Taking this into account jointly with~\eqref{eq:DeMorgan}, we can rewrite the functional $\Psi_r$ in~\eqref{eq:granulometry} as
\begin{equation*}
	\Psi_r (A) = \big[ A \cap ( A^c \oplus r B )^c \big] \oplus r B.
\end{equation*}
Hence, if $r_n (\lambda) \in (0, r_0 (\lambda))$, then $L (\lambda) = \Psi_{r_n (\lambda)} ( L (\lambda))$ and so
\begin{equation}
	\label{eq:conditionLlamda}
	L(\lambda)
	=
	\Big[ L ( \lambda ) \cap \big( L ( \lambda )^c \oplus r_n(\lambda) B \big)^c \Big] \oplus r_n (\lambda) B
	.
\end{equation}
If $f_n$ is a consistent estimator of $f$, one expects that $\mathcal{X}_n^{+} (\lambda) \approx L(\lambda)$ and $\mathcal{X}_n^{-} (\lambda) \approx L(\lambda)^c$, respectively. Thus, the proposed estimator in~\eqref{eq:setest} is just a plug-in estimator based on formula~\eqref{eq:conditionLlamda}.
Notice that $r_n (\lambda)$ is not required to converge to $r_0 (\lambda)$ and one only needs that $r_n (\lambda) \in (0, r_0 (\lambda))$. In that sense, the role of $r_n (\lambda)$ is similar to that of a smoothing parameter: consistency is achieved if $r_n (\lambda)$ is small enough. For this reason, most references in the literature avoid referring to $r_n (\lambda)$ as an estimator of $r_0 (\lambda)$ and they rather say that $r_n (\lambda)$ is a ``choice'' or ``selection'' of the radius function.

Figure~\ref{fig:exampleLn} illustrates the behavior in practice of this HDR estimator. The three plots show $L_n (\lambda)$ with $\lambda = 0.45$ for a mixture of two different von Mises-Fisher distributions. The mixture consists of the distributions $\mathrm{M}_2 (\mu_1, 10)$ and $\mathrm{M}_2 (\mu_2, 10)$ with equal weights, where
\begin{equation}
	\label{eq:mus}
	\begin{array}{l}
		\mu_1 = \big( \cos(-\pi/6),\sin(-\pi/6),0 \big)',\\
		\mu_2 = \big( \cos(\pi/6) \cos(\pi/6), \cos(\pi/6) \sin(\pi/6), \sin(\pi/6) \big)'. 
	\end{array}
\end{equation}
The sample size is different in each plot, with $n = 400$ in Figure~\ref{fig:exampleLn400}, $n = 800$ in Figure~\ref{fig:exampleLn800}, and $n = 1600$ in Figure~\ref{fig:exampleLn1600}. In all cases, $r_n (\lambda) = 0.05$ and $f_n$ is the kernel density estimator with von Mises-Fisher kernel (\citealp{GarciaPortugues2013}) and concentration parameter chosen by cross-validation. For the smallest sample size, the estimator $L_n (\lambda)$ (red line) is very different from the true HDR $L(\lambda)$ (blue line), with an irregular boundary and empty spaces within it. As $n$ grows, $\mathcal{X}^{+}_n ( \lambda )$ and $\mathcal{X}^{-}_n ( \lambda )$ become denser discretizations of $L(\lambda)$ and $L(\lambda)^c$ and consequently $L_n (\lambda)$ approaches $L(\lambda)$.

\begin{figure}[!t]
	\centering
	\begin{subfigure}{1.7in}
		\centering
		\includegraphics[width = 1.7in]{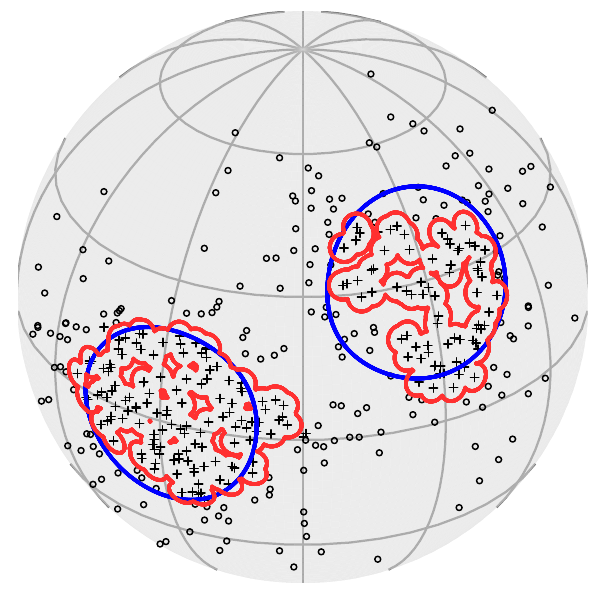}
		\caption{\label{fig:exampleLn400}}
	\end{subfigure}
	\hfill
	\begin{subfigure}{1.7in}
		\centering
		\includegraphics[width = 1.7in]{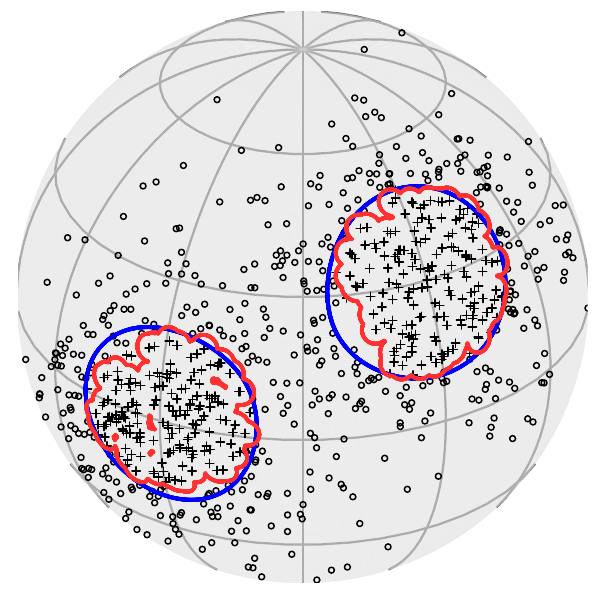}
		\caption{\label{fig:exampleLn800}}
	\end{subfigure}
	\hfill
	\begin{subfigure}{1.7in}
		\centering
		\includegraphics[width = 1.7in]{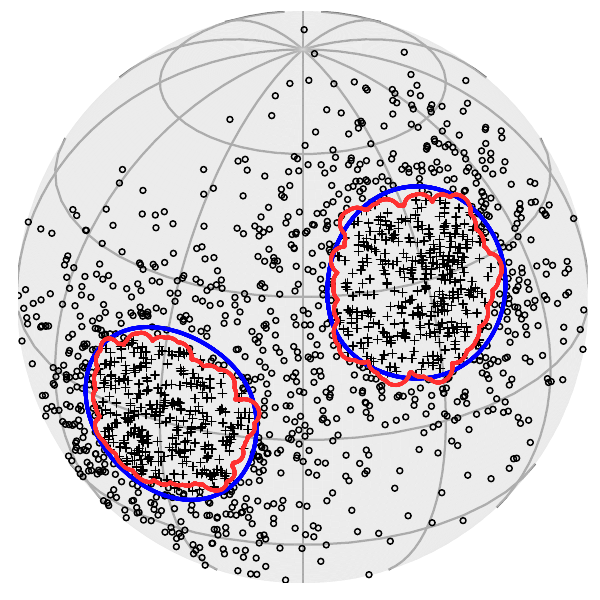}
		\caption{\label{fig:exampleLn1600}}
	\end{subfigure}
	\caption{\label{fig:exampleLn} HDR estimation of a mixture of two von Mises-Fisher distributions. All three graphs show the boundary of the estimator $L_n (\lambda)$ with $\lambda = 0.45$ for a mixture of the distributions $\mathrm{M}_2 (\mu_1, 10)$ and $\mathrm{M}_2 (\mu_2, 10)$ with equal weights (see equation~\eqref{eq:mus} for a definition of $\mu_1$ and $\mu_2$). Three sample sizes are considered, one for each graph: $n = 400$~(a), $n = 800$~(b) and $n = 1600$~(c). The boundary of $L_n (\lambda)$ is plotted in red in all graphs, whereas the boundary of the true HDR $L (\lambda)$ is depicted as a blue line. For reference, the points of the subsamples $\mathcal{X}^{+}_n ( \lambda )$ and $\mathcal{X}^{-}_n ( \lambda )$ are shown in the graphs as `$+$' and `$\circ$', respectively. In all three cases, \mbox{$r_n (\lambda) = 0.05$} and $f_n$ is the kernel density estimator with von Mises-Fisher kernel (\citealp{GarciaPortugues2013}) and concentration parameter chosen via cross-validation.}
\end{figure}

It should be remarked that the estimator $L_n (\lambda)$ does not rely solely on the shape assumption, but also on two auxiliary parameters: an estimator of the density function, $f_n$, and a choice of the radius, $r_n (\lambda)$. There are several proposals for density estimators in the manifold setting. For example, if $M$ is a compact $d$-dimensional submanifold of the Euclidean space $\R^D$ with $d < D$, then \cite{Wu2021} propose the following estimator of~$f$
\begin{equation}
	\label{eq:kdemanifold}
	f_n (x) = \frac{1}{n \mathcal{h}_n^d} \sum_{i = 1}^n \mathcal{K} \bigg( \frac{\norm{x - X_i}_{\R^D}}{\mathcal{h}_n}\bigg),
\end{equation}
where the bandwidth $\mathcal{h}_n$ converges to zero at a suitable rate and $\mathcal{K}: [0, + \infty) \rightarrow \R$ is a kernel function. If $f$ is Hölder continuous and $\mathcal{K}$ satisfies some regularity conditions, \cite{Wu2021} proved the global consistency of such an estimator and computed its convergence rate.

On the contrary, the choice of the radius $r_n (\lambda)$ has received almost no attention in statistical literature, even in the Euclidean setting. In the seminal work by \cite{Walther1997}, the radius is chosen as a real sequence slowly converging to zero. Thus, $r_n (\lambda) \in (0, r_0 (\lambda))$ will eventually hold regardless of the value of $r_0 (\lambda)$. However, \cite{Walther1997} does not provide any data-driven method for choosing the sequence $r_n (\lambda)$ and the statistician may choose $r_n (\lambda)$ with their own criterion. To the best of our knowledge, the only data-driven choice of the radius for Euclidean data is provided by \cite{RodriguezCasal2022}. Their proposal relies on the relation between $r_0 (\lambda)$ and the gradient of $f$, namely $\nabla f$, shown by \cite{Walther1997}. The main advantage of their approach is that it provides a sequence of radii such that $r_n (\lambda) \in (0, r_0 (\lambda))$ with probability one that does not converge to zero. Thanks to this fact, their HDR estimator achieves a slightly faster consistency rate than the original estimator by \cite{Walther1997}. Unfortunately, the relation between $\nabla f$ and $r_0 (\lambda)$ required for applying similar arguments in our context remains unclear in the manifold setting, and therefore this choice of the radius cannot be directly extended to manifold data.

The aforementioned works by \cite{Walther1997} and \cite{RodriguezCasal2022} avoid estimating the true radius $r_0 (\lambda)$. Instead, they just provide a sequence $r_n (\lambda)$ that is guaranteed to satisfy $r_n (\lambda) \in (0, r_0 (\lambda))$, since it is the only condition required for the consistency of $L_n (\lambda)$. This paper follows a different approach and faces directly the problem of estimating $r_0 (\lambda)$. A consistent estimator of this quantity for manifold data is introduced, that is, a sequence of radii $r_n (\lambda)$ such that $r_n (\lambda) \to r_0 (\lambda)$ almost surely, and this estimator is used as a choice for the radius in the HDR estimator $L_n (\lambda)$.

Furthermore, trying to keep the theoretical results as general as possible, this paper does not stick to a particular choice of $f_n$ and $r_n (\lambda)$ for studying the consistency of the HDR estimator $L_n (\lambda)$. In the results along this manuscript both quantities are assumed to fulfill some mild conditions, namely that $r_n (\lambda) \in (0, r_0 (\lambda))$ and $f_n$ is a globally consistent estimator of $f$. Then, the consistency of the estimator $L_n (\lambda)$ is shown, and its consistency rate is derived as a function of the global error of $f_n$. This approach broadens the applicability of our results, since the consistency of $L_n (\lambda)$ is guaranteed for a large variety of choices of $f_n$ and $r_n (\lambda)$.
Besides, if a new estimator of $f$ converging faster than the one defined on \eqref{eq:kdemanifold} is found, then our results will also guarantee that the estimator $L_n (\lambda)$ relying on such a new estimator will converge faster to $L(\lambda)$.

\section{Consistency of the HDR estimator}
\label{sec:Lnconsistency}

In this section, the consistency of the estimator $L_n (\lambda)$ defined at \eqref{eq:setest} is derived. First, Sections~\ref{sec:assM} and~\ref{sec:assA} detail the required conditions for the consistency of $L_n (\lambda)$. Specifically, Section~\ref{sec:assM} describes the manifold setting that is considered throughout this paper, while Section~\ref{sec:assA} details the shape conditions on the HDRs of $f$ that are required for the consistency of the HDR estimator. Finally, in Section~\ref{sec:convergence} the consistency of $L_n (\lambda)$ is shown and its asymptotic convergence rate is derived uniformly on~$\lambda$.

\subsection{Assumptions on the manifold $M$}
\label{sec:assM}

As we have seen in Section~\ref{sec:HDRproposal}, $L_n (\lambda)$ uses the Minkowski operations defined in~\eqref{eq:Minkplus} and~\eqref{eq:Minkminus} to recover the HDRs of the underlying density $f$. To guarantee the good behavior of the Minkowski operations, and therefore the consistency of $L_n (\lambda)$, one needs to impose the following assumptions on the manifold $M$.

\begin{assM}
	
	$M$ is a $d$-dimensional Riemannian manifold verifying:
	
	\begin{enumerate}[label = {(M\arabic*)}]
		\item \label{ass:M1} $M$ is complete and connected.
		\item \label{ass:M2} There exist two constants, $a, \rho > 0$, such that
		\begin{equation*}
			\mathrm{Vol}_{M}  \big( B_{r_1} [x_1] \cap B_{r_2} [x_2] \big) \geq a \min \{ r_1, r_2, \rho \}^{(d-1)/2} \min \{ \varepsilon, \rho \}^{(d+1)/2}.
		\end{equation*}
		for every pair of points $x_1, x_2 \in M$ and radii $r_1, r_2 > 0$ such that $\mathcal{d} (x_1, x_2) < r_1 + r_2$ and every $\varepsilon \leq r_1 + r_2 - \mathcal{d} (x_1, x_2)$.
	\end{enumerate}
	
\end{assM}

The role of Assumption~\ref{ass:M1} is twofold: the connectedness of $M$ ensures that the geodesic distance is well defined, while the completeness guarantees a good behavior of the Minkowski operations and the mapping $\Psi_r$ (see \citealp{suppA}, Section~A.1). Assumption~\ref{ass:M2} is required for showing the consistency of $L_n (\lambda)$. By \eqref{eq:setest}, $L_n (\lambda)$ is a union of balls of the same radius. So, in order to compute the convergence rate of $L_n (\lambda)$, a lower bound of the probability of the intersection of two balls must be determined. Since probability and volume are related through the density function $f$, Assumption~\ref{ass:M2} guarantees the existence of such a bound. Assumption~\ref{ass:M2} also provides a uniform bound for the packing numbers of any subset of $M$, which will be useful throughout the proofs. 


\begin{prop}
	\label{prop:M2impliesM3}
	Suppose that $M$ is a $d$-dimensional Riemannian manifold satisfying Assumption~\ref{ass:M2}. Given a bounded subset $A \subset M$, define:
	\begin{equation*}
		D(\varepsilon, A) = \max \left\lbrace \mathrm{card} (T) \colon T \subset A, \mathcal{d} (x, y) > \varepsilon \textit{ for all } x, y \in T, x \neq y \right\rbrace.
	\end{equation*}
	Then, there exists a constant $K$ such that $D(\varepsilon, A) \leq K \varepsilon^{-d}$ for all $\varepsilon \leq \rho$, where $\rho > 0$ is the same constant as in Assumption~\ref{ass:M2}.
\end{prop}

At first sight, Assumption~\ref{ass:M2} may seem an \textit{ad-hoc} condition for this problem.
Nevertheless, any manifold satisfying some mild conditions verify Assumption~\ref{ass:M2}, as the next result shows.

\begin{theorem}
	\label{th:manifold}
	Let $M$ be a complete and connected $d$-dimensional Riemannian manifold satisfying:
	\begin{enumerate}
		\item All sectional curvatures of $M$, denoted by $\kappa_M$, are non negative and bounded. That is, there exists a positive constant $\kappa > 0$ such that
		\[ 0 \leq \kappa_M < \kappa.\]
		\item $M$ has positive injectivity radius: $i_M > 0$.
	\end{enumerate}
	Then $M$ satisfies Assumption~\ref{ass:M2}.
\end{theorem}

Theorem~\ref{th:manifold} ensures that any manifold with bounded and non negative sectional curvatures fulfill Assumptions~M. This includes the Euclidean space $\R^d$, but also manifolds like the hypersphere $\mathbb{S}^d$, or the $d$-dimensional torus~$\mathbb{T}^d$. The proofs of Proposition~\ref{prop:M2impliesM3} and Theorem~\ref{th:manifold} can be found in Section~B.1 of Supplement B \citep{suppB}.

\subsection{Assumptions on the HDRs}
\label{sec:assA}

As discussed in Section~\ref{sec:HDRproposal}, the condition $r_0 (\lambda) > 0$ is the main requirement for the consistency of the new proposal. In practice, stronger assumptions on the HDRs of $f$ are required.

\begin{assL}
	Let $M$ be a connected Riemannian manifold, \hbox{$f: M \rightarrow \R$} a continuous density function, and $l$ and $u$ two constants such that $0 < l \leq u < \sup f$. Let $r_0 (\lambda)$ be the function defined in~\eqref{eq:defr}. 
	There exist two positive constants $k$ and $\delta$, with $k, \delta > 0$, satisfying
	\begin{enumerate}[label = {(L\arabic*)}]
		\item \label{ass:A1} $\delta < r_0 (\lambda) < + \infty$ for all $\lambda \in [l - \delta , u + \delta]$.
		\item \label{ass:A2} $L (\lambda)^c = \Psi_{\delta} \big[ L(\lambda)^c \big]$ for all $\lambda \in [l - \delta, u + \delta]$.
		\item \label{ass:A3} $L (\lambda)$ is a compact set for all $\lambda \in [l - \delta , u + \delta]$.
		\item \label{ass:A4} $\mathcal{d}_H \big( L(\lambda), L(\lambda + \eta) \big) \leq k \vert \eta \vert$ for all $\lambda \in [l, u]$ when $\eta \in [-\delta, \delta]$.
	\end{enumerate}
\end{assL}

\begin{remark}
	\label{rem:delta}
	The quantity $\delta$ in Assumptions~\ref{ass:A1}--\ref{ass:A4} should be understood as a ``small but positive constant''. Its specific value is not relevant, since it will not affect the convergence rate of the estimator $L_n (\lambda)$.
	
	Furthermore, assuming that $\delta$ is equal for all of Assumptions L is not restrictive. For example, if $f$ verifies \ref{ass:A1} with $\delta_1$, \ref{ass:A2} with $\delta_2$, \ref{ass:A3} with $\delta_3$, and \ref{ass:A4} with $\delta_4$; then $f$ also verifies all four assumptions with $\delta = \min \{ \delta_1, \delta_2 , \delta_3, \delta_4 \}$. Then, they will be assumed to be equal to simplify notation.
\end{remark}

The necessity of Assumption~\ref{ass:A1} for the consistency of $L_n (\lambda)$ is clear. This estimator relies entirely on the assumption that $L (\lambda)$ can be expressed as a union of balls, and this holds if and only if $r_0 (\lambda)$ is bounded away from zero. Moreover, note that $r_0 (\lambda) = + \infty$ if and only if that $L (\lambda) = M$. So, Assumption~\ref{ass:A1} also imposes $r_0 (\lambda) < + \infty$ to exclude the degenerate case where the HDR is the whole manifold. Assumption~\ref{ass:A3} serves merely as a technical condition to guarantee that the Hausdorff distance is well defined on the HDRs of~$f$. The consistency of $L_n (\lambda)$ will be derived in terms of the Hausdorff distance. Hence, $\mathcal{d}_H$ is required to be well defined on the highest density regions of $f$, and $\mathcal{d}_H$ is only defined for compact sets.

As mentioned in Section~\ref{sec:HDRproposal}, $L_n (\lambda)$ uses $\mathcal{X}^{+}_n ( \lambda )$ and $\mathcal{X}^{-}_n ( \lambda )$ as discretizations of the sets $L(\lambda)$ and $L(\lambda)^c$, respectively. The intuition is that, as $n$ gets larger, $\mathcal{X}^{+}_n ( \lambda )$ and $\mathcal{X}^{-}_n ( \lambda )$ should become finer and finer discretizations of $L(\lambda)$ and $L(\lambda)^c$, and, eventually, $L_n (\lambda)$ will recover $L (\lambda)$. However, $\mathcal{X}^{+}_n ( \lambda )$ and $\mathcal{X}^{-}_n ( \lambda )$ depend on the estimator $f_n$ and not on the underlying density $f$. So, without any further assumptions, it is not guaranteed that $\mathcal{X}^{+}_n ( \lambda ) \subset L(\lambda)$ and $\mathcal{X}^{-}_n ( \lambda ) \subset L(\lambda)^c$. Assumption~\ref{ass:A4} solves this problem, ensuring that, if $f_n$ is a consistent estimator of $f$, $\mathcal{X}^{+}_n ( \lambda ) \subset L(\lambda)$ and $\mathcal{X}^{-}_n ( \lambda ) \subset L(\lambda)^c$ will eventually hold. Thus, $\mathcal{X}^{+}_n ( \lambda )$ and $\mathcal{X}^{-}_n ( \lambda )$ will be valid discretizations of $L(\lambda)$ and $L(\lambda)^c$ for $n$ large.

In addition, for the consistency of $L_n (\lambda)$ both discretizations are required to become ``fine enough'' as the sample size grows. Assumption~\ref{ass:A2} guarantees that $\mathcal{X}^{-}_n ( \lambda )$ will become fine enough. Broadly speaking, the continuity of $f$ and Assumption \ref{ass:A2} ensure that there exists a set $A_ {\varepsilon} = L (\lambda - \varepsilon) \setminus L (\lambda)$ with $\Prob (A_{\varepsilon}) > 0$, which guarantees that $\mathcal{X}^{-}_n ( \lambda )$ will ``densely'' surround $L ( \lambda )$. For $\mathcal{X}^{+}_n ( \lambda )$ to be dense in $L ( \lambda )$ no further assumptions are required, since Assumption~\ref{ass:A1} will ensure this by similar arguments.


\begin{figure}[t]
	\centering
	\begin{subfigure}{2.5in}
		\centering
		\def\scale{1}
\begingroup%
  \makeatletter%
  \providecommand\color[2][]{%
    \errmessage{(Inkscape) Color is used for the text in Inkscape, but the package 'color.sty' is not loaded}%
    \renewcommand\color[2][]{}%
  }%
  \providecommand\transparent[1]{%
    \errmessage{(Inkscape) Transparency is used (non-zero) for the text in Inkscape, but the package 'transparent.sty' is not loaded}%
    \renewcommand\transparent[1]{}%
  }%
  \providecommand\rotatebox[2]{#2}%
  \newcommand*\fsize{\dimexpr\f@size pt\relax}%
  \newcommand*\lineheight[1]{\fontsize{\fsize}{#1\fsize}\selectfont}%
  \ifx\svgwidth\undefined%
    \setlength{\unitlength}{164.80857909bp}%
    \ifx\svgscale\undefined%
      \relax%
    \else%
      \setlength{\unitlength}{\unitlength * \real{\svgscale}}%
    \fi%
  \else%
    \setlength{\unitlength}{\svgwidth}%
  \fi%
  \global\let\svgwidth\undefined%
  \global\let\svgscale\undefined%
  \makeatother%
  \begin{picture}(1,0.6369688)%
    \lineheight{1}%
    \setlength\tabcolsep{0pt}%
    \put(0,0){\includegraphics[width=\unitlength,page=1]{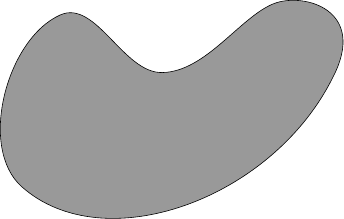}}%
    \put(0.41147607,0.24172244){\color[rgb]{0,0,0}\makebox(0,0)[lt]{\lineheight{1.25}\smash{\begin{tabular}[t]{l}$L(\lambda)$\end{tabular}}}}%
  \end{picture}%
\endgroup%

	\end{subfigure}
	\begin{subfigure}{2.5in}
		\centering
		\def\scale{1}
\begingroup%
  \makeatletter%
  \providecommand\color[2][]{%
    \errmessage{(Inkscape) Color is used for the text in Inkscape, but the package 'color.sty' is not loaded}%
    \renewcommand\color[2][]{}%
  }%
  \providecommand\transparent[1]{%
    \errmessage{(Inkscape) Transparency is used (non-zero) for the text in Inkscape, but the package 'transparent.sty' is not loaded}%
    \renewcommand\transparent[1]{}%
  }%
  \providecommand\rotatebox[2]{#2}%
  \newcommand*\fsize{\dimexpr\f@size pt\relax}%
  \newcommand*\lineheight[1]{\fontsize{\fsize}{#1\fsize}\selectfont}%
  \ifx\svgwidth\undefined%
    \setlength{\unitlength}{177.97678549bp}%
    \ifx\svgscale\undefined%
      \relax%
    \else%
      \setlength{\unitlength}{\unitlength * \real{\svgscale}}%
    \fi%
  \else%
    \setlength{\unitlength}{\svgwidth}%
  \fi%
  \global\let\svgwidth\undefined%
  \global\let\svgscale\undefined%
  \makeatother%
  \begin{picture}(1,0.57937966)%
    \lineheight{1}%
    \setlength\tabcolsep{0pt}%
    \put(0,0){\includegraphics[width=\unitlength,page=1]{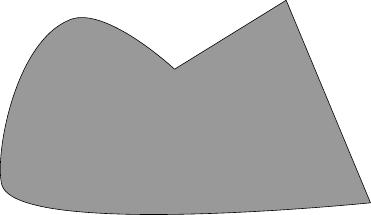}}%
    \put(0.41802588,0.21860739){\color[rgb]{0,0,0}\makebox(0,0)[lt]{\lineheight{1.25}\smash{\begin{tabular}[t]{l}$L(\lambda)$\end{tabular}}}}%
  \end{picture}%
\endgroup%

	\end{subfigure}
	\caption{\label{fig:AssA} Example of an HDR $L(\lambda)$ satisfying (left) and not satisfying (right) Assumptions \ref{ass:A1} and \ref{ass:A2}.}
\end{figure}

Actually, Assumptions~L can be viewed as smoothness conditions on the HDRs $L (\lambda)$. This is easy to see from Assumption~\ref{ass:A4}, since it is a Lipschitz condition on the mapping~$L$: \ref{ass:A4} states that $L$ is a (locally) Lipschitz continuous function with respect to the Hausdorff distance. Assumptions~\ref{ass:A1} and \ref{ass:A2} are smoothness conditions in a more geometric sense. If both $L (\lambda)$ and $L(\lambda)^c$ can be recovered as a union of balls of the same fixed radius, then the boundary of $L (\lambda)$ is smooth, without sharp edges (see Figure \ref{fig:AssA}). This fact implies that the boundary of the sets $L(\lambda)$ are regular submanifolds of $M$.

Since Assumptions L are smoothness conditions on the HDRs of $f$, it is no surprise that they translate into smoothness conditions on the density function $f$. For example, if $f$ is continuous, then the HDRs are guaranteed to be compact and Assumption~\ref{ass:A3} holds.

\begin{prop}
	\label{prop:assA3}
	Let $M$ be a Riemannian manifold under Assumptions M, let $\lambda > 0$ and $f: M \rightarrow \R$ a continuous density function. Suppose that $L(\lambda) = \Psi_{\delta} \big[ L (\lambda) \big]$ for some $\delta > 0$. Then, $L(\lambda)$ is a compact set.
\end{prop}

If $f$ is a differentiable function, the gradient $\nabla f$ is the direction in which $f$ varies the fastest, and $\norm{\nabla f}_M$ measures the variation rate in that direction. Hence, controlling $\norm{\nabla f}_M$ allows to control the distance between the HDRs, as the next result shows.

\begin{prop}
	\label{prop:assA4}
	Let $M$ be a $d$-dimensional Riemannian manifold under Assumptions~M, and $f: M \rightarrow \R$ a density function. Let $0 < \delta < l \leq u < \sup f$, and let $U$ be an open set such that $f^{-1} \big( [l - \delta, u + \delta] \big)\subset U$.
	
	Suppose that $f$ is continuously differentiable on $U$ and that there exists a positive constant $c > 0$ such that $\norm{ \nabla f }_M \geq c$ in $U$. Then:
	\begin{equation*}
		\mathcal{d}_H \big( L (\lambda), L(\lambda + \eta) \big) \leq \frac{\vert \eta \vert}{c}
	\end{equation*}
	for all $\eta \in [- \delta, \delta]$ and $\lambda \in [l , u]$.
\end{prop}

The proofs of Propositions~\ref{prop:assA3} and~\ref{prop:assA4} are collected in Section~B.2 of Supplement~B \citep{suppB}. These results ensure that any sufficiently smooth density function $f$ verifies Assumptions~\ref{ass:A3} and~\ref{ass:A4}. If $M = \R^d$, \cite{Walther1997} showed that this is also true for Assumptions~\ref{ass:A1} and~\ref{ass:A2} (see \citealp{Walther1997}, Th.~2). In the manifold setting, the relation between the differentiability of $f$ and Assumptions~\ref{ass:A1} and~\ref{ass:A2} is unclear, yet it is conjectured that all densities $f$ under the conditions of Proposition~\ref{prop:assA4} will satisfy them too.

\subsection{Convergence rate of the new estimator}
\label{sec:convergence}

Under Assumptions~M and Assumptions~L, the uniform behavior of the $L_n (\lambda)$ estimator can be studied. Theorem~\ref{th:uniform}, the main result in this work, ensures that $L_n (\lambda)$ is a consistent estimator of $L (\lambda)$, and its convergence rate (in Hausdorff distance) is derived in terms of the global error of the density estimator $f_n$.  The proof of this result is detailed in Section~A.2 of Supplement~A \citep{suppA}.

\begin{theorem}
	\label{th:uniform}
	Let $M$ be a $d$-dimensional Riemannian manifold under Assumptions~M. Let $X$ be a random point from $M$ with density function $f$, and suppose that $f$ satisfies Assumptions~L. Let $\mathcal{X}_n = \{ X_1, \ldots, X_n \}$ be an i.i.d.~sample from $X$, and let $f_n$ be an estimator of $f$ such that:
	\begin{equation*}
		D_n = \sup_{x \in M} \abs{f(x) - f_n(x)} \rightarrow 0
	\end{equation*}
	almost surely. Furthermore, for any $\lambda \in [l, u]$, let $r_n (\lambda)$ be a sequence of positive random variables fulfilling
	\begin{equation}
		\label{eq:Rn}
		\delta < r_n (\lambda) < r_0 (\lambda) - \delta, \quad \forall \lambda \in [l, u]
	\end{equation}
	e.a.s., where $r_0 (\lambda)$ is defined as~\eqref{eq:defr} and $\delta > 0$ is the same constant as in Assumptions~L.
	
	Then, if $L_n (\lambda)$ is the (random) set defined in \eqref{eq:setest},
	\begin{equation*}
		\sup_{\lambda \in [l,u]} \mathcal{d}_H \big( L(\lambda), L_n(\lambda) \big) = O_{\textrm{\rm a.s.}} \left( \max \left\lbrace D_n, \left( \frac{\log n}{n} \right)^{\frac{2}{d+1}} \right\rbrace \right).
	\end{equation*}
\end{theorem}

\begin{remark}
	\label{rem:Dn}
	If the manifold $M$ is compact, the density function $f$ is a Lipschitz continuous function, and $f_n$ is the kernel density estimator defined in \eqref{eq:kdemanifold} with bandwidth $\mathcal{h}_n$ of order $(\log (n) / n)^{1/(d + 2)}$ with $\epsilon > 0$, Theorem~2.4 of \cite{Wu2021} guarantees that, for some kernels $\mathcal{K}$,
	\begin{equation*}
		D_n
		=
		O_{\textrm{a.s.}}
		\bigg(
		\left( \frac{\log n}{n} \right)^{\frac{1}{d + 2}}
		\bigg)
	\end{equation*}
	Hence, in this specific case, $L_n (\lambda)$ will converge at rate $(\log (n)/n)^{1/(d+2)}$. This illustrates that the new estimator suffers from the so-called ``curse of dimensionality''. These rates are known to be optimal in the minimax sense for estimating $L (\lambda)$ in the Euclidean setting (see \citealp{Tsybakov1997}). 
\end{remark}

Theorem~\ref{th:uniform}, jointly with Assumptions~M and~L, imposes some restrictions on the primary estimators~$f_n$ and~$r_n (\lambda)$. These requirements were already discussed in Section~\ref{sec:HDRproposal}: $f_n$ must be globally consistent and $r_n (\lambda)$ has to be smaller than $r_0 (\lambda)$, otherwise $L_n (\lambda)$ will not recover $L (\lambda)$. However, the restriction imposed on~$r_n (\lambda)$ is slightly stronger than the one introduced in Section~\ref{sec:HDRproposal}. In principle, one only need to ensure that $L (\lambda)$ can be recovered as a union of balls of radius $r_n (\lambda)$ for consistency, and this holds if and only if \mbox{$r_n (\lambda) \leq r_0 (\lambda)$}. Nevertheless, in Equation \eqref{eq:Rn} $r_n (\lambda)$ is forced to be strictly smaller than~$r_0 (\lambda)$. This condition is required to ensure that $\Prob \big[ L (\lambda) \ominus r_n (\lambda) B \big] \neq 0$. Note that
$$
L (\lambda) \ominus r_n(\lambda)B = L(\lambda) \cap \big( L (\lambda)^c  \oplus r_n(\lambda)B \big)^c.
$$	
Consequently, if $\Prob \big[ L (\lambda) \ominus r_n (\lambda) B \big] = 0$, then $\mathcal{X}_n^{+} (\lambda) \cap (\mathcal{X}_n^{-} (\lambda) \oplus r_n (\lambda)B)^c = \emptyset$ will happen with positive probability and $L_n (\lambda)$ will not recover $L(\lambda)$. Moreover, Equation \eqref{eq:Rn} forces the sequence of random variables $r_n (\lambda)$ to not converge to zero. Section~\ref{sec:rn} will show that this is not a limitation in practice, since it is possible to provide a choice of $r_n (\lambda)$ that is bounded away from zero.

Theorem~\ref{th:uniform} shows that $L_n (\lambda)$ achieves the same consistency rate as the proposal by \cite{RodriguezCasal2022}. The comparison with the original estimator by \cite{Walther1997} is not direct, since the approaches to choose the radius are different. If $r_0 (\lambda)$ is known, \cite{Walther1997} proposes to select $r_n (\lambda) = r_0 (\lambda)$ (or slightly smaller), and in this case his proposal has the same rate of convergence as $L_n (\lambda)$. If $r_0 (\lambda)$ is unknown, \cite{Walther1997} chooses the radius as a sequence slowly converging to zero, $r_n \to 0$, but this results in a penalization factor of the form $r_n^{-(d-1)/(d+1)}$ in the term $( \log (n) /n)^{2/(d+1)}$ (see Theorem 3 in \citealp{Walther1997}).


The derived upper bound for the Hausdorff distance between $L(\lambda)$ and $L_n (\lambda)$ is larger than $D_n$, which means that we cannot guarantee that $L_n (\lambda)$ will converge faster than the plug-in HDR estimator. The new proposal compensates this with a much simpler form (it is a union of balls) that makes it much easier to handle in practice. Moreover, the term $( \log (n) /n)^{2/(d+1)}$ in the consistency rate can be improved by resampling (see Remark~3 in \citealp{Walther1997}). 
Thus, technically, one can always take a resample that is large enough so that the resulting HDR estimator converges at the same rate as the plug-in estimator.
In addition, $D_n$ will be the dominant term of the consistency rate of $L_n (\lambda)$ in some specific situations. For example, if the density function is Lipschitz continuous, under the Assumptions of Remark~\ref{rem:Dn}, $D_n$ is of order $(\log (n)/n )^{1/(d+2)}$ and therefore $L_n (\lambda)$ will converge at the same rate than $D_n$. Obviously, if $f$ is smoother, faster rates for $D_n$ are known in the Euclidean setting (see Remark~3.1 in \citealp{RodriguezCasal2022}). To the best of the authors' knowledge, the problem of finding the minimax rates for $D_n$ in the manifold setting has not been addressed in the literature yet.

\section{HDRs by probability content}
\label{sec:HDRgamma}

In Sections~\ref{sec:HDRproposal} and~\ref{sec:Lnconsistency}, we have studied HDRs as sets where the value of a given density function $f$ exceeds some threshold $\lambda$, as expressed in equation~\eqref{eq:Llambda}. This is convenient from a theoretical point of view, but the specific value of $\lambda$ is often irrelevant in practice. Instead, the interest is usually focused on recovering the HDR verifying a (previously fixed) probability content. That is, given a probability $\gamma \in (0, 1)$, the goal is to recover the set
\begin{equation}
	L(\lambda_{\gamma}) = f^{-1} \big( [\lambda_{\gamma}, + \infty ) \big) = \{ x \in M \colon f(x) \geq \lambda_{\gamma}\}
\end{equation}
where
\begin{equation}
	\label{eq:lambdagamma}
	\lambda_{\gamma} = \sup \Big\lbrace \lambda \in \R \colon \Prob \big[ L ( \lambda ) \big] \geq 1 - \gamma \Big\rbrace.
\end{equation}
Thus, $L(\lambda_{\gamma})$ is the smallest HDR such that $\Prob \big[ L ( \lambda_{\gamma } ) \big] \geq 1 - \gamma$. This perspective links HDR estimation with the problem of estimating the minimum volume set satisfying a probability content (see \citealp{Polonik1997}).

Theorem~\ref{th:uniform} states that the convergence of the estimator $L_n (\lambda)$ is uniform with respect to the level $\lambda$. This fact allows to address the problem of recovering the set $L(\lambda_{\gamma})$ using this estimator via a plug-in process. First, one computes an estimator of $\lambda_{\gamma}$, namely $\tilde{\lambda}_{\gamma, n}$, and then, $L(\lambda_{\gamma})$ is recovered by $L_n(\tilde{\lambda}_{\gamma, n})$. In Section~\ref{sec:plug-in}, the consistency of such a plug-in technique is studied, while Section~\ref{sec:lambdan} introduces an estimator of the level~$\lambda_{\gamma}$.

\subsection{Consistency of the plug-in estimation technique}
\label{sec:plug-in}

If $\tilde{\lambda}_{\gamma, n}$ is a consistent estimator of $\lambda_{\gamma}$, then $L_n(\tilde{\lambda}_{\gamma, n})$ is also a consistent estimator of $L(\lambda_{\gamma})$, as the following theorem shows.

\begin{theorem}
	\label{th:probs}
	Let $M$ be a $d$-dimensional Riemannian manifold under Assumptions~M. Let $X$ be a random point from $M$ with density function $f$, and suppose that $f$ satisfies Assumptions~L. Let $\mathcal{X}_n = \{ X_1, \ldots, X_n \}$ be an i.i.d.~sample from $X$, and let $f_n$ be an estimator of $f$ such that:
	\begin{equation*}
		D_n = \sup_{x \in M} \abs{f(x) - f_n(x)} \rightarrow 0
	\end{equation*}
	almost surely. Furthermore, for any $\lambda \in [l, u]$, let $r_n (\lambda)$ be a sequence of positive random variables fulfilling
	\begin{equation*}
		\delta < r_n (\lambda) < r_0 (\lambda) - \delta, \quad \forall \lambda \in [l, u]
	\end{equation*}
	e.a.s., where $r_0 (\lambda)$ is defined as~\eqref{eq:defr} and $\delta > 0$ is the same constant as in Assumptions~L.
	
	Let $0 < \underline{\gamma} \leq \overline{\gamma} < 1$ such that $[\lambda_{\underline{\gamma}}, \lambda_{\overline{\gamma}}] \subset (l, u)$. Let $\tilde{\lambda}_{\gamma, n}$ be an estimator of $\lambda_{\gamma}$ verifying:
	\begin{equation*}
		T_n = \sup_{\gamma \in [\underline{\gamma}, \overline{\gamma}]} \abs{\tilde{\lambda}_{\gamma, n} - \lambda_{\gamma}} \rightarrow 0
	\end{equation*}
	almost surely.
	
	Hence:
	\begin{equation*}
		\sup_{\gamma \in [\underline{\gamma}, \overline{\gamma}]} \mathcal{d}_H \big( L(\lambda_{\gamma}), L_n(\tilde{\lambda}_{\gamma, n}) \big) = O_{\textrm{\rm a.s.}} \left( \max \left\lbrace D_n, T_n, \left( \frac{\log n}{n} \right)^{\frac{2}{d+1}} \right\rbrace \right).
	\end{equation*}
\end{theorem}

Theorem~\ref{th:probs} (see the proof in Section~A.3 of Supplement A, \citealp{suppA}) shows that $L_n (\tilde{\lambda}_{\gamma , n})$ is a consistent estimator of $L (\lambda_{\gamma})$ and its convergence rate depends on the global error of the primary estimator of the level $\tilde{\lambda}_{\gamma , n}$, in such a way that, if $\tilde{\lambda}_{\gamma , n}$ converges slowly to $\lambda_{\gamma}$, the estimator $L_n (\tilde{\lambda}_{\gamma , n})$ will then inherit that convergence rate and it will also converge slowly to $L (\lambda_{\gamma})$. Therefore, the choice of $\tilde{\lambda}_{\gamma , n}$ is crucial. Not only it must be a consistent estimator of the true level $\lambda_{\gamma}$, but it should converge as fast as possible to avoid any penalization on the convergence rate of estimator $L_n (\tilde{\lambda}_{\gamma , n})$. The next section explores this problem and provides an estimator of the level that does not penalize the convergence rate of the final HDR estimator.

\subsection{Estimation of the level $\lambda_{\gamma}$}
\label{sec:lambdan}

As Theorem~\ref{th:probs} states, estimating $L (\lambda_{\gamma})$ reduces to choosing a consistent estimator of the level $\lambda_{\gamma}$ defined in~\eqref{eq:lambdagamma}. Obviously, there is no unique choice, and several estimators of $\lambda_{\gamma}$ have been provided in the Euclidean setting (see \citealp{Hyndman1996}, \citealp{Walther1997} and \citealp{Cadre2006}, among others). In this paper, we extend the proposal by \cite{Hyndman1996} to manifold data. That is, we propose to estimate $\lambda_{\gamma}$ with
\begin{equation}
	\label{eq:lambdan}
	\hat{\lambda}_{\gamma, n} = \sup \left\lbrace \lambda \in \R \colon \Prob_n \Big[ f_n^{-1} \big( [\lambda, + \infty) \big) \Big] \geq 1 - \gamma \right\rbrace,
\end{equation}
where $\Prob_{n}$ is the empirical probability distribution. The main advantage of this proposal is its simplicity. Notice that
\begin{equation*}
	\Prob_n \Big[ f_n^{-1} \big( [\lambda, + \infty) \big) \Big] = \frac{1}{n} \sum_{i = 1}^{n} \mathbbm{1} \big( f_n (X_i) \geq \lambda \big),
\end{equation*}
and therefore
\begin{equation*}
	\hat{\lambda}_{\gamma, n} = \sup \left\lbrace \lambda \in \R \colon \frac{1}{n} \sum_{i = 1}^{n} \mathbbm{1} \big( f_n (X_i) \geq \lambda \big) \geq 1 - \gamma \right\rbrace.
\end{equation*}
So, the estimator $\hat{\lambda}_{\gamma, n}$ is the $\gamma$-quantile of the empirical distribution of the transformed sample $f_n (X_1), \ldots, f_n (X_n)$ and, consequently, $\hat{\lambda}_{\gamma, n}$ can be computed quite efficiently in practice.

In order to guarantee the uniform convergence of the estimator $\hat{\lambda}_{\gamma, n}$ defined in~\eqref{eq:lambdan}, a smoothness assumption on $\lambda_{\gamma}$ is required.

\begin{assP}
	Define
	$\lambda_{\gamma} = \sup \Big\lbrace \lambda \in \R \colon \Prob \big[ L ( \lambda ) \big] \geq 1 - \gamma \Big\rbrace.$
	There exist $0 < \underline{\gamma} \leq \overline{\gamma}$ such that
	\begin{equation*}
		\sup_{\gamma \in [\underline{\gamma}, \overline{\gamma}]} \abs{\lambda_{\gamma} - \lambda_{\gamma + \eta}} \leq k \abs{\eta}
	\end{equation*}
	for all $\eta \in [- \delta, \delta]$, where $\delta$ and $k$ are the same constants as Assumptions~L.
\end{assP}

Assumption~P states that $\lambda_{\gamma}$ is a (locally) Lipschitz continuous function of~$\gamma$. Hence, it can be seen as another smoothness condition on the level sets of $f$. Like Assumptions~\ref{ass:A3} and~\ref{ass:A4}, $f$ will verify Assumption~P if it is smooth enough. More specifically, the next result holds (see Section~B.3 of Supplementary~B, \citealp{suppB}).

\begin{prop}
	\label{prop:assA5}
	Let $M$ be a Riemannian manifold under Assumptions M. Let \mbox{$f: M \rightarrow \R$} be a density function. Let $0 < l \leq u < \sup f$, and let $U$ be an open set such that $f^{-1} \big( [l, u] \big)\subset U$. Suppose that $f$ is a $\mathcal{C}^{1}$ function in $U$ and there exists a positive constant $c > 0$ satisfying such that $\norm{ \nabla f }_M \geq c$ in $U$.
	
	Let $0 < \underline{\gamma} \leq \overline{\gamma} < 1$  and $\delta > 0$ such that $l < \lambda_{\underline{\gamma} - \delta} \leq \lambda_{\overline{\gamma} + \delta} < u$. Then, there exists a constant $k > 0$ such that
	\begin{equation*}
		\sup_{\gamma \in [\underline{\gamma}, \overline{\gamma}]} \abs{\lambda_{\gamma} - \lambda_{\gamma + \eta}} \leq k \abs{\eta}, \quad \text{ for all } \eta \in [-\delta, \delta].
	\end{equation*}
\end{prop}

\begin{remark}
	Notice that Proposition~\ref{prop:assA5} does not establish any relation between $c$ and $k$, whereas Proposition~\ref{prop:assA4} in Section~\ref{sec:assA} ensures that Assumption~\ref{ass:A4} holds for $k = 1/c$.
\end{remark}

Under Assumption~P, $\hat{\lambda}_{\gamma , n}$ in~\eqref{eq:lambdan} is a globally consistent estimator of $\lambda_{\gamma}$ and its convergence rate is known, as the next theorem shows.

\begin{theorem}
	\label{th:lambdaconvergence}
	Let $M$ be a Riemannian manifold and $\mathcal{X}_n = \{ X_1, \ldots, X_n \}$ an i.i.d.~sample from a random point $X$ from $M$. Suppose that $X$ has a density function $f$ satisfying Assumption~P. Let $f_n$ be an estimator of $f$ verifying:
	\begin{equation*}
		D_n = \sup_{x \in M} \abs{f(x) - f_n(x)} \rightarrow 0
	\end{equation*}
	almost surely.
	
	Then,
	\begin{equation*}
		\sup_{\gamma \in [\underline{\gamma}, \overline{\gamma} ] } \abs{\lambda_{\gamma} - \hat{\lambda}_{\gamma, n}} = O_{\textrm{\rm a.s.}} \left( \max \left\lbrace D_n, \sqrt{\frac{\log n}{n}} \right\rbrace \right).
	\end{equation*}
\end{theorem}

The proof of Theorem~\ref{th:lambdaconvergence} can be found in Section~A.3 of Supplement A \citep{suppA}. From Theorems~\ref{th:probs} and~\ref{th:lambdaconvergence} it immediately follows that $L_n(\hat{\lambda}_{\gamma, n})$ is a consistent estimator of $L (\lambda_{\gamma})$ and its convergence rate is derived.

\begin{theorem}
	\label{th:uniflevel}
	Let $M$ be a $d$-dimensional Riemannian manifold under Assumptions~M. Let $X$ be a random point from $M$ with density function $f$, and suppose that $f$ satisfies Assumptions~L. Let $\mathcal{X}_n = \{ X_1, \ldots, X_n \}$ be an i.i.d.~sample from $X$, and let $f_n$ be an estimator of $f$ such that:
	\begin{equation*}
		D_n = \sup_{x \in M} \abs{f(x) - f_n(x)} \rightarrow 0
	\end{equation*}
	almost surely. Furthermore, for any $\lambda \in [l, u]$, let $r_n (\lambda)$ be a sequence of positive random variables fulfilling
	\begin{equation*}
		\delta < r_n (\lambda) < r_0 (\lambda) - \delta, \quad \forall \lambda \in [l, u]
	\end{equation*}
	e.a.s.
	
	Hence,
	\begin{equation*}
		\sup_{\gamma \in [\underline{\gamma}, \overline{\gamma}]} \mathcal{d}_H \big( L(\lambda_{\gamma}), L_n(\hat{\lambda}_{\gamma, n}) \big)
		=
		O_{\textrm{\rm a.s.}} \left( \max \left\lbrace D_n, \sqrt{\frac{\log n}{n}}, \left( \frac{\log n}{n} \right)^{\frac{2}{d+1}} \right\rbrace \right).
	\end{equation*}
\end{theorem}

The term $(\log (n) / n)^{1/2}$ in the convergence rate of $L_n(\hat{\lambda}_{\gamma, n})$ comes from the estimation of the level and does not suppose a penalization in practice, since it gets dominated by $D_n$ under some mild conditions. For example, under the assumptions of Remark~\ref{rem:Dn}, $D_n$ is of order $(\log (n)/n)^{1/(d + 2)}$ and, consequently, Theorem~\ref{th:uniflevel} ensures that $L_n(\hat{\lambda}_{\gamma, n})$ converges at rate $(\log (n)/n)^{1/(d + 2)}$ too.

\section{The choice of $r_n (\lambda)$}
\label{sec:rn}

The estimator $L_n (\lambda)$ depends on a choice of the radius $r_n (\lambda)$ and the density function estimator $f_n$. As previously discussed, there exist several density function estimators in the manifold setting. An example is the kernel density estimator provided by Equation~\eqref{eq:kdemanifold}. However, to the best of the authors' knowledge, the problem of selecting a radius $r_n (\lambda)$ has only been addressed in the Euclidean setting (see \citealp{RodriguezCasal2022}).
This section addresses this problem with full generality and provides a data-driven choice of the radius for manifold data.

The new choice of $r_n (\lambda)$ is an estimator of the true radius $r_0 (\lambda)$. Our proposal is to recover $r_0 (\lambda)$ with
\begin{equation}
	\label{eq:defrn}
	r_n (\lambda) = \sup \Big\lbrace r \geq 0 \colon \mathcal{d}_H \Big( \mathcal{X}_n^{+} (\lambda) \cap \big( \mathcal{X}_n^{-} (\lambda) \oplus r B \big)^c , \mathcal{X}_n^{+} (\lambda) \Big) \leq r + h_n \Big\rbrace;
\end{equation}
where $h_n$ is a sequence of positive random variables converging to zero. This estimator is based on the very same ideas as $L_{n} (\lambda)$. Just note that
\begin{equation*}
	L (\lambda) = \Psi_r \big[ L(\lambda) \big]
	\Leftrightarrow
	\mathcal{d}_H \big(L (\lambda) \ominus r B , L (\lambda) \big) \leq r
\end{equation*}
Since $L (\lambda) \ominus r B = L (\lambda) \cap \big[ L (\lambda)^c \oplus r B\big]^c$, this implies
\begin{equation}
	\label{eq:char}
	r_0(\lambda)
	=
	\sup \Big\lbrace r > 0 \colon \mathcal{d}_H \Big(L (\lambda) \cap \big[ L (\lambda)^c \oplus r B\big]^c , L (\lambda) \Big) \leq r \Big\rbrace.
\end{equation}
The estimator $r_n (\lambda)$ relies on the expression of $r_0 (\lambda)$ given by \eqref{eq:char} and replaces $L(\lambda)$ and $L(\lambda)^{c}$ by their discretizations $\mathcal{X}_n^{+} (\lambda)$ and $\mathcal{X}_n^{-} (\lambda)$, respectively. However, this substitution adds some random error, in the sense that
\begin{equation}
	\label{eq:approxerror}
	\mathcal{d}_H \big( L (\lambda) \cap \big[ L (\lambda)^c \oplus r B\big]^c , L (\lambda) \big)
	-
	\mathcal{d}_H \big( \mathcal{X}_n^{+} (\lambda) \cap \big( \mathcal{X}_n^{-} (\lambda) \oplus r B \big)^c , \mathcal{X}_n^{+} (\lambda) \big)
\end{equation}
will be different from zero with probability one. The role of the sequence $h_n$ in Equation~\eqref{eq:defrn} is to deal with this random error that naturally occurs when $L(\lambda)$ and $L(\lambda)^{c}$ are replaced by their discretizations. If $h_n$ goes to zero slower than \eqref{eq:approxerror}, then $r_n (\lambda)$ will be a consistent estimator of $r_0 (\lambda)$.

\begin{theorem}
	\label{th:rconsistent}
	Let $M$ be a Riemannian manifold under Assumptions M. Let $X$ be a random point from $M$ with density function $f$, and suppose that $f$ satisfies Assumptions L. Let $\mathcal{X}_n = \{ X_1, \ldots ,X_n \}$ be an i.i.d.~sample from X, and $f_n$ an estimator of $f$ such that
	\begin{equation*}
		D_n = \sup_{x \in M} \vert f_n(x) - f(x) \vert \to 0
	\end{equation*}
	almost surely. Take $r_0 (\lambda)$ and $r_n (\lambda)$ as in Equations~\eqref{eq:defr} and~\eqref{eq:defrn}, respectively.
	
	If $r_0 (\lambda)$ is continuous in $\lambda \in [l, u]$ (where $l$ and $u$ are the constants in Assumptions L), $h_n \to 0$ almost surely, and
	\begin{equation}
		\label{eq:hn}
		h_n^{-1} \bigg( \dfrac{\log n}{n} \bigg)^{1/d} \to 0,
		\quad \text{ and } \quad
		h_n^{-1} D_n \to 0,
	\end{equation}
	almost surely, then
	\begin{equation*}
		\sup_{\lambda \in [l, u] } \vert r_n (\lambda) - r_{0}(\lambda) \vert \to 0
	\end{equation*}
	almost surely.
\end{theorem}

The proof of Theorem~\ref{th:rconsistent} can be found in Section~A.4 of Supplement~A \citep{suppA}. As previously discussed, $h_n$ should converge to zero slower than \eqref{eq:approxerror} for consistency, and this is ensured by Equation~\eqref{eq:hn} (see Corollary~A.4.1 in \citealp{suppA}). Under the assumptions of Remark~\ref{rem:Dn}, $D_n$ is of order $(\log (n)/n)^{1/(d+2)}$, so Equation~\eqref{eq:hn} holds for $h_n = n^{-1/(d+3)}$. The more conservative choice $h_n = 1/\log(n)$ satisfies Equation~\eqref{eq:hn} for any dimension. It is also possible to select $h_n$ in a way that takes the scale of the data into account, for example by using $h_n = \sigma_n / \log (n)$, where $\sigma_n$ is the average between two different observations, i.e., $\sigma_n = 2 n^{-1} (n-1)^{-1} \sum_{1 \leq i < j \leq n} \mathcal{d} (X_i, X_j)$.

Theorem~\ref{th:rconsistent} assumes that $r_0$ is a continuous function, a hypothesis also required by \cite{RodriguezCasal2022}. It is conjectured that the continuity of $r_0$ is guaranteed if $f$ is smooth enough, although we were unable to find a proof. Besides, the continuity of $r_0$ is only needed to ensure that the convergence is uniform in $\lambda$. The estimator $r_n (\lambda)$ will always converge pointwise to $r_0 (\lambda)$ if Equation~\eqref{eq:hn} holds, even though $r_0$ is not continuous.

Theorem~\ref{th:uniform} requires that the choice of the radius satisfies Equation~\eqref{eq:Rn} for $L_n (\lambda)$ to be consistent. However, Theorem~\ref{th:rconsistent} yields $r_n (\lambda) \to r_0(\lambda)$ almost surely, and therefore $r_n (\lambda)$ does not satisfies Equation~\eqref{eq:Rn} and cannot be used directly as a choice of the radius. This problem is readily solved taking $\tilde{r}_n (\lambda) = \nu r_n (\lambda)$, where $\nu$ is any number between $0$ and $1$. This new quantity $\tilde{r}_n (\lambda)$ satisfies Equation~\eqref{eq:Rn}, so the estimator $L_n (\lambda)$ based on this radius choice is consistent and the convergence rates are known. Notice that the estimation of $r_0 (\lambda)$ does not add any penalty factor to the convergence rate of $L_n (\lambda)$.

\begin{figure}[t!]
	\centering
	\begin{subfigure}{1.7in}
		\centering
		\includegraphics[width = 1.7in]{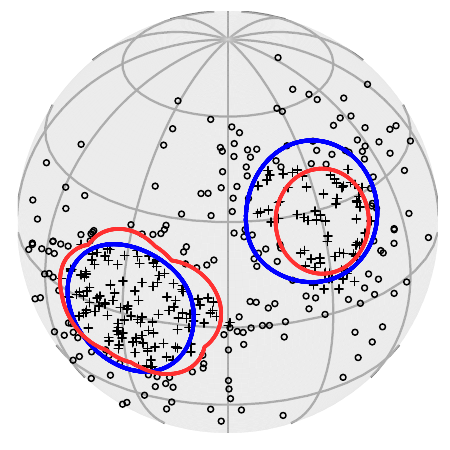}
		\caption{\label{fig:exampleLnrn400}}
	\end{subfigure}
	\hfill
	\begin{subfigure}{1.7in}
		\centering
		\includegraphics[width = 1.7in]{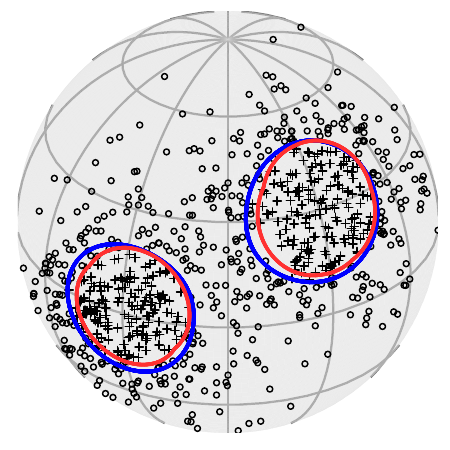}
		\caption{\label{fig:exampleLnrn800}}
	\end{subfigure}
	\hfill
	\begin{subfigure}{1.7in}
		\centering
		\includegraphics[width = 1.7in]{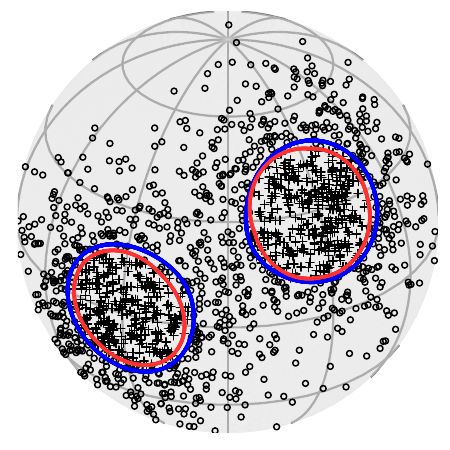}
		\caption{\label{fig:exampleLnrn1600}}
	\end{subfigure}
	\caption{\label{fig:exampleLnrn}
		HDR estimation of a mixture of two von Mises-Fisher distributions. All three graphs show the boundary of the estimator $L_n (\lambda)$ with $\lambda = 0.45$ for a mixture of the distributions $\mathrm{M}_2 (\mu_1, 10)$ and $\mathrm{M}_2 (\mu_2, 10)$ with equal weights (see equation~\eqref{eq:mus} for a definition of $\mu_1$ and $\mu_2$). Three sample sizes are considered, one for each graph: $n = 400$~(a), $n = 800$~(b) and $n = 1600$~(c). The boundary of $L_n (\lambda)$ is plotted in red in all graphs, whereas the boundary of the true HDR $L (\lambda)$ is depicted as a blue line. For reference, the points of the subsamples $\mathcal{X}^{+}_n ( \lambda )$ and $\mathcal{X}^{-}_n ( \lambda )$ are shown in the graphs as `$+$' and `$\circ$', respectively. In all three cases, the radius is chosen as $0.99r_n (\lambda)$ with $h_n = 1/\log (n)$ and $f_n$ is the kernel density estimator with von Mises-Fisher kernel (\citealp{GarciaPortugues2013}) and concentration parameter selected via cross-validation.}
\end{figure}

We apply the estimator $r_n (\lambda)$ to the same simulated data used in Figure~\ref{fig:exampleLn} in order to illustrate its behavior in practice. The values of $r_n (\lambda)$ obtained for the three sample sizes are \mbox{$r_n (\lambda) = 0.254$} ($n = 400$), $r_n (\lambda) = 0.291$ ($n = 800$), and $r_n (\lambda) = 0.296$ ($n = 1600$), where all three values were computed using $h_n = 1/\log(n)$ and rounded to three decimal places. We have also computed the corresponding HDR estimators $L_n (\lambda)$ with radii $0.99r_n (\lambda)$, which are shown in Figure \ref{fig:exampleLnrn}. The differences with Figure~\ref{fig:exampleLn} are remarkable, specially for $n = 400$ and $n = 800$. This illustrates that, even though the asymptotics are the same in both situations (in terms of consistency rates), the choice of the radius may notably impact the finite-sample behavior of the estimator.

This example also allows us to illustrate that choosing a larger radius might be counterproductive. Figure~\ref{fig:exampleLnr0} shows the same situation as in Figure~\ref{fig:exampleLnrn} but using $r_n (\lambda)$ directly as a radius for $L_n (\lambda)$, without shrinking. This slight modification drastically changes the behavior of the estimator. Now, $L_n (\lambda)$ fails to recover one of the two connected components of $L(\lambda)$. Hence, the condition $r_n (\lambda) < r_0 (\lambda)$ is not just a technical assumption required for our proof of Theorem~\ref{th:uniform}, but a true necessary condition for the consistency of $L_n (\lambda)$ that has tangible consequences in practical examples.

\begin{figure}[t]
	\centering
	\begin{subfigure}{1.7in}
		\centering
		\includegraphics[width = 1.7in]{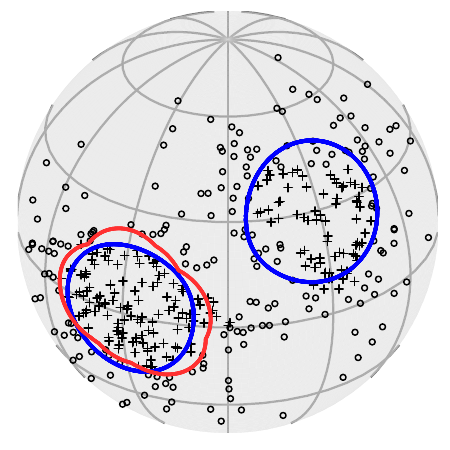}
		\caption{\label{fig:exampleLnr0400}}
	\end{subfigure}
	\hfill
	\begin{subfigure}{1.7in}
		\centering
		\includegraphics[width = 1.7in]{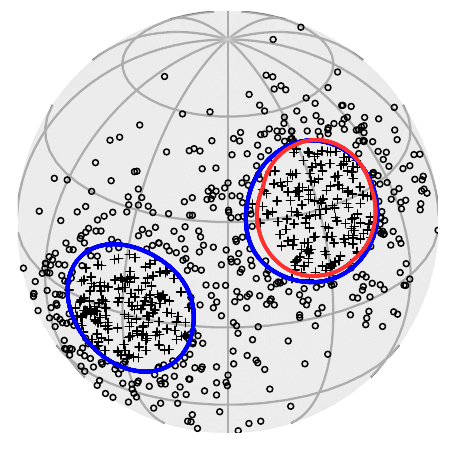}
		\caption{\label{fig:exampleLnr0800}}
	\end{subfigure}
	\hfill
	\begin{subfigure}{1.7in}
		\centering
		\includegraphics[width = 1.7in]{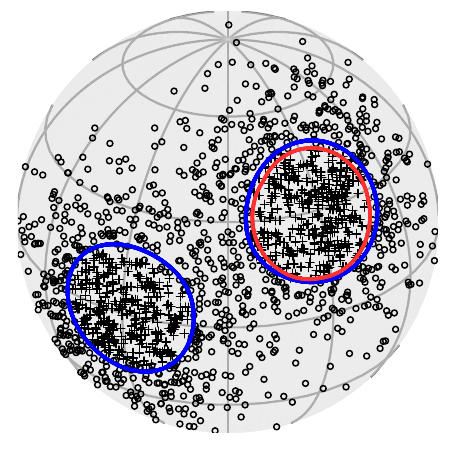}
		\caption{\label{fig:exampleLnr01600}}
	\end{subfigure}
	\caption{\label{fig:exampleLnr0} Same setup as in Figure~\ref{fig:exampleLnrn}, but the radius is chosen directly as $r_n (\lambda)$ with $h_n = 1/\log n$.}
\end{figure}

\begin{remark}
	Similarly to what was discussed in Section~\ref{sec:HDRgamma}, the problem of estimating $r_0$ could be posed in terms of the probability content $\gamma\in(0,1)$ instead of the level. The fact that $r_n (\lambda)$ converges to $r_0 (\lambda)$ uniformly in $\lambda$ guarantees that $r_n(\hat{\lambda}_{\gamma, n})$ is a consistent estimator of $r_0(\lambda_{\gamma})$ (where $\hat{\lambda}_{n, \tau}$ is the level estimator given by \eqref{eq:lambdan}), and the proof is analogous to the proof of Theorem~\ref{th:probs}. Nevertheless, notice that both Theorems~\ref{th:probs} and~\ref{th:uniflevel} only require consistency of $r_n$ with respect to the level $\lambda$, and not in terms of the probability content $\gamma$. In consequence, consistency of $r_n$ in with respect to $\gamma$ is not relevant in the problem under scope.
\end{remark}

\section{Real data illustration}
\label{sec:realdataHDR}

The performance in practice of the new estimator $L_n (\hat{\lambda}_{\gamma, n})$ is illustrated with two real data examples, supported on two different manifolds. In the first case data are a point cloud in the 2--dimensional sphere ($M = \mathbb{S}^2$), whereas in the second data are registered in a torus ($M = \mathbb{S}^1 \times \mathbb{S}^1$). 

\subsection{Long-period comets}
\label{sec:comets}

The first example is the comet orbits data introduced in Section~\ref{sec:intro}.
These data are available from the JPL Small-Body Database (\url{https://ssd.jpl.nasa.gov/tools/sbdb_query.html}) by NASA. In this repository, orbits of objects in the solar system are registered measuring two angles called inclination, $i \in [0, \pi]$, and longitude of the ascending node, $\Omega \in [0, 2\pi)$. Given an orbit with parameters $(i, \Omega)$, its normal unit vector is $v = (\sin(i) \sin(\Omega), - \sin(i) \cos( \Omega), \cos(i))' \in \mathbb{S}^2$, where the direction of the vector depends on the rotation direction of the comet (see \citealp{Jupp2003}). We considered the normal unit vectors of all the long-period comets detected up to 30th of May 2022. This data set is known to have duplicate entries, so, similarly to \cite{GarciaPortugues2023}, only comets with distinct $(i, \Omega)$ up to the second decimal position are considered, obtaining $n = 612$ comets in total. The normal unit vector of these $n = 612$ comet orbits are represented in Figure~\ref{fig:comets}.

\begin{figure}[!b]
	\centering
	\includegraphics{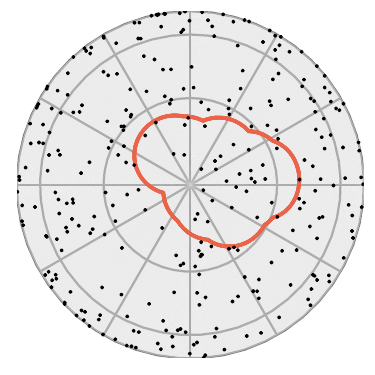}
	\hspace{0.5cm}
	\includegraphics{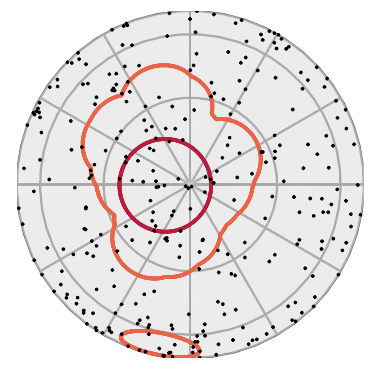}
	\caption{\label{fig:hdrcomets} Comet orbits data and their HDR estimators $L_n (\hat{\lambda}_{\gamma, n})$ for $\gamma = 0.8$ (in~{\color{laranxa}\faCircle}) and $\gamma = 0.95$ (in~{\color{vermello}\faCircle}). The data and the estimators are represented on the sphere using orthogonal projections centered on the north pole (left) and south pole (right).}
\end{figure}

As explained in Section~\ref{sec:intro}, the orbits of this kind of comets are expected to be uniformly distributed due to their formation process. Nevertheless, the observed sample have been shown to be not uniform (see \citealp{GarciaPortugues2023}), a situation that is suggested to be caused by an observational bias: astronomers search for new objects near the ecliptic more frequently than elsewhere on the celestial sphere, and therefore observations corresponding to comets located near the ecliptic plane are overrepresented within the sample. HDR estimation allows to check whether this hypothesis has empirical evidence or not, since, if true, the HDRs of the populations should be located near the north and south poles, as such normal vectors correspond to orbits near the ecliptic.

Figure~\ref{fig:hdrcomets} shows the boundary of the estimator $L_n (\hat{\lambda}_{\gamma, n})$ for $\gamma = 0.8$ and $\gamma = 0.95$. For the two choices of $\gamma$, the HDR estimator was computed with radii $0.99 r_n (\hat{\lambda}_{\gamma, n})$, choosing \mbox{$h_n = 1/\log n$ }for computing $r_n$. In both cases, $f_n$ was the kernel density estimator with von Mises-Fisher kernel (\citealp{GarciaPortugues2013}) and concentration parameter chosen via Likelihood Cross-Validation (LCV). For \mbox{$\gamma = 0.8$}, $L_n (\hat{\lambda}_{\gamma, n})$ consists of three connected components, two of them approximately centered near the north and south pole, and a third one near the equator. For $\gamma = 0.95$, $L_n (\hat{\lambda}_{\gamma, n})$ only has one connected component located near the south pole. Surprisingly, this seems to point out that one direction of rotation is slightly more common than the other. However, since normal vectors near the south pole correspond to orbits near the ecliptic plane, this behavior still provides empirical evidence supporting the observational bias theory.


We also compute the plug-in HDR estimator for directional data introduced by \cite{SaavedraNieves2021} for comparison. Figure~\ref{fig:hdrplugcomets} shows the boundary of the set $f_n^{-1} \big( [\hat{\lambda}_{\gamma, n}, + \infty] \big)$ for $\gamma = 0.8$ and $\gamma = 0.95$, where $f_n$ is the kernel density estimator with von Mises-Fisher kernel (\citealp{GarciaPortugues2013}) and LCV concentration parameter. The shape of the plug-in HDR estimators is quite similar to the ones shown in Figure~\ref{fig:hdrcomets}, having the same number connected components with roughly the sample location. Notice that, even though we draw the same conclusions from both estimators, the new proposal is much simpler to compute, store and represent than the plug-in estimator, since it just consists of a list of centers and a radius.


\begin{figure}[!b]
	\centering
	\includegraphics{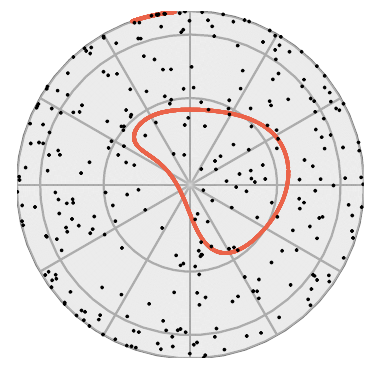}
	\hspace{0.5cm}
	\includegraphics{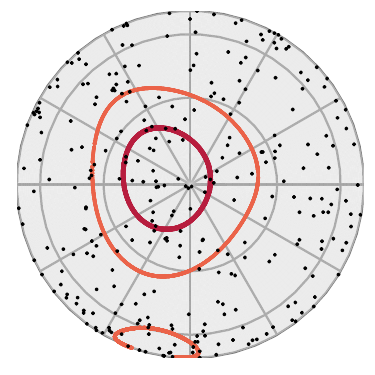}
	\caption{\label{fig:hdrplugcomets} Comet orbits data and their plug-in HDR estimators $f_n^{-1} \big( [\hat{\lambda}_{\gamma, n}, + \infty] \big)$ for $\gamma = 0.8$ (in~{\color{laranxa}\faCircle}) and $\gamma = 0.95$ (in~{\color{vermello}\faCircle}). The data and the estimators are represented on the sphere using orthogonal projections centered on the north pole (left) and south pole (right).}
\end{figure}

\subsection{Circadian gene expressions}

Several organisms present biological processes that are synchronized with the solar day. For example, mammals have multiple physiological parameters that fluctuate in an approximately 24 hours period, like blood pressure, heartbeat rate, and liver metabolism (see, for instance, \citealp{Balsalobre1998}). These biochemical oscillators are called circadian clocks, and they allow the organism to prepare to the usual activities that occur during a day.  

There are a large variety of genes involved in circadian clocks. The expression of these genes is characterized by exhibiting a periodical behavior, with a time of maximum expression that repeats every 24 hours and receives the name of \textit{phase}. This periodicity should be taken into account to correctly analyze circadian gene data. Thus, phases are often regarded as circular data, i.e., points of the unit circle $\mathbb{S}^1$. Indeed, circadian clocks are classical example of circular data, see \cite{Brown1976} and \cite{Batschelet1981}.

The genetic circadian behavior has been shown to be tissue-specific. That is, the same gene may have a circadian expression in one tissue and not in another. Or, even when a gene shows a circadian behavior in both tissues, its phase may be different in every tissue. In order to explore this tissue-specific behavior, \cite{Liu2006} analyzed 48 common circadian genes present in the heart and liver of mice detected by \cite{Storch2002}. \cite{Liu2006} studied the difference between the phases in both tissues and they discovered two groups of genes: one consisting of 30 genes with similar phases in heart and liver, and a smaller group of 10 genes with an average lag of 8 hours between the two.

The new HDR estimator allows to analyze both phases jointly, and therefore to perform a more detailed exploration of the data. Since both phases are points in the unit circle $\mathbb{S}^1$, the joint sample is a point cloud in the two-dimensional torus $\mathbb{T}^2 = \mathbb{S}^1 \times \mathbb{S}^1$. Figure~\ref{fig:hdrcircadian} shows the heart and liver phases of all 48 genes and the boundary of the estimator $L_n (\hat{\lambda}_{\gamma, n})$ for three different values of $\gamma$: $\gamma = 0.25$, $\gamma = 0.5$, and $\gamma = 0.75$. In all cases, $f_n$ was the kernel density estimator with von Mises kernel and LCV concentration parameter (\citealp{DiMarzio2011}). The three HDR estimators were computed with radii $0.99 r_n (\hat{\lambda}_{\gamma, n})$, choosing $h_n = 1/\log n$ for computing $r_n$. 


\begin{figure}[!b]
	\centering
	\includegraphics[width = 2.5in]{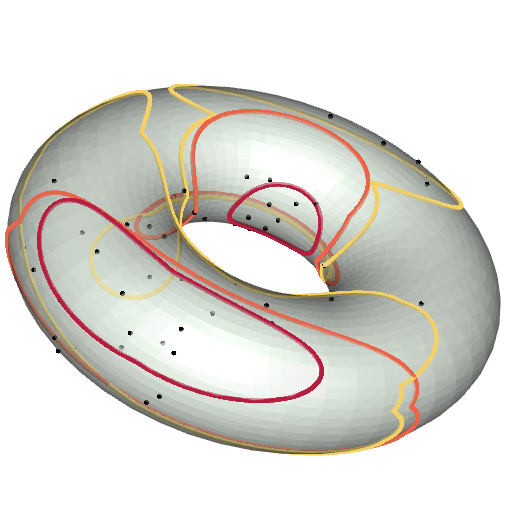}
	\hspace{0.5cm}
	\includegraphics[width = 2.5in]{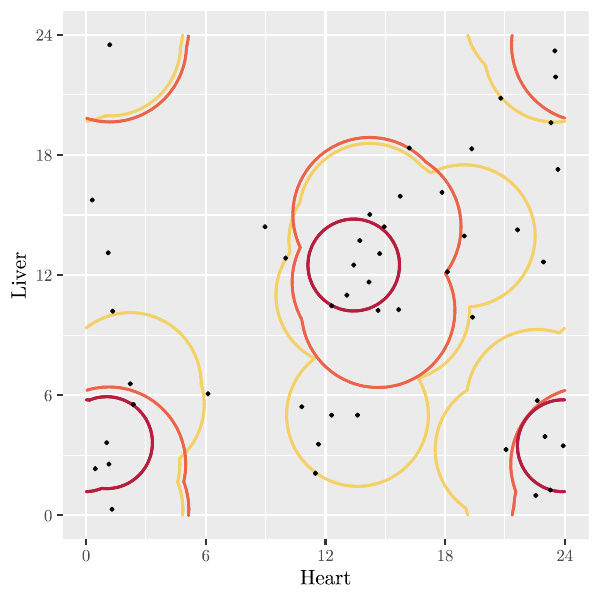}
	\caption{\label{fig:hdrcircadian} Circadian gene data (\citealp{Liu2006}) and their HDR estimator for $\gamma = 0.25$ (in~{\color{amarelo} \faCircle}), $\gamma = 0.5$ (in~{\color{laranxa} \faCircle}) and $\gamma = 0.75$ (in {\color{vermello} \faCircle}). The data are represented in the torus $\mathbb{T}^2$ (left) and in the $2$-dimensional plane (right). Phases were rescaled between $0$ and $24$ for better interpretation.}
\end{figure}

Figure~\ref{fig:hdrcircadian} shows that $L_n (\hat{\lambda}_{\gamma, n})$ is relatively close to the diagonal for the three values of $\gamma$. This fact seems to point out that circadian genes commonly exhibit a similar behavior in both tissues. In addition, the HDR estimators reveal the presence of two different clusters within the data: one centered at $(0,0)$ and other around $(12, 12)$. One group consists of \textit{morning genes}, i.e., genes with both phases near $0$, indicating that the expression of these genes peaks around the beginning of circadian day in both tissues. The second group is composed of \textit{night genes}, with phases close to the start of circadian night in heart and liver. Figure~\ref{fig:hdrcircadian} also illustrates that the new estimator does not have to be monotonic in $\gamma$ (or $\lambda$), since $L_n (\hat{\lambda}_{0.5, n})$ is not included in $L_n (\hat{\lambda}_{0.25, n})$. This non-monotonic behavior is shared by other non plug-in estimation techniques (see, for example, \citealp{Polonik1995a}). If the monotonicity of $L_n(\lambda)$ is an important feature for a particular application, the practitioner could use $\inf_{\lambda \in \Lambda} r_n (\lambda)$ as a choice for the radius, where $\Lambda \subset (0, + \infty)$ is the set of all considered levels.

\begin{figure}[!b]
	\centering
	\includegraphics[width = 2.5in]{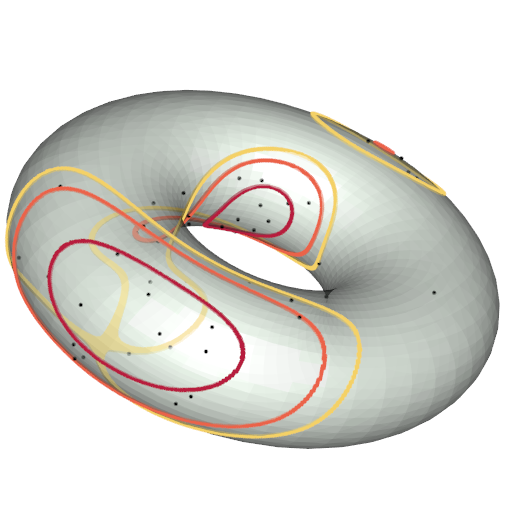}
	\hspace{0.5cm}
	\includegraphics[width = 2.5in]{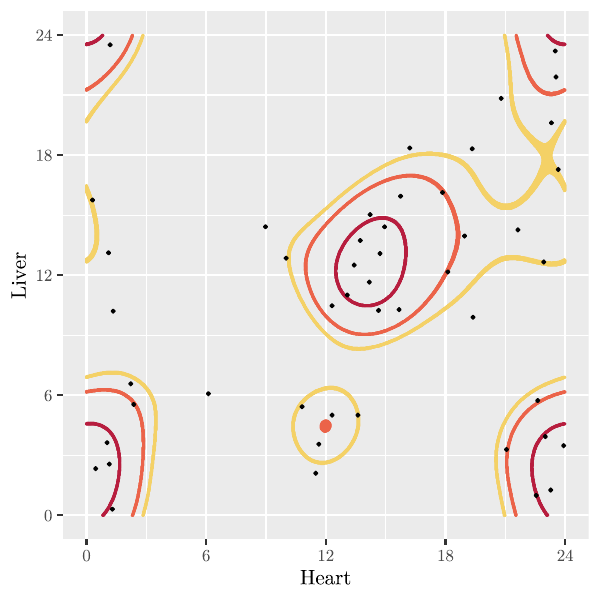}
	\caption{\label{fig:hdrplugcircadian} Circadian gene data (\citealp{Liu2006}) and their plug-in HDR estimators $f_n^{-1} \big( [\hat{\lambda}_{\gamma, n}, + \infty] \big)$ for $\gamma = 0.25$ (in~{\color{amarelo} \faCircle}), $\gamma = 0.5$ (in~{\color{laranxa} \faCircle}) and $\gamma = 0.75$ (in {\color{vermello} \faCircle}). The data are represented in the torus $\mathbb{T}^2$ (left) and in the $2$-dimensional plane (right). Phases were rescaled between $0$ and $24$ for better interpretation.}
\end{figure}

For comparison, we also estimate the HDRs of the circadian gene data using the plug-in estimator for manifold data studied by \cite{Cholaquidis2022}. Figure~\ref{fig:hdrplugcircadian} shows the boundary of $f_n^{-1} \big( [\hat{\lambda}_{\gamma, n}, + \infty] \big)$ for $\gamma = 0.25$, $\gamma = 0.5$, and $\gamma = 0.75$, where $f_n$ is the kernel density estimator with von Mises kernel and LCV concentration parameter (\citealp{DiMarzio2011}). The general conclusions drawn from Figure~\ref{fig:hdrcircadian} still hold with the plug-in HDR estimator. The three estimators in Figure~\ref{fig:hdrplugcircadian} are located near the diagonal, which indicates that most circadian genes have a similar behavior in heart and liver. The distinction between morning and night genes described above can also be appreciated in Figure~\ref{fig:hdrplugcircadian}. However, the shape of the plug-in estimators are quite different from the ones in Figure~\ref{fig:hdrcircadian}. For example, for $\gamma = 0.25$, the distinction between morning and night genes in $f_n^{-1} \big( [\hat{\lambda}_{\gamma, n}, + \infty] \big)$ is not that clear, and a new small cluster appears at $(12, 6)$. For $\gamma = 0.5$, the plug-in estimator has a third connected component with no points on it that was not present in $L_n (\hat{\lambda}_{\gamma, n})$ (see Figure~\ref{fig:hdrcircadian}). These differences may be caused by the well-known fact that cross-validation techniques lead to undersmoothed estimators of the density function (see \citealp{Hall1987a}). Of course, this behavior of the LCV bandwidth especially affects the plug-in estimator, while in the new proposal the incorporation of the shape conditions into the estimation method reduces the effect that the undersmoothed density estimator $f_n$ has on the final form of the HDR estimator $L_n (\hat{\lambda}_{\gamma, n})$.

\section{Discussion and future work}\label{sec:gurus}

We have introduced a new estimation technique for HDRs in Riemannian manifolds. This new proposal is based on the HDR estimator for Euclidean data introduced by~\cite{Walther1997} adapted to a (quite general) manifold setting. The central result is Theorem~\ref{th:uniform}, which proves the consistency of the proposed estimator and establishes its convergence rate. Theorem~\ref{th:uniform} also ensures that the convergence of this estimator is uniform with respect to the level~$\lambda$. This allows one to use the new estimator for recovering the smallest set that satisfies a specified probability content through a plug-in approach. This technique is explored in Section~\ref{sec:plug-in}, where its consistency is proved. Additionally, Section~\ref{sec:lambdan} introduces an estimator for the corresponding level, and its convergence rate is derived. The new HDR estimator depends on a choice of a radius $r_n$ for its construction. There exists a proposal for selecting $r_n$ in the Euclidean setting, but this selector cannot be extended to manifold data. To solve this issue, a data-driven selector of $r_n$ for manifold data is provided following a novel approach, and its theoretical properties are studied in Section~\ref{sec:rn}. Finally, to highlight its flexibility, the new estimator is applied to analyze two real datasets supported on the sphere and the torus.

Perspectives for new research are rich and diverse. For example, the proposed HDR estimator may studied under weaker smoothness assumptions, that is, for levels $\lambda$ where the density is not continuous, such as a uniform or a mixture of uniform distributions. The seminal paper by \cite{Walther1997} addresses this point in the Euclidean context. In Section~3.2, he considers a refinement of $L_n (\lambda)$ specially suited for this nonsmooth setting that uses Monte-Carlo resampling combined with excess mass ideas. The resulting estimator achieves the same consistency rates as the excess mass approach \citep{Polonik1995} for dimension $d > 2$, which are known to be minimax (see Remarks 6 and 7 in \citealp{Walther1997}). However, this refinement is computer-intensive, so this extension of $L_n (\lambda)$ loses one of its main advantages: its computational feasibility. For that reason, it would be extremely valuable not simply to adapt this refinement for manifold data, but to do so in a way that does not compromise the applicability of the estimator.

The choice of the shrinking factor $\nu$ is also an interesting issue for future research. We have chosen $\nu = 0.99$ in all our examples, but Figure~\ref{fig:exampleLnr0} points out that selecting a value of $\nu$ too close to $1$ may be problematic. One could try to address this issue providing an optimal choice of $\nu$. The fact that the consistency rate of $L_n (\lambda)$ does not depend on $r_n(\lambda)$ and that the biggest differences between Figure~\ref{fig:exampleLn} and Figure~\ref{fig:exampleLnrn} are found when $n = 400$ make us think that one should rely on finite-sample results in order to provide a choice of $\nu$ which is useful in practice, not on asymptotics. Since the finite-sample behavior of $L_n (\lambda)$ is beyond the scope of this paper, we leave this question for future work.

The analysis of the two data examples in Section~\ref{sec:realdataHDR} raises the question of how to visualize variability in this context, i.e., how to make inference on the HDRs of a population supported on a manifold. In the Euclidean setting, \cite{Mason2009} show the asymptotic normality of the HDR plug-in estimators in terms of Lebesgue measure distance:
\begin{equation*}
	\mathcal{d}_{\mathrm{Vol}_{\R^d}} \big( L (\lambda), f_n^{-1} [\lambda, + \infty) \big) = \mathrm{Vol}_{\R^d} \big(L (\lambda) \triangle f_n^{-1} [\lambda, + \infty) \big)
\end{equation*}
where $A \triangle B = (A \setminus B) \cup (B \setminus A)$ for all $A, B \subset \R^d$. This distance, yet useful, does not allow one to develop confidence regions for the HDRs of the population. In this sense, \cite{Chen2017} derive the asymptotic distribution of the Hausdorff distance between the HDR of level $\lambda$ and its plug-in estimator, $\mathcal{d}_{H} \big( L (\lambda), f_n^{-1} [\lambda, + \infty) \big)$, and develop a bootstrap method to establish a confidence region for the boundary of $L (\lambda)$. \cite{Mammen2013} also propose a method based on the (Euclidean) plug-in HDR estimator to construct confidence regions for $L (\lambda)$ but relying on a entirely different approach that does not involve the Hausdorff distance. To the best of our knowledge, all the contributions in this line are restricted to the Euclidean context and the plug-in estimator, and this problem remains unexplored for other kinds of HDR estimators in a general manifold. Deriving the asymptotic distribution of $\mathcal{d}_{H} \big( L (\lambda), L_n (\lambda) \big)$ will not only provide a reliable technique to construct confidence regions for $L(\lambda)$, but will also allow to find a lower bound of the number of connected components of $L(\lambda)$, which is related with the number of clusters within the population.

The HDR estimator introduced in~\eqref{eq:setest} can be an interesting preliminary tool for several tasks in data analysis. For example, as mentioned above, the different connected components of~$L (\lambda)$ can be regarded as population clusters, and consequently, HDR estimation can be used for cluster analysis. This approach has been explored in the literature for Euclidean data, see \cite{Cuevas2001}, but, to the best of the authors' knowledge, it has not been treated in the manifold context yet. In addition, the number of connected components of $L_n (\lambda)$ is an estimator of the number of clusters, which is a crucial initial parameter to set in many clustering techniques. HDR estimation is also useful for outlier detection. For small values of $\tau$, $L (\lambda_{\tau})$ approaches the substantial support of the population, and a new observation is classified as an outlier if lies outside of it. Similarly to cluster analysis, this application of HDR estimation has been investigated in the Euclidean context (see \citealp{Devroye1980} and \citealp{Baillo2001}), but it has not been considered yet for data on manifolds.
Comparison of two (or more) populations is another possible application, as one can distinguish density functions by measuring the Hausdorff distance between their HDRs. This approach allows for different weights to be assigned to different values of~$\lambda$, thereby prioritizing differences in the central parts of the populations over the tails, or vice versa.

\bibliographystyle{apalike}
\bibliography{biblio}

\appendix

\section{Properties of Minkowski operations on manifolds}
\label{sec:propMink}

\cite{Matheron1975} studied the properties of the $\Psi_r$ mapping defined in Equation~(3.2) in the main manuscript \citep{main} for subsets of the Euclidean space. For example, in the Euclidean setting $\Psi_r$ satisfies:
\begin{itemize}
	\item If $s \leq r$, then $\Psi_r (A) \subset \Psi_s (A)$.
	\item If $A \subset A'$, then $\Psi_r (A) \subset \Psi_s (A')$.
	\item If $s \leq r$, then $\Psi_s \circ \Psi_r = \Psi_r$.
\end{itemize}
In this appendix, we prove that the properties of monotonicity and idempotence are preserved in the manifold setting. First, some simple statements regarding the monotonicity of $\Psi_r$ are shown.

\begin{prop}
	\label{prop:trivial}
	Let $M$ be a Riemannian manifold, $A, A' \subset M$ two arbitrary subsets, and $r > 0$ a positive constant. Then:
	\begin{enumerate}[label = (\alph*)]
		\item $\Psi_r (A) \subset A$.
		\item If $A \subset A'$, then $\Psi_r (A) \subset \Psi_r (A')$.
		\item $\cup_{A \in \mathcal{F}} \Psi_r (A) \subset \Psi_r ( \cup_{A \in \mathcal{F}} A )$, where $\mathcal{F} \subset \mathcal{P} (M)$.
		\item $\Psi_r (A \oplus r B) = A \oplus r B$.
	\end{enumerate}
\end{prop}

\begin{proof}
	The proof of this result is simple and is left to the reader.
\end{proof}

If $M$ is assumed to be a connected and complete manifold, the Hopf-Rinow Theorem (\citealp{Lee2018}, Cor.~6.21) guarantees that each pair of points in $M$ can be joined with a length-minimizing geodesic. Specifically, given two points in $M$, $x, y \in M$, there exists a geodesic curve $\alpha: [a, b] \rightarrow M$ such that
$$\alpha(a) = x, \quad \alpha(b) = y, \quad \text{ and } \quad \mathcal{d}(x, y) = \int_{a}^{b} \Vert \alpha' (t) \Vert dt.$$
Every admissible curve can be parameterized by arc length, so it will be always assumed that $\alpha: [0, \mathcal{d} (x, y)] \rightarrow M$ without loss of generality. Furthermore, since $\alpha$ minimizes the distance between $x$ and $y$, it also minimizes the distance between any two points within its path. That is to say  
\begin{equation*}
	\mathcal{d} \big( \alpha (t), \alpha (s) \big) = \abs{t - s}, \text{ for all } 0 \leq t, s \leq \mathcal{d}(x, y).
\end{equation*}

This simple property guarantees that the $\Psi_{\bullet}$ mapping and Minkowski operations behave satisfactorily when two radii are used, as the following two propositions show.
\begin{prop}
	\label{prop:randspsi}
	Let $M$ be a connected and complete Riemann manifold, $A \subset M$ and $0 < s \leq r$. The following statements hold:
	\begin{enumerate}[label = (\alph*)]
		\item $B_r [x] = \Psi_s \big( B_r [x] \big)$ for all $x \in M$.
		\item $\Psi_{r} ( A ) = \Psi_s \big( \Psi_{r} ( A ) \big)$.
		\item $\Psi_{r} ( A ) \subset \Psi_s ( A ).$
		\item If $A = \Psi_r(A)$, then $A = \Psi_s(A)$.
	\end{enumerate}
\end{prop}

\begin{proof}
	Step by step.
	\begin{enumerate}[label = (\alph*)]
		\item Let $y \in B_{r} [x]$. Denote by $\alpha$ the length-minimizing geodesic joining $x$ with $y$. We distinguish two cases:
		\begin{itemize}
			\item If $\mathcal{d} \big( x, y \big) \leq s$, then $y \in B_s [x] \subset B_{r} [x]$ and therefore $y \in \Psi_s \big( B_{r} [x] \big)$.
			\item If $\mathcal{d} \big( x, y \big) > s$, the value $t = \mathcal{d} \big( x, y \big) - s $ satisfies that $\mathcal{d} \big( x, \alpha (t) \big) = \mathcal{d} \big( x, y \big) - s $ and $\mathcal{d} \big( y, \alpha (t) \big) = s$. Furthermore, given any $z \in B_s \big[ \alpha(t) \big]$:
			\begin{equation*}
				\mathcal{d} \big( y, z \big) \leq \mathcal{d} \big( y, \alpha(t) \big) + \mathcal{d} \big( \alpha(t), z \big) \leq \mathcal{d} \big( x, y \big) - s + s = \mathcal{d} \big( x, y \big) \leq r.
			\end{equation*}
			So, $y \in B_s \big[ \alpha (t) \big] \subset B_{r} [x]$ and that implies $y \in \Psi_s \big( B_{r} [x] \big)$.
		\end{itemize}
		\item From (a) and Propositions~\ref{prop:trivial}a and~\ref{prop:trivial}c, the next chain of inclusions follows:
		\begin{multline*}
			\Psi_r (A) = \left(\bigcup_{B_{r} [x] \subset A} B_{r} [x]\right) =  \left(\bigcup_{B_{r} [x] \subset A} \Psi_s \big( B_{r}[x] \big) \right) \subset \\
			\subset \Psi_s \left( \bigcup_{B_{r} [x] \subset A} B_{r} [x]\right) = \Psi_s \big( \Psi_r (A) \big) \subset \Psi_r (A).
		\end{multline*}
		Hence, the equality also holds.
		\item It is straightforward from (b), $\Psi_r( A ) \subset A$ and Proposition~\ref{prop:trivial}b.
		\item From (c) and Proposition~\ref{prop:trivial}a, we deduce the next chain of inclusions
		\begin{equation*}
			A = \Psi_r ( A ) \subset \Psi_s (A) \subset A,
		\end{equation*}
		so the equality holds.\qedhere
	\end{enumerate}
\end{proof}

\begin{prop}
	\label{prop:randsmin}
	Let $M$ be a connected and complete Riemann manifold, $A \subset M$ and $r, s > 0$. Then:
	\begin{enumerate}[label = (\alph*)]
		\item $(A \oplus r B) \oplus s B = A \oplus (r + s) B$.
		\item $(A \ominus r B) \ominus s B = A \ominus (r + s) B$.
	\end{enumerate}
\end{prop}

\begin{proof}
	Let $x \in (A \oplus r B) \oplus s B$. Then, there exists a point $y \in A \oplus r B$ such that $x \in B_s [y]$. In addition, there exists a point $z\in A$ satisfying that $y \in B_r [z]$. From the triangular inequality, one derives
	\begin{equation*}
		\mathcal{d} (x, z) \leq \mathcal{d} (x, y) + \mathcal{d} (y, z) \leq r + s.
	\end{equation*}
	Therefore $x \in B_{r + s} [z]$ and then $x \in A \oplus (r + s)B$.
	
	Now let $x \in A \oplus (r + s)B$. There is a point $y \in A$ such that $x \in B_{r + s} [y]$. We distinguish two cases:
	\begin{itemize}
		\item If $\mathcal{d} (x, y) \leq r$, then $x \in A \oplus rB \subset (A \oplus r B) \oplus s B$.
		\item Assume that $\mathcal{d} (x, y) > r$. Let $\alpha: \big[ 0, \mathcal{d} (x, y) \big] \rightarrow M$ be the length-minimizing geodesic joining $y$ with $x$, that is, $\alpha$ satisfies
		$$\alpha(0) = y, \quad \alpha \big( \mathcal{d} (x, y) \big) = x, \quad \text{ and } \quad \mathcal{d}(x, y) = \int_{0}^{\mathcal{d} (x, y)} \Vert \alpha' (t) \Vert dt.$$
		Thus, $\alpha ( r ) \in B_r [ y ]$ and then $\alpha (r) \in A \oplus r B$. Moreover,
		$$d \big( \alpha(r), x \big) = d \big( x, y \big) - r \leq s$$
		Therefore $x \in B_s \big[ \alpha (r) \big]$ and so $x \in (A \oplus r B) \oplus s B$.
	\end{itemize}
	Hence, (a) holds. Part (b) is a direct consequence of (a):
	\begin{equation*}
		(A \ominus r B) \ominus s B = \big[ (A^c \oplus r B) \oplus s B \big]^c = \big[ A^c \oplus (r + s) B \big]^c = A \ominus (r + s) B. \qedhere
	\end{equation*}
\end{proof}

For closed subsets, Minkowski operations are closely related to the distance function.

\begin{prop}
	\label{prop:dist}
	Let $M$ be a connected and complete manifold, $A \subset M$ a closed subset, and $r > 0$. The following statements hold.
	\begin{enumerate}[label = (\alph*)]
		\item $A \oplus r B = \{ x \in M \colon \mathcal{d} (x, A) \leq r \}.$
		\item $A \ominus r B = \{ x \in M \colon \mathcal{d} (x, A^c) \geq r \}.$
		\item $A \oplus r B$, $A \ominus r B$ and $\Psi_{r} (A)$ are closed subsets of $M$.
	\end{enumerate} 
\end{prop}

\begin{proof}
	
	Let $x \in A \oplus rB$. Then, there exists $a \in A$ such that $x \in B_r [a]$ and so $\mathcal{d} (x, A) \leq r$. Reciprocally, let $x \in M$ such that $\mathcal{d} (x, A) \leq r$. Then,
	$B_{r + 1/n} [x] \cap A \neq \emptyset$ holds for all $n \in \N$. From Cantor's intersection theorem, it follows that
	\begin{equation*}
		B_{r} [x] \cap A = \bigcap_{n \in N} B_{r + 1/n} [x] \cap A \neq \emptyset.
	\end{equation*}
	So, there exists a point $a \in B_{r} [x] \cap A$. Therefore, $x \in B_{r} [a]$ and $x \in A \oplus r B$. This concludes the proof of (a).
	
	The proof of (b) is similar. Let $x \in A \ominus rB$. Then $B_r [x] \subset A$. Therefore, $B_r[x] \cap A^c = \emptyset$ and $\mathcal{d} (x, A^c) \geq r$. Assume then that $x \in M$ fulfills $\mathcal{d} (x, A^c) \geq r$. Consequently, $B_r (x) \cap A^c = \emptyset$, that is to say, $B_r (x) \subset A$. Since $A$ is closed, it follows that $ \overline{B_r (x)} = B_r [x] \subset A$ and then $x \in A \ominus rB$. 
	
	Part (c) is a direct consequence of (a) and (b).
\end{proof}

\begin{remark}
	Note that, since $(A \oplus rB)^c = A^c \ominus rB$ for every $A \subset M$, the previous proposition guarantees that Minkowski operations and $\Psi_r$ also preserve open subsets.
\end{remark}

By Proposition~\ref{prop:dist}c, the Minkowski sum is a functional that maps non--empty compact sets onto non--empty compact sets. So, it make sense to wonder whether this mapping is continuous with respect to the Hausdorff distance. The answer is affirmative: the Minkowski sum is continuous both on the set $A$ and the radius $r$. However, we can be more precise stating how this continuity works. For that matter, we need to define the upper Hausdorff distance between two sets. Given two compact and non-empty subsets of $M$, $A_1, A_2 \in F(M)$, the upper Hausdorff distance between $A_1$ and $A_2$ is
\begin{equation}
	\label{eq:uphaus}
	\overrightarrow{\mathcal{d}_H} ( A_1, A_2 )
	=
	\inf \{ r > 0 \colon A_1 \subset A_2 \oplus r B \}.
\end{equation}
Notice that $\overrightarrow{\mathcal{d}_H} ( A_1, A_2 ) = 0$ if and only if $A_1 \subset A_2$ and
\begin{equation}
	\label{eq:haus}
	\mathcal{d}_H ( A_1, A_2 )
	=
	\max \{ \overrightarrow{\mathcal{d}_H} ( A_1, A_2 ), \overrightarrow{\mathcal{d}_H} ( A_2, A_1 ) \}.
\end{equation}
Then, $\overrightarrow{\mathcal{d}_H} ( A_1, A_2 )$ may be thought as a one--sided version of the Hausdorff distance.

The functional $\overrightarrow{\mathcal{d}_H}$ is not a proper metric, since it is not symmetric neither positive definite. Nevertheless, it $\overrightarrow{\mathcal{d}_H}$ is always non--negative and fulfills the triangular inequality. That is to say, for any $A_1, A_2, A_3 \in F(M)$, it holds that
\begin{equation*}
	\overrightarrow{\mathcal{d}_H} ( A_1, A_3 )
	\leq
	\overrightarrow{\mathcal{d}_H} ( A_1, A_2 ) + \overrightarrow{\mathcal{d}_H} ( A_2, A_3 )
\end{equation*}
(the order of the sets $A_1, A_2$ and $A_3$ in the inequality above is important, since $\overrightarrow{\mathcal{d}_H}$ is not symmetric). For this reason, some authors call $\overrightarrow{\mathcal{d}_H}$ a "pseudoquasimetric", although this term is not standard.

The upper Hausdorff distance allows for more precise characterization of the continuity of the Minkowski sum. This continuity turns out to be a consequence of the monotonicity of $\overrightarrow{\mathcal{d}_H}$ with respect to the Minkowski sum.

\begin{prop}
	\label{prop:distdecr}
	Let $M$ be a connected and complete manifold. Let $A_1, A_2 \in F(M)$ two compact and non-empty subsets of $M$ and $r > 0$ a positive constant. The following two inequalities hold:
	\begin{enumerate}[label = (\alph*)]
		\item $\overrightarrow{\mathcal{d}_H} ( A_1 \oplus r B , A_2 \oplus r B ) \leq \overrightarrow{\mathcal{d}_H} ( A_1 , A_2 )$.
		\item $\mathcal{d}_H ( A_1 \oplus r B , A_2 \oplus r B ) \leq \mathcal{d}_H ( A_1 , A_2 )$.
	\end{enumerate}
\end{prop}

\begin{proof}
	Let $\varepsilon > 0$ such that $\overrightarrow{\mathcal{d}_H} ( A_1, A_2 ) \leq \varepsilon$. Therefore,
	\begin{equation*}
		A_1 \subset A_2 \oplus \varepsilon B
		\Rightarrow
		A_1 \oplus r B \subset A_2 \oplus (r + \varepsilon) B
		\Rightarrow
		\overrightarrow{\mathcal{d}_H} ( A_1 \oplus r B , A_2 \oplus r B ) \leq \varepsilon,
	\end{equation*}
	where we have applied Proposition~\ref{prop:randsmin}a. Part~(b) follows from~(a) and \eqref{eq:haus}.
\end{proof}

\begin{prop}
	\label{prop:continuousoplus}
	Let $M$ be a connected and complete manifold. Let $A_n \subset M$ be a sequence of non--empty compact sets, $A \subset M$ a non-empty compact set, and $r_n \in [0, + \infty)$ a bounded sequence of positive numbers.
	\begin{enumerate}[label = (\alph*)]
		\item If $\overrightarrow{\mathcal{d}_H} ( A_n, A ) \to 0$, then
		$$
		\overrightarrow{\mathcal{d}_H} \big( A _n \oplus r_n B, A \oplus \limsup_{n \to + \infty} r_n B \big) \to 0.
		$$
		\item If $\overrightarrow{\mathcal{d}_H} ( A , A_n ) \to 0$, then
		$$
		\overrightarrow{\mathcal{d}_H} \big( A \oplus \liminf_{n \to + \infty} r_n B, A \oplus r_n B \big) \to 0.
		$$
		\item If $\mathcal{d}_H ( A_n, A ) \to 0$ and $r_n \to r$, then
		$$
		\mathcal{d}_H \big( A_n \oplus r_n B, A \oplus r B \big) \to 0.
		$$
	\end{enumerate}
\end{prop}

\begin{proof}
	Taking into account that $\limsup_{n \to + \infty} r_n \leq \sup_{m \geq n} r_m$ and $r_n \leq \sup_{m \geq n} r_m$ hold for every $n \in \N$, we have
	\begin{equation*}
		A_n \oplus r_n B \subset
		\big( A \oplus \limsup_{n \to + \infty} r_n B \big) \oplus (\sup_{m \geq n} r_m - \limsup_{n \to + \infty} r_n) B.
	\end{equation*}
	Then, triangular equality and Proposition~\ref{prop:distdecr}a yield
	\begin{multline*}
		\overrightarrow{\mathcal{d}_H} \big( A_n \oplus r_n B, A \oplus \limsup_{n \to + \infty} r_n B \big)
		\\
		\leq
		\overrightarrow{\mathcal{d}_H} \big( A_n \oplus r_n B, A_n \oplus \limsup_{n \to + \infty} r_n B \big)
		+
		\overrightarrow{\mathcal{d}_H} \big( A_n \oplus \limsup_{n \to + \infty} r_n B, A \oplus \limsup_{n \to + \infty} r_n B \big)
		\\
		\leq 
		\sup_{m \geq n} r_m - \limsup_{n \to + \infty} r_n + \overrightarrow{\mathcal{d}_H} \big( A_n, A \big)
		\to 0,
	\end{multline*}
	showing that Part~(a) holds.
	
	The proof of Part~(b) is completely analogous. Part~(c) is a consequence of~(a) and~(b).
\end{proof}

Proposition~\ref{prop:continuousoplus}c implies both $\overrightarrow{\mathcal{d}_H} ( A _n \oplus r_n B, A \oplus r B ) \to 0$ and $\overrightarrow{\mathcal{d}_H} ( A \oplus r B, A _n \oplus r_n B ) \to 0$, but Parts (a) and (b) of the same result ensure that these two convergences also hold under much weaker assumptions.

The relation between Hausdorff distance and Minkowski difference is way more involved, though.

\begin{prop}
	\label{prop:continuousominus}
	Let $M$ be a connected and complete manifold. Let $A_n \subset M$ be a sequence of compact sets, $A \subset M$ a compact set, and $r_n \in [0, + \infty)$ a bounded sequence of positive numbers.
	
	If $A_n \ominus r_n B \neq \emptyset$ for all $n \in \N$ and $\overrightarrow{\mathcal{d}_H} ( A _n, A ) \to 0$, then
	$$
	\overrightarrow{\mathcal{d}_H} \big( A _n \ominus r_n B, A \ominus \liminf_{n \to + \infty} r_n B \big) \to 0.
	$$
\end{prop}

\begin{proof}
	Both $A \ominus \liminf_{n \to + \infty} r_n B \subset A$ and $A_n \ominus r_n B \subset A$ hold for all $n \in \N$. Since $A$ is a compact set, Lemma 2.14 of \cite{Evans2024} ensures that we only need to show
	\begin{equation*}
		\bigcap_{n = 1}^{+\infty} \bigg( \overline{ \bigcup_{m = n}^{+\infty} A _m \ominus r_m B } \bigg)
		\subset
		A \ominus \liminf_{n \to + \infty} r_n B
	\end{equation*}
	to prove the result.
	
	The inclusion
	\begin{equation*}
		A _k \ominus r_k B
		\subset
		\bigg( \overline{ \bigcup_{m = n}^{+\infty} A _m } \bigg) \ominus \inf_{m \geq n} r_m B
	\end{equation*}
	holds for all $k \geq n$. This and Proposition~\ref{prop:dist}c guarantee that
	\begin{equation*}
		\bigcap_{n = 1}^{+\infty} \bigg( \overline{ \bigcup_{m = n}^{+\infty} A _m \ominus r_m B } \bigg)
		\subset
		\bigcap_{n = 1}^{+\infty} \Bigg( \bigg( \overline{ \bigcup_{m = n}^{+\infty} A _m } \bigg) \ominus \inf_{m \geq n} r_m B \Bigg).
	\end{equation*}
	We will show that the set in the right-hand side is contained in $A \ominus \liminf_{n \to + \infty} r_n B$. Let $x \in \bigcap_{n = 1}^{+\infty} \Big( \big( \overline{ \bigcup_{m = n}^{+\infty} A _m } \big) \ominus \inf_{m \geq n} r_m B \Big)$. Hence, 
	$$
	B_{\inf_{m \geq n} r_m} [x] \subset \overline{ \bigcup_{m = n}^{+\infty} A _m }
	$$
	for all $n \in \N$. Let $\varepsilon > 0$. There exists a natural number, $n_0 \in \N$, such that $\inf_{m \geq n} r_m \geq \liminf_{n \to + \infty} r_n - \varepsilon$ for any $n \geq n_0$. Then,
	\begin{align*}
		B_{(\liminf_{n \to + \infty} r_n - \varepsilon)} [x]
		\subset&
		\bigcap_{n = n_0}^{+\infty} B_{\inf_{m \geq n} r_m} [x]
		\\
		\subset&
		\bigcap_{n = n_0}^{+\infty} \bigg( \overline{ \bigcup_{m = n}^{+\infty} A _m } \bigg)
		=
		\bigcap_{n = 1}^{+\infty} \bigg( \overline{ \bigcup_{m = n}^{+\infty} A _m } \bigg)
		\subset
		A,
	\end{align*}
	where the last inclusion follows from Lemma 2.14 of \cite{Evans2024}. This holds for any $\varepsilon > 0$, so $B_{\liminf_{n \to + \infty} r_n} (x) \subset A$. Since $A$ is a closed set, we have
	\begin{equation*}
		\overline{B_{\liminf_{n \to + \infty} r_n} (x)}
		=
		B_{\liminf_{n \to + \infty} r_n} [x]
		\subset
		A,
	\end{equation*}
	and $x \in A \ominus \liminf_{n \to + \infty} r_n B$. Consequently, 
	\begin{equation*}
		\bigcap_{n = 1}^{+\infty} \bigg( \overline{ \bigcup_{m = n}^{+\infty} A _m \ominus r_m B } \bigg)
		\subset
		\bigcap_{n = 1}^{+\infty} \Bigg( \bigg( \overline{ \bigcup_{m = n}^{+\infty} A _m } \bigg) \ominus \inf_{m \geq n} r_m B \Bigg)
		\subset
		A \ominus \liminf_{n \to + \infty} r_n B
	\end{equation*}
	and the result holds.
\end{proof}

Broadly speaking, Proposition~\ref{prop:continuousominus} states that the Minkowski sum is upper semicontinuous with respect to the Hausdorff metric. Full continuity does not hold in general. Take, for example, the following subsets of $\R^2$:
\begin{equation*}
	A_n = B_{2(n-1)/n} \big[(0,0)\big] \cup B_{(n-1)/n} \big[(4,0)\big], \quad A = B_{2} \big[(0,0)\big] \cup B_{1} \big[(4,0)\big].
\end{equation*}
It holds that $\mathcal{d}_H (A_n, A) \to 0$, but $\mathcal{d}_H (A_n \ominus 1B, A \ominus 1B)$ does not converge to zero since $\mathcal{d}_H (A_n \ominus 1B, A \ominus 1B) \geq 3$ for every $n \in \N$.

The functional $\Psi_{r}$ has a similar behavior to the Minkowski difference. One readily sees that it is not continuous: taking $A_n$ and $A$ as before, one has that $\mathcal{d}_H ( \Psi_1 (A_n) , \Psi_1 (A) )$ does not converge to zero even though $\mathcal{d}_H (A_n, A) \to 0$. But it follows directly from Propositions~\ref{prop:continuousoplus}a and~\ref{prop:continuousominus} that $\Psi_{r}$ is upper semicontinuous.

\begin{cor}
	\label{cor:continuouspsi}
	Let $M$ be a connected and complete manifold. Let $A_n \subset M$ be a sequence of compact sets, $A \subset M$ a compact set, and $r_n \in [0, + \infty)$ a convergent sequence of positive numbers, $r_n \to r$.
	
	If $A_n \ominus r_n B \neq \emptyset$ for all $n \in \N$ and $\overrightarrow{\mathcal{d}_H} ( A _n, A ) \to 0$, then
	$$
	\overrightarrow{\mathcal{d}_H} \big( \Psi_{r_n} (A_n), \Psi_{r} (A) \big) \to 0.
	$$
\end{cor}

Distance and Minkowski operations have a stronger connection than Proposition~\ref{prop:dist} shows. For example, given a closed subset $A$, the distances $\mathcal{d} (x, A^c)$ and $\mathcal{d} (x, A \ominus r B)$ are linked by a simple and useful equation.

\begin{prop}
	\label{prop:geodesic}
	Let $M$ be a connected and complete manifold, $A \subset M$ be a closed set and $r, \delta > 0$. Assume that $\Psi_{r} (A) = A$ and $\Psi_{\delta} (A^c) = A^c$. Then 
	\begin{equation*}
		\mathcal{d} (x, A^c) + \mathcal{d} (x, A \ominus r B) = r,
	\end{equation*}
	for all $x \in A \setminus \big[ A \ominus r B \big]$.
\end{prop}

\begin{proof}
	Let $x$ be an arbitrary point in $A \setminus \big[ A \ominus r B \big]$, and let $p \in \partial A$ such that $\mathcal{d} (x, p) = d(x, A^c)$. First, we need to show that there exists a point $y \in A^c$ satisfying $p \in B_{\delta} [y]$ and $B_{\delta} (y) \subset A^c$.
	
	Since $p \in \partial A$, there exists a sequence of points of $A^c$ satisfying $p_n \rightarrow p$. Keeping in mind that $A^c = \Psi_{\delta} (A^c)$, for every $p_n$ there exists a center $y_n$ such that $p_n \in B_{\delta} [y_n] \subset A^c$. The sequence $p_n$ converges to $p$, so there exists a constant $R > 0$ such that $\mathcal{d} (p, p_n) < R$ for all $n \in \N$. Thus, $\mathcal{d} (p, y_n) \leq \mathcal{d} (p, p_n) + \mathcal{d} (p_n, y_n) < R + \delta$ holds for all $n \in \N$ and the sequence $y_n$ is bounded. $M$ is a complete metric space, and so $y_n$ has a convergent subsequence. For simplicity, the convergent subsequence of centers and the corresponding $p_n$'s subsequence are named as the original ones, $\{y_n \}_{n \in \N}$ and $\{p_n\}_{n \in \N}$, $y_n \rightarrow y$, $p_n \rightarrow p$. Then $\mathcal{d} (y, p) = \lim_{n \rightarrow + \infty} \mathcal{d} (y_n, p_n) \leq \delta$ and $\mathcal{d} (y, A) = \lim_{n \rightarrow + \infty} \mathcal{d} (y_n, A) \geq \delta$. From the previous statements it follows that $p \in B_{\delta} [y]$ and $B_{\delta} (y) \subset A^c$, and since $p \in \partial A \subset A$, $\mathcal{d} (p, y) = \delta$.
	
	On the other hand, since $p \in A$ and $\Psi_r (A) = A$, there exists a point $q \in A$ such that $p \in B_r [q] \subset A$. Furthermore, $p \in \partial A$, so $d(p, q) = r$.
	
	\begin{figure}[b]
		\centering
		\def\scale{1}
\begingroup%
  \makeatletter%
  \providecommand\color[2][]{%
    \errmessage{(Inkscape) Color is used for the text in Inkscape, but the package 'color.sty' is not loaded}%
    \renewcommand\color[2][]{}%
  }%
  \providecommand\transparent[1]{%
    \errmessage{(Inkscape) Transparency is used (non-zero) for the text in Inkscape, but the package 'transparent.sty' is not loaded}%
    \renewcommand\transparent[1]{}%
  }%
  \providecommand\rotatebox[2]{#2}%
  \newcommand*\fsize{\dimexpr\f@size pt\relax}%
  \newcommand*\lineheight[1]{\fontsize{\fsize}{#1\fsize}\selectfont}%
  \ifx\svgwidth\undefined%
    \setlength{\unitlength}{141.73228346bp}%
    \ifx\svgscale\undefined%
      \relax%
    \else%
      \setlength{\unitlength}{\unitlength * \real{\svgscale}}%
    \fi%
  \else%
    \setlength{\unitlength}{\svgwidth}%
  \fi%
  \global\let\svgwidth\undefined%
  \global\let\svgscale\undefined%
  \makeatother%
  \begin{picture}(1,1)%
    \lineheight{1}%
    \setlength\tabcolsep{0pt}%
    \put(0,0){\includegraphics[width=\unitlength,page=1]{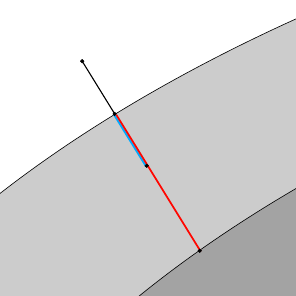}}%
    \put(0.24090717,0.81788258){\color[rgb]{0,0,0}\makebox(0,0)[lt]{\lineheight{1.25}\smash{\begin{tabular}[t]{l}$y$\end{tabular}}}}%
    \put(0.39060038,0.64757706){\color[rgb]{0,0,0}\makebox(0,0)[lt]{\lineheight{1.25}\smash{\begin{tabular}[t]{l}$p$\end{tabular}}}}%
    \put(0.49774361,0.45394236){\color[rgb]{0,0,0}\makebox(0,0)[lt]{\lineheight{1.25}\smash{\begin{tabular}[t]{l}$x$\end{tabular}}}}%
    \put(0.69335055,0.1129555){\color[rgb]{0,0,0}\makebox(0,0)[lt]{\lineheight{1.25}\smash{\begin{tabular}[t]{l}$q$\end{tabular}}}}%
    \put(0.36517134,0.47372873){\color[rgb]{0,0.66666667,1}\makebox(0,0)[lt]{\lineheight{1.25}\smash{\begin{tabular}[t]{l}$\beta_x$\end{tabular}}}}%
    \put(0.57577003,0.35152004){\color[rgb]{1,0.02352941,0}\makebox(0,0)[lt]{\lineheight{1.25}\smash{\begin{tabular}[t]{l}$\beta_q$\end{tabular}}}}%
    \put(0.27111843,0.66177039){\color[rgb]{0,0,0}\makebox(0,0)[lt]{\lineheight{1.25}\smash{\begin{tabular}[t]{l}$\alpha$\end{tabular}}}}%
  \end{picture}%
\endgroup%

		\caption{The four points $p, x, q$ and $y$, together with the length-minimizing segments $\alpha$, $\beta_x$ and $\beta_q$.}
		\label{fig:lemmadist}
	\end{figure}
	
	The next step is to prove that $p, x, q$ and $y$ lie on the same length-minimizing geodesic segment. This can be derived from uniqueness of geodesics. Let $\alpha$ be a geodesic segment joining $y$ and $p$; and $\beta_x$ a geodesic segment joining $p$ and $x$ (see Figure~\ref{fig:lemmadist} for an illustration). Since $p \in B_{\delta} [y] \cap B_{\mathcal{d} (x, p)} [x]$ and $B_{\delta} (y) \cap B_{\mathcal{d} (x, p)} (x) \subset A^c \cap A = \emptyset$, one derives that
	$$\mathcal{d} (x, y) = \mathcal{d} (x, p) + \delta = \mathcal{d} (x, p) + \mathcal{d} (p, y)$$
	and therefore, the curve
	\begin{equation*}
		\gamma_1 (t) =
		\left\lbrace
		\begin{array}{ll}
			\alpha (t), & t \in [0, \delta]\\
			\beta_x (t - \delta), & t \in [\delta, \delta + \mathcal{d}(x, p)]
		\end{array}
		\right.
	\end{equation*}
	satisfies
	\begin{equation*}
		\int_{0}^{\delta + \mathcal{d} (x, p)} \Vert \gamma_1' (t) \Vert dt = \delta + \mathcal{d} (x, p) = \mathcal{d} (x, c).
	\end{equation*}
	So, $\gamma_1$ is a length-minimizing curve joining $y$ and $x$, and then $\gamma_1$ is a geodesic segment.
	
	Let now $\beta_q$ be a geodesic segment from $p$ to $q$ (see Figure~\ref{fig:lemmadist}). Following a similar argument, one shows that
	\begin{equation*}
		\gamma_2 (t) =
		\left\lbrace
		\begin{array}{ll}
			\alpha (t), & t \in [0, \delta]\\
			\beta_q (t - \delta), & t \in [\delta, \delta + \mathcal{d} (q, p)]
		\end{array}
		\right.
	\end{equation*}
	is a length-minimizing curve and a geodesic segment from $y$ to $q$. But $\gamma_1 (t) = \gamma_2(t)$ for all $t \in [0, \delta]$, and since geodesics are unique, it follows that $\gamma_1 = \gamma_2$ in their common support. Particularly,
	\begin{equation*}
		\beta_q \big( \mathcal{d} (x, p) \big) = \gamma_2 \big( \delta + \mathcal{d} (x, p) \big) = \gamma_1 \big( \delta + \mathcal{d} (x, p) \big) = \beta_x \big( \mathcal{d} (x, p) \big) = x;
	\end{equation*}
	so the geodesic segment $\beta_q$ joins the points $p, x$ and $q$.
	
	Finally, one derives that
	\begin{multline*}
		r \geq \mathcal{d} (p, q) = \mathcal{d} (p, x) + \mathcal{d} (x, q) = \mathcal{d} (x, A^c) + \mathcal{d} (x, q) \geq \\
		\geq \mathcal{d} (x, A^c) + \mathcal{d} (x, A \ominus r B) \geq \mathcal{d} (A^c, A \ominus r B)  \geq r,
	\end{multline*}
	where the last inequality follows from $\Psi_r (A) = A$ and Proposition \ref{prop:dist}a. Consequently, $\mathcal{d} (x, A \ominus r B) + \mathcal{d} (x, A^c) = r$ holds.
\end{proof}

This section ends with three lemmas exploring how $\Psi_r$, Minkowski operations, and set difference are related. The first one yields that, if a set $A \subset M$ fulfills $\Psi_{r + s} (A) = A$, then $\Psi_s (A \ominus r B) = A \ominus r B$.

\begin{lemma}
	\label{lemma:psidifference}
	Let $M$ be a connected and complete Riemannian manifold. Let $A \subset M$ be a closed set and $s, r, \delta > 0$ some positive constants. Assume that $\Psi_{r + s} (A) = A$ and $\Psi_{\delta} (A^c) = A^c$. Then $\Psi_s (A \ominus r B) = A \ominus r B$.
\end{lemma}

\begin{proof}
	The content $\Psi_s (A \ominus r B) \subset A \ominus r B$ is straightforward by Proposition~\ref{prop:trivial}a. It follows from the definition of $\Psi_s$ that $A \ominus (r + s)B \subset \Psi_s (A \ominus r B)$. Let $x \in A \setminus \big[ A \ominus (r + s) B \big]$. Proposition \ref{prop:geodesic} guarantees that
	\begin{equation}
		\mathcal{d}(x, A^c) + \mathcal{d} \big( x, A \ominus (r + s) B \big) = r + s.
	\end{equation}
	Therefore, for all $x \in \big[ A \ominus r B \big] \setminus \big[ A \ominus (r + s) B \big] $, it holds that
	\begin{equation*}
		\mathcal{d} \big( x, A \ominus (r + s) B \big) = r + s - \mathcal{d}(x, A^c) \leq s;
	\end{equation*}
	where the last inequality is derived from Proposition \ref{prop:dist}b. Then, taking into account that Proposition~\ref{prop:randsmin}b ensures that $A \ominus (r + s) B = ( A \ominus rB ) \ominus sB$, Proposition \ref{prop:dist}a yields
	\begin{equation*}
		x \in \big[ (A \ominus rB) \ominus sB \big] \oplus s B = \Psi_s (A \ominus r B). \qedhere
	\end{equation*}
\end{proof}

\begin{remark}
	\label{rem:psidifference}
	It may not be obvious at first, but the hypothesis that $\Psi_{\delta} (A^c) = A^c$ for some $\delta > 0$ is necessary for Lemma \ref{lemma:psidifference} to hold. For example, consider the closed set $A \subset \R^2$ represented in Figure \ref{fig:remlemma1a} and defined by
	\[ A = B_{\sqrt{2}} \big[ (-1, 0) \big] \cup B_{\sqrt{2}} \big[ (1, 0) \big].\]
	This set satisfies that $\Psi_{\sqrt{2}} (A) = \Psi_{1 + (\sqrt{2} - 1)} (A) = A$, but $\Psi_{\delta} (A^c) = A^c$ does not hold for any $\delta > 0$.
	
	The set $A \ominus 1 B$ is represented in Figure \ref{fig:remlemma1b}. For this set, one can easily show that $\Psi_{\sqrt{2} - 1} (A \ominus 1 B) \neq A \ominus 1B$. In fact, $(0, 0) \in A \ominus 1B$ and $(0,0) \not\in \big(A \ominus \sqrt{2} B\big) \oplus (\sqrt{2} - 1) B$. Therefore \[\Psi_{\sqrt{2} - 1} (A \ominus 1B) = \big( A \ominus \sqrt{2} B \big) \oplus (\sqrt{2} - 1) B \neq A \ominus 1B.\]
	
	\begin{figure}[t]
		\centering
		\begin{subfigure}{2.7in}
			\centering
			\def\scale{1}
			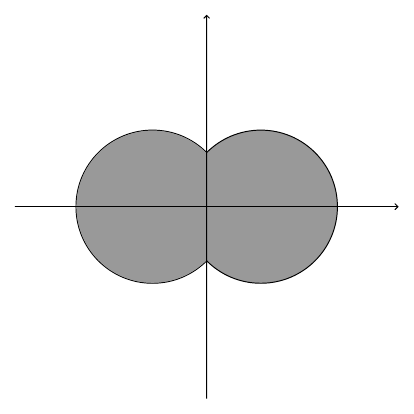
			\caption{$A$}
			\label{fig:remlemma1a}
		\end{subfigure}
		\begin{subfigure}{2.7in}
			\centering
			\def\scale{1}
			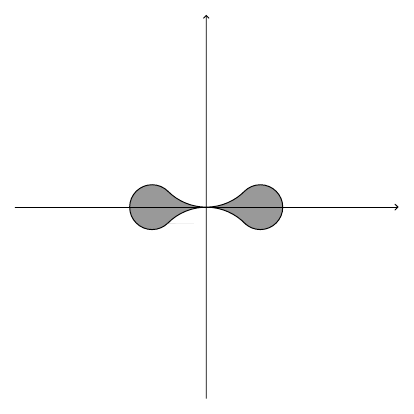
			\caption{$A \ominus 1B$}
			\label{fig:remlemma1b}
		\end{subfigure}
		\caption{\label{fig:remlemma1} The sets $A$ and $A \ominus 1B$ from Remark~\ref{rem:psidifference}.}
	\end{figure}
\end{remark}

The next lemma ensures that $\Psi_{s} \big( [A \oplus rB] \setminus A \big) = [A \oplus rB] \setminus A$ when $s$ is small enough.

\begin{lemma}
	\label{lemma:psiminusset}
	Let $M$ be a connected and complete manifold, $A \subset M$ a closed set and $r, s > 0$ two positive constants such that $s \leq r / 3$. If $\Psi_{s} (A^c) = A^c$ then $$\Psi_{s} \big( [A \oplus r B] \setminus A \big) = [A \oplus r B] \setminus A.$$
\end{lemma}

\begin{proof}
	If $[A \oplus r B] \setminus A = \emptyset$, the result is obvious. Contrarily, assume that $[A \oplus r B] \setminus A \neq \emptyset$. Let $x \in [A \oplus r B] \setminus A = [A \oplus r B] \cap A^c$. Since $A$ is a closed set, then $0 < \mathcal{d}(x, A) \leq r$. We distinguish three cases:
	\begin{enumerate}[label = {Case \arabic*:}, wide = \parindent, leftmargin = 0pt]
		\item $0 < \mathcal{d}(x, A) \leq r / 3$. Since $\Psi_{s} (A^c) = A^c$ there exists a point $y \in A^c$ satisfying $x \in B_{s} [y] \subset A^c$. Furthermore, every $z \in B_{s} [y]$ satisfies
		\begin{equation*}
			\mathcal{d}(z, A) \leq \mathcal{d}(z, y) + \mathcal{d}(y, x) + \mathcal{d}(x, A) \leq r;
		\end{equation*}
		so $B_{s} [y] \subset [A \oplus r B] \setminus A$. Consequently, $x \in \Psi_{s} \big( [A \oplus r B] \setminus A \big)$.
		\item $r / 3 < \mathcal{d}(x, A) \leq 2 r / 3$. Then, $B_{s}[x] \subset [A \oplus r B] \setminus A$ and it follows that $x \in \Psi_{s} \big( [A \oplus r B] \setminus A \big)$.
		\item $2 r /3 < \mathcal{d}(x, A) \leq r $. Therefore, Propositions \ref{prop:trivial}d and \ref{prop:randspsi}d guarantee that $A \oplus r B = \Psi_{s} (A \oplus r B)$. So, there exists a point $y \in A \oplus r B$ such that $x \in B_{s} [y] \subset A \oplus r B$. In addition, every $z \in B_{s} [y]$ satisfies
		\begin{multline*}
			\mathcal{d}(x, A) \leq \mathcal{d}(x, y) + \mathcal{d}(y, z) + \mathcal{d}(z, A) \leq 2 s + \mathcal{d}(z, A) \Rightarrow \\
			\Rightarrow \mathcal{d}(z, A) \geq \mathcal{d}(x, A) - 2s > \frac{2}{3}r - 2s > 0.
		\end{multline*}
		It follows that $ B_{s} [y] \subset [A \oplus r B] \setminus A$, and so $x \in \Psi_{s} \big( [A \oplus r B] \setminus A \big)$. \qedhere
	\end{enumerate}
\end{proof}

Finally, the last lemma states the relation between Minkowski operations and set difference.

\begin{lemma}
	\label{lemma:difsum}
	Let $M$ be a connected and complete Riemannian manifold. Let $A \subset M$ be a closed set and $r, s > 0$. Then $$ \big[ (A \oplus r B) \setminus A \big] \oplus sB = \big[ A \oplus (r + s) B \big] \setminus [A \ominus sB].$$
\end{lemma}

\begin{proof}
	The first inclusion can be derived directly:
	\begin{align*}
		\big[ (A \oplus r B) \setminus A \big] \oplus sB =& \big[ (A \oplus r B) \cap A^c \big] \oplus sB \\
		\subset& \big[ A \oplus (r + s) B \big] \cap \big[ A^c \oplus s B \big] \\
		=& \big[ A \oplus (r + s) B \big] \cap \big[ A \ominus s B \big]^c \\
		=& \big[ A \oplus (r + s) B \big] \setminus \big[ A \ominus sB \big],
	\end{align*}
	where we have applied Proposition~\ref{prop:randsmin}a. Let $x \in \big[ A \oplus (r + s) B \big] \cap \big[ A^c \oplus sB \big]$. We distinguish three cases:
	\begin{enumerate}[label = Case \arabic*:, wide = \parindent, leftmargin = 0pt]
		\item $x \in A$. Since $x \in A^c \oplus sB$, there exists a point $y \in A^c$ such that $x \in B_s [y]$. If $y \in A \oplus r B$ the inclusion holds. Assume that $y \notin A \oplus r B$. Let $\alpha$ be a geodesic segment joining $x$ with $y$, $\alpha : [0, \mathcal{d} (x, y)] \rightarrow M$. Define
		\begin{equation*}
			t_1 = \max \Big\lbrace t \in \big[ 0, \mathcal{d} (x, y) \big] \colon \alpha (t) \in A \Big\rbrace.
		\end{equation*}
		By definition $\alpha ( t_1 ) \in A$. Since $y \notin A \oplus rB$, Proposition \ref{prop:dist}a ensures that 
		$$r < \mathcal{d} \big( \alpha (t_1), y \big) = \mathcal{d} (x, y) - t_1,$$
		or equivalently $r + t_1 < \mathcal{d} (x, y)$. Then, the point $\alpha (t_1 + r)$ satisfies $\alpha(t_1 + r) \in (A \oplus rB) \setminus A$ and $\mathcal{d} \big( x, \alpha(t_1 + r) \big) < \mathcal{d} (x, y) \leq s$, and so $x \in \big[ (A \oplus r B) \setminus A \big] \oplus sB$.
		\item $x \notin A$ and $x \in A \oplus rB$. It is straightforward that $x \in \big[ (A \oplus r B) \setminus A \big] \oplus sB$.
		\item $x \notin A \oplus rB$. Since $x \in (A \oplus r B) \oplus s B$ by Proposition~\ref{prop:randsmin}a, there exists a point $y \in A \oplus r B$ such that $x \in B_s [y]$. If $y \in A^c$ the result holds. Suppose that $y \in A$ and let $\alpha$ be a geodesic segment joining $x$ with $y$, $\alpha : \big[0, \mathcal{d} (x, y) \big] \rightarrow M$. Define
		\begin{equation*}
			t_2 = \min \Big\lbrace t \in \big[ 0, \mathcal{d} (x, y) \big] \colon \alpha ( t ) \in A \Big\rbrace.
		\end{equation*}
		By definition $\alpha ( t_2) \in A$. Since $x \notin A \oplus rB$, Proposition \ref{prop:dist}a shows that $$r < \mathcal{d} \big( x, \alpha (t_2) \big) = t_2.$$
		Then, $\alpha( t_2 - r) \in (A \oplus rB) \setminus A$ and 
		$$\mathcal{d} \big( x, \alpha( t_2 - r) \big) < \mathcal{d} (x, y) \leq s.$$
		So, $x \in \big[ (A \oplus r B) \setminus A \big] \oplus sB$. \qedhere
	\end{enumerate}
\end{proof}

\section{Consistency of $L_n (\lambda)$, $\hat{\lambda}_{\gamma, n}$ and $r_n (\lambda)$}
\label{sec:appB}

This appendix contains all the proofs related with the consistency of the estimators $L_n (\lambda)$, $\hat{\lambda}_{\gamma, n}$ and $r_n (\lambda)$. The appendix is organized in three sections. Section~\ref{sec:proofset} is fully devoted to the proof of Theorem~\ref{th:uniform}. Section~\ref{sec:lambdarate} contains the proofs of Theorems~\ref{th:probs} and~\ref{th:lambdaconvergence}. Finally, Section~\ref{sec:rnproof} contains the proof of Theorem~\ref{th:rconsistent}.

\subsection{Proof of Theorem~4.2 \label{sec:proofset}}

This section is fully devoted to the proof of Theorem~\ref{th:uniform}. The proof of this results relies on several auxiliary results that must be introduced. The basis of the HDR estimator $L_n(\lambda)$ defined in~\eqref{eq:setest} consists in using $\mathcal{X}^{+}_n (\lambda) \cap \big( \mathcal{X}^{-}_n (\lambda) \oplus r_n (\lambda) B \big)^c$ as a discretization of the set $L (\lambda) \ominus r_n (\lambda) B$, and then recovering $L (\lambda)$ by inflating this discretization. Lemma~\ref{lemma:covering} below studies the consistency of the last step of this procedure. If $\mathcal{X}_n$ is an i.i.d.~sample from a given density, will $(\mathcal{X}_n \cap A) \oplus r B$ eventually recover the set $A \oplus r B$? The answer is affirmative, assuming some general conditions on the set $A$.

\begin{lemma}
	\label{lemma:covering}
	Let $M$ be a Riemannian manifold under Assumptions M and $f:M \rightarrow \R$ a density function. Let $\mathcal{X}_n = \{ X_1, \ldots, X_n \}$ be an i.i.d.~sample from the density function $f$.
	
	Let $A \subset M$ a bounded set such that $f(x) \geq b > 0$ for all $x \in A$. Let $ r, \delta, \varepsilon > 0$ satisfy $\varepsilon < \delta \leq r$.
	\begin{enumerate}[label = (\alph*)]
		\item If $\Psi_{\delta} (A) = A$, then
		\begin{multline*}
			\Prob \big[ A \oplus (r - \varepsilon) B \not\subset (\mathcal{X}_n \cap A) \oplus r B \big] \leq D \left( \frac{\varepsilon}{2}, A \oplus r B \right) \times \\
			\times \exp \left[ -nab \min{\left\lbrace r - \frac{\varepsilon}{2}, \delta, \rho \right\rbrace}^{(d-1)/2} \min \left\lbrace \frac{\varepsilon}{2} , \rho \right\rbrace^{(d+1)/2} \right];
		\end{multline*}
		where $\rho$ and $a$ are the constants in Assumption (M2) and $D (\cdot, \cdot)$ is the functional defined in Proposition~\ref{prop:M2impliesM3}.
		\item The following inequality holds:
		\begin{multline*}
			\Prob \Big[ S \oplus (r - \varepsilon) B \not\subset (\mathcal{X}_n \cap S) \oplus r B  \text{ for some } S \subset A \text{ such that } \Psi_{\delta} (S) = S \Big]  \\
			\leq D \left( \frac{\varepsilon}{4}, A \oplus r B \right)^2 \exp \left[ -nab \min{\left\lbrace \frac{\delta}{2}, \rho \right\rbrace}^{(d-1)/2} \min \left\lbrace \frac{\varepsilon}{4} , \rho \right\rbrace^{(d+1)/2} \right];
		\end{multline*}
		where $\rho$ and $a$ are the constants in Assumption (M2) and $D (\cdot, \cdot)$ is the functional defined in Proposition~\ref{prop:M2impliesM3}.
	\end{enumerate}
\end{lemma}

	

\begin{proof}
	Let $T \subset A \oplus (r - \varepsilon) B $ be a subset such that
	\begin{equation*}
		\mathcal{d} (x, y) > \frac{\varepsilon}{2}, \quad \forall x, y \in T, x \neq y
	\end{equation*}
	with maximal cardinality. This implies that $ A \oplus (r - \varepsilon) B  \subset T \oplus \varepsilon/2 B $. Thus, given any subset $S \subset M$, it follows that
	\begin{equation*}
		T \subset S \Rightarrow A \oplus (r - \varepsilon) B \subset S \oplus \varepsilon/2 B.
	\end{equation*}
	
	Then, keeping in mind that $\big[ ( \mathcal{X}_n \cap A ) \oplus (r - \varepsilon/2) B \big] \oplus \varepsilon/2 B = ( \mathcal{X}_n \cap A ) \oplus r B$, one deduces
	\begin{align*}
		\Prob \big[ A \oplus (r - \varepsilon) B \not\subset ( \mathcal{X}_n \cap A ) & \oplus r B \big] \\
		\leq & \Prob \big[ T \not\subset ( \mathcal{X}_n \cap A ) \oplus (r - \varepsilon/2) B \big] \\
		\leq & \sum_{t \in T} \Prob \big[ t \notin ( \mathcal{X}_n \cap A ) \oplus (r - \varepsilon/2) B \big] \\
		\leq & D \big( \varepsilon/2, A \oplus r B \big) \sup_{t \in A \oplus (r - \varepsilon) B} \Prob \big( B_{r - \varepsilon/2} [t] \cap \mathcal{X}_n \cap A  = \emptyset \big);
	\end{align*}
	where the last inequality holds because $T \subset A \oplus (r - \varepsilon) B$. Then, showing that
	\begin{multline}
		\label{eq:ineqsup}
		\sup_{t \in A \oplus (r - \varepsilon) B} \Prob \big( B_{r - \varepsilon/2} [t] \cap \mathcal{X}_n \cap A  = \emptyset \big) \\
		\leq \exp \left[ -nab \min{\left\lbrace r - \frac{\varepsilon}{2}, \delta, \rho \right\rbrace}^{(d-1)/2} \min \left\lbrace \frac{\varepsilon}{2} , \rho \right\rbrace^{(d+1)/2} \right]
	\end{multline}
	will complete the proof of part (a).
	
	Fix a point $t \in A \oplus (r - \varepsilon) B$. It satisfies that
	\begin{equation*}
		\Prob \big( B_{r - \varepsilon/2} [t] \cap \mathcal{X}_n \cap A  = \emptyset \big)
		= \Big[ 1 - \Prob \big(X_1 \in A \cap B_{r - \varepsilon/2} [t] \big) \Big]^n.
	\end{equation*}
	Since $t \in A \oplus (r - \varepsilon) B$, there exists a point $x \in A$ such that $\mathcal{d} (x, t) \leq r - \varepsilon$. Furthermore, $\Psi_{\delta} (A) = A$, so there exists another point $y \in A$ that satisfies $x \in B_{\delta} [y] \subset A$. Therefore, \mbox{$\mathcal{d} (t, y) \leq \delta + r - \varepsilon$} and
	\begin{align*}
		\Prob \big( B_{r - \varepsilon/2} [t] \cap \mathcal{X}_n \cap A  = \emptyset \big)
		\leq & \Big[ 1 - \Prob \big( X_1 \in B_{\delta} [y] \cap B_{r - \varepsilon/2} [t] \big) \Big]^n \\
		\leq & \exp \Big[ - n \Prob \big( X_1 \in B_{\delta} [y] \cap B_{r - \varepsilon/2} [t] \big) \Big] \\
		\leq & \exp \Big[ - n b \mathrm{Vol}_M \big( B_{\delta} [y] \cap B_{r - \varepsilon/2} [t] \big) \Big].
	\end{align*}
	Following Assumption (M2), one deduces
	\begin{equation*}
		\Prob \big( B_{r - \varepsilon/2} [t] \cap \mathcal{X}_n \cap A  = \emptyset \big)
		\leq \exp \left[ -nab \min{\left\lbrace r - \frac{\varepsilon}{2}, \delta, \rho \right\rbrace}^{(d-1)/2} \min \left\lbrace \frac{\varepsilon}{2} , \rho \right\rbrace^{(d+1)/2} \right].
	\end{equation*}
	This inequality holds for all $t \in A \oplus (r - \varepsilon) B$, so \eqref{eq:ineqsup} holds too, concluding the proof of part (a).
	
	Denote $\mathcal{S}_{\delta} (A) = \{ S \subset A \colon \Psi_{\delta} (S) = S \}$. To prove part (b), the aim is to show:
	\begin{multline*}
		\Prob \left[ \bigcup_{S \in \mathcal{S}_{\delta} (A)} \big\lbrace S \oplus (r - \varepsilon) B \not\subset (\mathcal{X}_n \cap S) \oplus r B \big\rbrace \right] \leq  D \left( \frac{\varepsilon}{4}, A \oplus r B \right)^2 \times \\
		\times \exp \left[ -nab \min{\left\lbrace \frac{\delta}{2}, \rho \right\rbrace}^{(d-1)/2} \min \left\lbrace \frac{\varepsilon}{4} , \rho \right\rbrace^{(d+1)/2} \right].
	\end{multline*}
	The key fact is that one does not have to check the condition $$S \oplus (r - \varepsilon) B \not\subset (\mathcal{X}_n \cap S) \oplus r B$$ for all $S \in \mathcal{S}_{\delta} (A)$. It is enough to check it for a finite number of sets in $\mathcal{S}_{\delta/2} (A)$.
	
	Let $T \subset A \ominus \delta B$ such that
	\begin{equation*}
		\mathcal{d} (x, y) > \frac{\varepsilon}{4}, \quad \forall x, y \in T, x \neq y
	\end{equation*}
	with maximal cardinality. Consider the following two statements:
	\begin{enumerate}
		\item[(S)] \label{stat:inftysets} There exists a $S \in \mathcal{S}_{\delta} (A)$ such that $S \oplus (r - \varepsilon) B \not\subset (\mathcal{X}_n \cap S) \oplus r B$.
		\item[(T)] \label{stat:ftysets} There exists a $t \in T$ such that $B_{\delta - \varepsilon/4} [t] \oplus (r - \varepsilon/2) B \not\subset \big( \mathcal{X}_n \cap B_{\delta - \varepsilon/4} [t] \big) \oplus r B$.
	\end{enumerate}
	We are going to show that (S) implies (T).
	
	Suppose that there exists a set $S \in \mathcal{S}_{\delta} (A)$ such that
	\[S \oplus (r - \varepsilon) B \not\subset (\mathcal{X}_n \cap S) \oplus r B.\]
	Hence, there exists a point $x \in S \oplus (r - \varepsilon) B$ satisfying $x \not\in (\mathcal{X}_n \cap S) \oplus r B$. From $S \in \mathcal{S}_{\delta} (A)$ and Proposition~\ref{prop:randsmin}a, it follows that
	\[x \in S \oplus (r - \varepsilon) B = \big[ (S \ominus \delta B) \oplus \delta B \big] \oplus (r - \varepsilon) B = (S \ominus \delta B) \oplus (r + \delta - \varepsilon) B.\]
	So, there exists a point $s \in (S \ominus \delta B)$ such that $x \in B_{r + \delta - \varepsilon} [s]$. The definition of $T$ yields that $T \oplus \varepsilon/4 B \supset A \ominus \delta B \supset S \ominus \delta B$. Therefore, there exists $t \in T$ satisfying $\mathcal{d}(s, t) \leq \varepsilon/4$. We will show that this point $t$ satisfies the required condition in (T).
	
	On the one hand:
	\[\mathcal{d}(x, t) \leq \mathcal{d}(x, s) + \mathcal{d}(s, t) \leq r + \delta - \frac{3}{4} \varepsilon.\]
	Consequently, 
	\begin{equation}
		\label{eq:xin}
		x \in B_{r + \delta - 3 \varepsilon / 4} [t] = B_{\delta - \varepsilon / 4} [t] \oplus (r - \varepsilon / 2) B.
	\end{equation}
	On the other hand, $\mathcal{d}(s, t) \leq \varepsilon/4$ and the triangular inequality guarantee that
	$B_{\delta - \varepsilon/4} [t] \subset B_{\delta} [s] \subset S$
	where the last content follows from $s \in S \ominus \delta B$. Hence,
	\[\big( \mathcal{X}_n \cap B_{\delta - \varepsilon/4} [t] \big) \oplus r B \subset (\mathcal{X}_n \cap S) \oplus r B.\]
	and so
	\begin{equation}
		\label{eq:xnotin}
		x \not\in \big( \mathcal{X}_n \cap B_{\delta - \varepsilon/4} [t] \big) \oplus r B.
	\end{equation}
	Finally, from \eqref{eq:xin} and \eqref{eq:xnotin} one deduces
	\[ B_{\delta - \varepsilon / 4} [t] \oplus (r - \varepsilon / 2) B \not\subset \big( \mathcal{X}_n \cap B_{\delta - \varepsilon/4} [t] \big) \oplus r B \]
	and therefore (S) implies (T).
	
	Taking this into account, it follows that
	\begin{align*}
		\Prob \Bigg[ &\bigcup_{S \in \mathcal{S}_{\delta} (A)} \big\lbrace S \oplus (r - \varepsilon) B \not\subset (\mathcal{X}_n \cap S) \oplus r B \big\rbrace \Bigg]  \\
		&\leq \Prob \left[ \bigcup_{t \in T} \big\lbrace B_{\delta - \varepsilon / 4} [t] \oplus (r - \varepsilon / 2) B \not\subset \big( \mathcal{X}_n \cap B_{\delta - \varepsilon/4} [t] \big) \oplus r B \big\rbrace \right] \\
		&\leq \sum_{t \in T} \Prob \Big[ B_{\delta - \varepsilon / 4} [t] \oplus (r - \varepsilon / 2) B \not\subset \big( \mathcal{X}_n \cap B_{\delta - \varepsilon/4} [t] \big) \oplus r B \Big] \\
		&\leq D \left( \frac{\varepsilon}{4}, A \oplus r B \right) \sup_{t \in A \ominus \delta B} \Prob \Big[ B_{\delta - \varepsilon / 4} [t] \oplus (r - \varepsilon / 2) B \not\subset \big( \mathcal{X}_n \cap B_{\delta - \varepsilon/4} [t] \big) \oplus r B \Big],
	\end{align*}
	where the last inequality follows from $A \ominus \delta B \subset A \oplus r B$ and $T \subset A \ominus \delta B$. So, it only remains to show that 
	\begin{multline}
		\label{eq:lemma5sup}
		\sup_{t \in A \ominus \delta B} \Prob \Big[ B_{\delta - \varepsilon / 4} [t] \oplus (r - \varepsilon / 2) B \not\subset \big( \mathcal{X}_n \cap B_{\delta - \varepsilon/4} [t] \big) \oplus r B \Big] \\
		\leq D \left( \frac{\varepsilon}{4}, A \oplus r B \right) \exp \left[ -nab \min{\left\lbrace \frac{\delta}{2}, \rho \right\rbrace}^{(d-1)/2} \min \left\lbrace \frac{\varepsilon}{4} , \rho \right\rbrace^{(d+1)/2} \right]
	\end{multline}
	to conclude the proof of part (b).
	
	Let $t$ be an arbitrary point in $A \ominus \delta B$. One gets that:
	\begin{itemize}
		\item $B_{\delta - \varepsilon / 4} [t] \subset B_{\delta} [t] \subset A$, then $f(x) \geq b > 0$ for all $x \in B_{\delta - \varepsilon / 4} [t]$.
		\item Since $\varepsilon < \delta$, then $\delta / 2 < \delta - \varepsilon/4$ and Proposition \ref{prop:randspsi}a ensures that $\Psi_{\delta/2} \big( B_{\delta - \varepsilon / 4} [t] \big) = B_{\delta - \varepsilon / 4} [t]$.
	\end{itemize}
	
	Hence, part (a) guarantees that:
	\begin{multline*}
		\Prob \Big[ B_{\delta - \varepsilon / 4} [t] \oplus (r - \varepsilon / 2) B \not\subset \big( \mathcal{X}_n \cap B_{\delta - \varepsilon/4} [t] \big) \oplus r B \Big] \\
		\leq D \left( \frac{\varepsilon}{4}, B_{\delta - \varepsilon / 4} [t] \oplus r B \right) \exp \left[ -nab \min{\left\lbrace r - \frac{\varepsilon}{4}, \frac{\delta}{2}, \rho \right\rbrace}^{(d-1)/2} \min \left\lbrace \frac{\varepsilon}{4} , \rho \right\rbrace^{(d+1)/2} \right] \\
		\leq  D \left( \frac{\varepsilon}{4}, A \oplus r B \right) \exp \left[ -nab \min{\left\lbrace  \frac{\delta}{2}, \rho \right\rbrace}^{(d-1)/2} \min \left\lbrace \frac{\varepsilon}{4} , \rho \right\rbrace^{(d+1)/2} \right].
	\end{multline*}
	Since this bound holds for every $t \in A \ominus \delta B$, \eqref{eq:lemma5sup} also holds, finishing the proof.
\end{proof}

Three conditions are required on the set $A$ for Lemma~\ref{lemma:covering}a to hold. The first one is a technical detail: $A$ is needed to be bounded. This ensures that $D \left( \varepsilon / 2, A \oplus r B \right) < + \infty$ and therefore the result is nontrivial.

The second condition is that there exists some positive constant $b$ such that $f \geq b$ in~$A$. It could be interesting to find an analogous result without this restriction, allowing $f$ to vanish in a small subset of $A$ (in terms of measure). But we focus just on HDR estimation, so this assumption does not implies an important limitation and no further exploration will be done in that sense.

The last condition is the cornerstone of the result. Assuming that $\Psi_{\delta} (A) = A$ for some $\delta > 0$ implies that $A$ is a set with a non-zero measure and so $\Prob (A) \neq 0$. Moreover, this condition forces the set to be smooth, without any sharp edge or too narrow part. This fact combined with Assumption (M2) guarantees that the discretization $\mathcal{X}_n \cap A$ will eventually be fine enough to recover the set $A \oplus r B$ with the required convergence rate.

As noticed before, the proposed estimator recovers $L(\lambda)$ by inflating a discretization of the set $L(\lambda) \ominus r_n (\lambda) B$. However, it should be noted that, since $r_n(\lambda)$ is a random variable, $L(\lambda) \ominus r_n (\lambda) B$ is a random set. Lemma~\ref{lemma:covering}a assumes that $A$ is a fixed set, and therefore, one cannot apply this result directly to show the consistency of $L_n (\lambda)$. Lemma~\ref{lemma:covering}b offers a solution for this issue, transforming the result of Lemma~\ref{lemma:covering}a into a uniform one.

A direct consequence of Lemma~\ref{lemma:covering} is that $\mathcal{d}_H \big( (\mathcal{X}_n \cap S) \oplus r B,  S \oplus r B \big)$ converges to zero uniformly with respect to $S$.

\begin{cor}
	\label{cor:1}
	Let $M$ be a Riemannian manifold under Assumptions M and $f:M \rightarrow \R$ a density function. Let $\mathcal{X}_n = \{ X_1, \ldots, X_n \}$ be an i.i.d.~sample from the density function $f$.
	
	Let $A \subset M$ a bounded set such that $f(x) \geq b > 0$ for all $x \in A$, and let $r, \delta > 0$ such that $\delta \leq r$. Then,
	\begin{equation*}
		\sup_{S \in \mathcal{S}_{\delta} (A)} \mathcal{d}_H \big( (\mathcal{X}_n \cap S) \oplus rB, S \oplus r B)
		=
		O_{\textrm{\rm a.s.}} \Bigg( \bigg( \frac{\log n}{n} \bigg)^{2/(d+1)} \Bigg),
	\end{equation*}
	where $\mathcal{S}_{\delta} (A) = \{ S \subset A \colon \Psi_{\delta} (S) = S \}$.
\end{cor}

\begin{proof}
	First, take into account that, for any $\varepsilon > 0$,
	\begin{align*}
		(\mathcal{X}_n \cap S) \oplus (r - \varepsilon) B \subset S \oplus r B
		&\Rightarrow
		(\mathcal{X}_n \cap S) \oplus r B \subset S \oplus (r + \varepsilon) B
		\\
		&\Rightarrow
		\mathcal{d}_H \big( (\mathcal{X}_n \cap S) \oplus rB, S \oplus r B) < \varepsilon.
	\end{align*}
	So, 
	\begin{multline*}
		\Prob \Big[ \sup_{S \in \mathcal{S}_{\delta} (A)} \mathcal{d}_H \big( (\mathcal{X}_n \cap S) \oplus rB, S \oplus r B) > \varepsilon \Big]
		\\
		\leq
		\Prob \Big[ S \oplus (r - \varepsilon) B \not\subset (\mathcal{X}_n \cap S) \oplus r B  \text{ for some } S \subset A \text{ such that } \Psi_{\delta} (S) = S \Big].
	\end{multline*}
	
	Now, let
	\begin{equation*}
		\varepsilon_n = \bigg( \frac{c \log n}{n} \bigg)^{2/(d+1)}
	\end{equation*}
	where $c > 0$ is a constant that will be specified later on. Since $\varepsilon_n \to 0$, $\varepsilon_n < \min \{ \delta, \rho \}$ for large $n$, where $\rho > 0$ is the constant present on Assumption (M2). If so, Lemma~\ref{lemma:covering}b ensures that
	\begin{multline*}
		\Prob \Big[ \sup_{S \in \mathcal{S}_{\delta} (A)} \mathcal{d}_H \big( (\mathcal{X}_n \cap S) \oplus rB, S \oplus r B) > \varepsilon_n \Big]
		\\
		\leq
		D \left( \frac{\varepsilon_n}{4}, A \right)^2 \exp \left[ - nab \min \bigg\lbrace \frac{\delta}{2}, \rho \bigg\rbrace^{(d-1)/2} \bigg\lbrace \frac{\varepsilon_n}{4} \bigg\rbrace^{(d+1)/2} \right]
		\\
		=
		O
		\bigg(
		\varepsilon_n^{-2d}
		\exp \left[ - C n \varepsilon_n^{(d+1)/2} \right]
		\bigg),
	\end{multline*}
	where $C = ab \min \lbrace \frac{\delta}{2}, \rho \rbrace^{(d-1)/2} 4^{-(d+1)/2}$ and we have used Proposition~\ref{prop:M2impliesM3}. Furthermore,
	\begin{equation*}
		\varepsilon_n^{-2d}
		\exp \left[ - C n \varepsilon_n^{(d+1)/2} \right]
		=
		o
		\big(
		n^{4d/(d+1) - cC}
		\big)
	\end{equation*}
	Taking $c$ large enough so that $$4d/(d+1) - cC < -2$$ guarantees that
	\begin{equation*}
		\Prob \Big[ \sup_{S \in \mathcal{S}_{\delta} (A)} \mathcal{d}_H \big( (\mathcal{X}_n \cap S) \oplus rB, S \oplus r B \big) > \varepsilon_n \Big]
		=
		o
		\big(
		n^{-2}
		\big).
	\end{equation*}
	Therefore,
	\begin{equation*}
		\sum_{n = 1}^{+ \infty}\Prob \Big[ \sup_{S \in \mathcal{S}_{\delta} (A)} \mathcal{d}_H \big( (\mathcal{X}_n \cap S) \oplus rB, S \oplus r B \big) > \varepsilon_n \Big]
		< + \infty
	\end{equation*}
	and the result holds by the Borel-Cantelli Lemma.
\end{proof}

The following lemma guarantees that $\mathcal{X}_n^{+} (\lambda)$ and $\mathcal{X}_n^{-} (\lambda)$ will eventually be equal to \mbox{$\mathcal{X}_n \cap  L(\lambda)$} and $\mathcal{X}_n \cap  L(\lambda)^{c}$, respectively. From now on and until the end of Section~\ref{sec:lambdarate}, $l, u, \delta$ and $k$ are the same constants as in Assumptions~L.

\begin{lemma}
	\label{lemma:X+andX-}
	Under the hypotheses of Theorem~4.2, if $D_n < \delta$ and $k D_n < \delta$, the following four inclusions hold for all $\lambda \in [l, u]$:
	\begin{align}
		\label{eq:X+large}&\mathcal{X}_n \cap \big( L(\lambda) \ominus k D_n B \big) \subset \mathcal{X}_n^{+} (\lambda); \\
		\label{eq:X-large}&\mathcal{X}_n \cap \big( L(\lambda)^c \ominus k D_n B \big) \subset \mathcal{X}_n^{-} (\lambda); \\
		\label{eq:X+small} &\mathcal{X}_n^{+} (\lambda) \subset L(\lambda) \oplus k D_n B; \\
		\label{eq:X-small} &\mathcal{X}_n^{-} (\lambda) \subset L(\lambda)^c \oplus k D_n B.
	\end{align}
\end{lemma}

\begin{proof}
	The key is dealing with the error made when $f$ is estimated by $f_n$. Given any sample point, $X_i \in \mathcal{X}_n$, one derives that
	\begin{align*}
		X_i \in L(\lambda + D_n) \Rightarrow f(X_i) \geq \lambda + D_n \Rightarrow \lambda \leq f(X_i) - D_n \leq f_n(X_i) \Rightarrow X_i \in \mathcal{X}_n^{+} (\lambda), \\
		X_i \notin L(\lambda - D_n) \Rightarrow f(X_i) < \lambda - D_n \Rightarrow \lambda > f(X_i) + D_n \geq f_n(X_i) \Rightarrow X_i \in \mathcal{X}_n^{-} (\lambda).
	\end{align*}
	Consequently $\mathcal{X}_n \cap L(\lambda + D_n) \subset \mathcal{X}_n^{+}$ and $\mathcal{X}_n \cap L(\lambda - D_n)^c \subset \mathcal{X}_n^{-}$.
	
	Since $D_n < \delta$ and $k D_n < \delta$, one deduces from Assumption (L4) that
	\begin{align*}
		\mathcal{d}_H \big( L(\lambda), L(\lambda - D_n) \big) < k D_n &\Rightarrow L (\lambda) \oplus k D_n B \supset L(\lambda - D_n) \Rightarrow \\
		&\Rightarrow L (\lambda)^{c} \ominus k D_n B \subset L(\lambda - D_n)^c \\
		\mathcal{d}_H \big( L(\lambda), L(\lambda + D_n) \big) < k D_n &\Rightarrow L(\lambda) \subset L(\lambda + D_n) \oplus k D_n B \Rightarrow \\
		&\Rightarrow L(\lambda) \ominus k D_n B \subset \big[ L(\lambda + D_n) \oplus k D_n B \big] \ominus k D_n B = \\
		& \qquad = \Big[ \Psi_{k D_n} \big( L(\lambda + D_n)^c \big) \Big]^c = L(\lambda + D_n);
	\end{align*}
	when the last equality follows from Assumption (L2). All the above implies that
	\begin{align*}
		&X_i \in L(\lambda) \ominus k D_n B \Rightarrow X_i \in \mathcal{X}_n^{+} (\lambda), \\
		&X_i \in L(\lambda)^{c} \ominus k D_n B \Rightarrow X_i \in \mathcal{X}_n^{-} (\lambda);
	\end{align*}
	for all $X_i \in \mathcal{X}_n$. The first two inclusions follow directly from these ones:
	\begin{align*}
		&\mathcal{X}_n \cap \big( L(\lambda) \ominus k D_n B \big) \subset \mathcal{X}_n^{+} (\lambda); \\
		&\mathcal{X}_n \cap \big( L(\lambda)^c \ominus k D_n B \big) \subset \mathcal{X}_n^{-} (\lambda).
	\end{align*}
	The other two inclusions are derived from these two taking complements:
	\begin{align*}
		&\mathcal{X}_n^{+} (\lambda) = \mathcal{X}_n \setminus \mathcal{X}_n^{-} (\lambda) \subset \mathcal{X}_n \setminus \big( L(\lambda)^c \ominus k D_n B \big) \subset L(\lambda) \oplus k D_n B; \\
		&\mathcal{X}_n^{-} (\lambda) = \mathcal{X}_n \setminus \mathcal{X}_n^{+} (\lambda) \subset \mathcal{X}_n \setminus \big( L(\lambda) \ominus k D_n B \big) \subset L(\lambda)^c \oplus k D_n B. \qedhere
	\end{align*}
\end{proof}

Lastly, the proof of Theorem~\ref{th:uniform} ultimately relies on Lemma~\ref{lemma:transU} below.

\begin{lemma}
	\label{lemma:transU}
	Under the hypotheses of Theorem~\ref{th:uniform}, let
	\begin{equation*}
		\xi = \min_{x \in L(l)} \mathcal{d} \big( x, L(l/2)^c \big) = d \big( L(l), L (l/2)^{c} \big);
	\end{equation*}
	which exists since $L(l)$ is compact. Consider the random sets
	\[S_U (\lambda) = \big( L(\lambda) \oplus \xi B \big) \setminus \big( L(\lambda) \oplus k D_n B \big) \]
	and take
	\begin{equation*}
		\Delta_n = \sup_{\lambda \in [l,u]} \mathcal{d}_H \Big( S_U (\lambda) \oplus r_n (\lambda)/2 B , \big( \mathcal{X}_n \cap S_U (\lambda) \big) \oplus r_n (\lambda)/2 B \Big).
	\end{equation*}
	Then, $\xi > 0$ and 
	\begin{equation*}
		\Delta_n
		=
		O_{\textrm{\rm a.s.}} \bigg( \left( \frac{\log n}{n} \right)^{\frac{2}{d+1}} \bigg).
	\end{equation*}
\end{lemma}

\begin{proof}
	First, let us show that $\xi > 0$. By compactness, there exists $x_{\mathrm{min}} \in L(l)$ such that
	\begin{equation*}
		\mathcal{d} \big( x_{\mathrm{min}}, L(l/2)^c \big) = \min_{x \in L(l)} \mathcal{d} \big( x, L(l/2)^c \big) = \xi.
	\end{equation*}
	Since $f$ is continuous at $x_{\mathrm{min}} \in L(l)$, there exists $\epsilon > 0$ such that every $y \in B_{\epsilon} (x_{\mathrm{min}})$ satisfies:
	\begin{equation*}
		\frac{l}{2} - f(x_{\mathrm{min}}) < f(y) - f(x_{\mathrm{min}}) \Rightarrow \frac{l}{2} < f(y).
	\end{equation*}
	Therefore $B_{\epsilon} (x_{\mathrm{min}}) \subset L(l/2)$ and so $\xi = \mathcal{d} \big( x_{\mathrm{min}}, L(l/2)^c \big) \geq \epsilon > 0$. Note that from the definition of $\xi$ it follows that
	\begin{equation*}
		L(\lambda) \oplus \xi B \subset L(l) \oplus \xi B \subset L \left( \frac{l}{2} \right)
	\end{equation*}
	holds for all $\lambda \in [l, u]$.
	
	Assume now that $kD_n < \min \{ \delta/2, \xi/2 \} $ and $r_n (\lambda) > \delta$ for all $\lambda \in [l, u]$. Then, Lemma~\ref{lemma:psidifference} and Assumption (L2) ensure that 
	\begin{equation*}
		(L(\lambda) \oplus k D_n B)^c = \Psi_{\delta/2}  \big( (L(\lambda) \oplus k D_n B)^c \big).
	\end{equation*}
	This and Lemma~\ref{lemma:psiminusset} guarantee that, for $c = \min \{ \delta/2, \xi/6 \}$,
	\begin{equation*}
		\Psi_{c}  \big( S_U (\lambda) \big) = S_U (\lambda).
	\end{equation*}
	Furthermore, $S_U (\lambda) \subset L (l/2)$. Then, if $kD_n < \min \{ \delta/2, \xi/2 \} $ and $r_n (\lambda) > \delta$ for all $\lambda \in [l, u]$, one has
	\begin{align*}
		\Delta_n 
		\leq&
		\sup_{S \in \mathcal{S}_{c} (L (l/2))}
		\mathcal{d}_H
		\Big( S \oplus r_n (\lambda)/2 B , \big( \mathcal{X}_n \cap S \big) \oplus r_n (\lambda)/2 B \Big)
		\\
		\leq&
		\sup_{S \in \mathcal{S}_{c} (L (l/2))}
		\mathcal{d}_H
		\Big( S \oplus\delta/2 B , \big( \mathcal{X}_n \cap S \big) \oplus \delta/2 B \Big),
	\end{align*}
	where we have applied Propositions~\ref{prop:randsmin}a and~\ref{prop:distdecr}b. Since $kD_n < \min \{ \delta/2, \xi/2 \} $ and $r_n (\lambda) > \delta$ for all $\lambda \in [l, u]$ happen e.a.s., it follows from Corollary~\ref{cor:1} that
	\begin{equation*}
		\Delta_n
		=
		O_{\textrm{a.s.}} \bigg( \left( \frac{\log n}{n} \right)^{\frac{2}{d+1}} \bigg).
		\qedhere
	\end{equation*}
\end{proof}

We can finally proof Theorem~\ref{th:uniform}.

\begin{proof}[Proof of Theorem~\ref{th:uniform}]
	Notice that
	\begin{equation*}
		\sup_{\lambda \in [l,u]} \mathcal{d}_H \big( L(\lambda), L_n(\lambda) \big)
		\leq 
		\max \left\lbrace
		\sup_{\lambda \in [l,u]} \overrightarrow{\mathcal{d}_H} \big( L(\lambda), L_n(\lambda) \big),
		\sup_{\lambda \in [l,u]} \overrightarrow{\mathcal{d}_H} \big( L_n(\lambda), L(\lambda) \big)
		\right\rbrace,
	\end{equation*}
	where $\overrightarrow{\mathcal{d}_H}$ is upper Hausdorff distance defined in~\eqref{eq:uphaus}. We will prove the result showing that both $\sup_{\lambda \in [l,u]} \overrightarrow{\mathcal{d}_H} \big( L(\lambda), L_n(\lambda) \big)$ and $\sup_{\lambda \in [l,u]} \overrightarrow{\mathcal{d}_H} \big( L(\lambda), L_n(\lambda) \big)$ converge to zero at the required rate.
	
	Fix $\lambda \in [l, u]$, and take $S_L (\lambda) = L(\lambda) \ominus (r_n (\lambda) + k D_n) B$. Assume that $D_n < \min \{ \delta, \delta/2k \}$ and $r_n (\lambda) \in (\delta, r_0 (\lambda) - \delta)$. Lemma~\ref{lemma:X+andX-} yields
	\begin{equation*}
		\mathcal{X}_n \cap S_L (\lambda) \subset \mathcal{X}_n^{+} (\lambda) \cap \big( \mathcal{X}_n^{-} (\lambda) \oplus r_n (\lambda) B \big)^c.
	\end{equation*}
	Then,
	\begin{align}
		\nonumber
		\overrightarrow{\mathcal{d}_H} \big( L(\lambda), L_n (\lambda) \big)
		\leq&
		\overrightarrow{\mathcal{d}_H} \Big( L(\lambda), \big( \mathcal{X}_n \cap S_L (\lambda) \big) \oplus r_n (\lambda) B \Big)
		\\
		\nonumber
		\leq&
		\overrightarrow{\mathcal{d}_H} \Big( L(\lambda), S_L (\lambda) \oplus r_n (\lambda) B \Big)
		\\
		\label{eq:lower1}
		&+
		\overrightarrow{\mathcal{d}_H} \Big( S_L (\lambda) \oplus r_n (\lambda) B , \big( \mathcal{X}_n \cap S_L (\lambda) \big) \oplus r_n (\lambda) B \Big),
	\end{align}
	Assumption (L1) and Proposition~\ref{prop:randspsi}d ensure that $S_L (\lambda) \oplus (r_n (\lambda) + kD_n) B = L(\lambda)$. So,
	\begin{equation}
		\label{eq:lower2}
		\overrightarrow{\mathcal{d}_H} \big( L(\lambda), S_L (\lambda) \oplus r_n (\lambda) B \big)
		\leq
		k D_n.
	\end{equation}
	In addition, $S_L (\lambda) \subset L (l)$ and $\Psi_{\delta/2} ( S_L (\lambda) ) = S_L (\lambda)$ by Lemma~\ref{lemma:psidifference} and Assumptions (L1) and (L2). Then,
	\begin{multline}
		\label{eq:lower3}
		\overrightarrow{\mathcal{d}_H} \Big( S_L (\lambda) \oplus r_n (\lambda) B , \big( \mathcal{X}_n \cap S_L (\lambda) \big) \oplus r_n (\lambda) B \Big)
		\\
		\leq
		\sup_{S \in \mathcal{S}_{\delta/2} (L (l) )} \mathcal{d}_H \Big( S \oplus r_n (\lambda) B , \big( \mathcal{X}_n \cap S \big) \oplus r_n (\lambda) B \Big)
		\\
		\leq
		\sup_{S \in \mathcal{S}_{\delta/2} (L (l) )} \mathcal{d}_H \Big( S \oplus \delta B , \big( \mathcal{X}_n \cap S \big) \oplus \delta B \Big)
	\end{multline}
	where in the last line we have applied Propositions~\ref{prop:randsmin}a and~\ref{prop:distdecr}b.
	
	Joining \eqref{eq:lower1}, \eqref{eq:lower2} and \eqref{eq:lower3} leads to
	\begin{equation*}
		\sup_{\lambda \in [l,u]} \overrightarrow{\mathcal{d}_H} \big( L(\lambda), L_n (\lambda) \big)
		\leq
		\sup_{S \in \mathcal{S}_{\delta/2} (L (l) )} \mathcal{d}_H \Big( S \oplus \delta B , \big( \mathcal{X}_n \cap S \big) \oplus \delta B \Big)
		+
		k D_n,
	\end{equation*}
	which holds if $r_n (\lambda) \in (\delta, r_0 (\lambda) - \delta)$ for all $\lambda \in [l, u]$ and $D_n < \min \{ \delta, \delta/2k \}$. Under the assumptions of Theorem~\ref{th:uniform} these two events happen eventually almost surely. This fact and Corollary~\ref{cor:1} guarantee that
	\begin{equation}
		\label{eq:lowerbound}
		\sup_{\lambda \in [l,u]} \overrightarrow{\mathcal{d}_H} \big( L(\lambda), L_n (\lambda) \big)
		=
		O_\textrm{a.s.} \left( \max \left\lbrace D_n, \left( \frac{\log n}{n} \right)^{\frac{2}{d+1}} \right\rbrace \right).
	\end{equation}
	
	Now, fix $\lambda \in [l, u]$ and take $S_U (\lambda)$ and $\Delta_n$ as Lemma~\ref{lemma:transU}. Assume that these conditions hold: $\Delta_n < \delta/2$, $r_n (\lambda) \in ( \delta, r_0 (\lambda) - \delta)$ and \mbox{$D_n \leq \min \{ \delta, \delta/2k \}$}. It follows from the definition of $\Delta_n$ (see Lemma~\ref{lemma:transU}) that
	\begin{multline*}
		S_U (\lambda) \oplus r_n (\lambda)/2 B \subset \big( \mathcal{X}_n \cap S_U (\lambda) \big) \oplus (r_n (\lambda)/2 + \Delta_n) B
		\Rightarrow
		\\
		\Rightarrow
		S_U (\lambda) \oplus (r_n (\lambda) - \Delta_n) B \subset \big( \mathcal{X}_n \cap S_U (\lambda) \big) \oplus r_n (\lambda) B.
	\end{multline*}
	Lemma~\ref{lemma:X+andX-} guarantees that $ \mathcal{X}_n \cap S_U (\lambda) \subset \mathcal{X}_n^{-} (\lambda)$. So
	\begin{multline*}
		S_U (\lambda) \oplus (r_n (\lambda) - \Delta_n) B \subset \mathcal{X}_n^{-} (\lambda) \oplus r_n (\lambda) B
		\Rightarrow
		\\
		\Rightarrow
		\big( \mathcal{X}_n^{-} (\lambda) \oplus r_n (\lambda) B \big)^c
		\subset
		( S_U (\lambda) \oplus (r_n (\lambda) - \Delta_n) B )^c.
	\end{multline*}
	In addition, Lemmas~\ref{lemma:psidifference} and~\ref{lemma:difsum} ensure that 
	\begin{multline*}
		S_U (\lambda) \oplus (r_n (\lambda) - \Delta_n) B \\
		=
		\big( L(\lambda) \oplus ( \xi + r_n (\lambda) - \Delta_n ) B \big) \setminus \big( L(\lambda) \ominus (r_n (\lambda) - \Delta_n - k D_n) B \big).
	\end{multline*}
	Hence, if follows from Lemma~\ref{lemma:X+andX-} that
	\begin{multline*}
		\mathcal{X}_n^{+} (\lambda) \cap \big( \mathcal{X}_n^{-} (\lambda) \oplus r_n (\lambda) B \big)^c
		\\
		\subset
		\Big[ L (\lambda) \oplus kD_n B \Big] \cap
		\Big[
		\big( L(\lambda)^c \ominus ( \xi + r_n (\lambda) - \Delta_n ) B \big) \cup \big( L(\lambda) \ominus (r_n (\lambda) - \Delta_n - k D_n) B \big)
		\Big]
		\\
		=
		L(\lambda) \ominus (r_n (\lambda) - \Delta_n - k D_n) B,
	\end{multline*}
	where the last equality follows from $\xi + r_n (\lambda) - \Delta_n > kD_n$. Consequently,
	\begin{align*}
		L_n (\lambda)
		=&
		\Big[ \mathcal{X}_n^{+} (\lambda) \cap \big( \mathcal{X}_n^{-} (\lambda) \oplus r_n (\lambda) B \big)^c \Big]
		\oplus
		r_n (\lambda) B
		\\
		\subset&
		\big( L(\lambda) \ominus (r_n (\lambda) - \Delta_n - k D_n) B \big) \oplus r_n (\lambda) B
		\\
		=&
		\Psi_{r_n (\lambda) - \Delta_n - k D_n} \big( L(\lambda) \big) \oplus (\Delta_n + k D_n) B
		\\
		=&
		L(\lambda) \oplus (\Delta_n + k D_n) B
	\end{align*}
	where we have applied Proposition~\ref{prop:randsmin}a and Lemma~\ref{lemma:psidifference}. This holds for any $\lambda \in [l, u]$, so
	\begin{equation*}
		\sup_{\lambda \in [l,u]}
		\overrightarrow{\mathcal{d}_H} \big( L_n(\lambda), L (\lambda) \big)
		\leq
		\Delta_n + k D_n
	\end{equation*}
	also holds if $\Delta_n < \delta/2$, $D_n \leq \min \{ \delta, \delta/2k \}$ and $r_n (\lambda) \in ( \delta, r_0(\lambda) - \delta)$ for any $\lambda \in [l, u]$. Lemma~\ref{lemma:transU} and the assumptions of Theorem~\ref{th:uniform} guarantee that these three events happen eventually almost surely. This result and Lemma~\ref{lemma:transU} again yield
	\begin{equation}
		\label{eq:upperbound}
		\sup_{\lambda \in [l,u]} \overrightarrow{\mathcal{d}_H} \big( L_n(\lambda), L (\lambda) \big)
		=
		O_{\textrm{a.s.}} \left( \max \left\lbrace D_n, \left( \frac{\log n}{n} \right)^{\frac{2}{d+1}} \right\rbrace \right).
	\end{equation}
	Since both \eqref{eq:lowerbound} and \eqref{eq:upperbound} hold, the result holds.
\end{proof}

\subsection{Proofs for Section~5}
\label{sec:lambdarate}

Theorem~\ref{th:probs} is a direct consequence of Theorem~\ref{th:uniform} 

\begin{proof}[Proof of Theorem~\ref{th:probs}]
	From the triangular inequality one deduces that
	\begin{equation*}
		\sup_{\gamma \in [\underline{\gamma}, \overline{\gamma}]} \mathcal{d}_H \big( L(\lambda_{\gamma}), L_n(\tilde{\lambda}_{\gamma, n}) \big) \leq \sup_{\gamma \in [\underline{\gamma}, \overline{\gamma}]} \mathcal{d}_H \big( L(\lambda_{\gamma}), L(\tilde{\lambda}_{\gamma, n}) \big) +
		\sup_{\gamma \in [\underline{\gamma}, \overline{\gamma}]} \mathcal{d}_H \big( L(\tilde{\lambda}_{\gamma, n}), L_n(\tilde{\lambda}_{\gamma, n}) \big).
	\end{equation*}
	
	If $T_n < \delta$, then Assumption (L4) implies
	\begin{equation*}
		\sup_{\gamma \in [\underline{\gamma}, \overline{\gamma}]} \mathcal{d}_H \big( L(\lambda_{\gamma}), L(\tilde{\lambda}_{\gamma, n}) \big) \leq k T_n.
	\end{equation*}
	
	On the other hand, if $T_n < \min \{u - \lambda_{\overline{\gamma}}, \lambda_{\underline{\gamma}} - l \}$, then $\tilde{\lambda}_{\gamma, n} \in [l, u]$ for all $\gamma \in [\underline{\gamma}, \overline{\gamma}]$ and therefore
	\begin{equation*}
		\sup_{\gamma \in [\underline{\gamma}, \overline{\gamma}]} \mathcal{d}_H \big( L(\tilde{\lambda}_{\gamma, n}), L_n(\tilde{\lambda}_{\gamma, n}) \big) \leq \sup_{\lambda \in [l, u]} \mathcal{d}_H \big( L(\lambda), L_n(\lambda) \big).
	\end{equation*}
	
	Since $T_n \rightarrow 0$ almost surely, $T_n < \min \{ u - \lambda_{\overline{\gamma}}, \lambda_{\underline{\gamma}} - l, \delta \}$ happens e.a.s. Consequently, all the above jointly with Theorem~\ref{th:uniform} guarantee that there exists a constant $C > 0$ such that, e.a.s., it holds that
	\begin{align*}
		\sup_{\gamma \in [\underline{\gamma}, \overline{\gamma}]} \mathcal{d}_H \big( L(\lambda_{\gamma}), L_n(\tilde{\lambda}_{\gamma, n}) \big) \leq& k T_n + C \max \left\lbrace D_n, \left( \frac{\log n}{n} \right)^{\frac{2}{d+1}} \right\rbrace \\
		\leq& (k + C) \max \left\lbrace D_n, T_n, \left( \frac{\log n}{n} \right)^{\frac{2}{d+1}} \right\rbrace. \qedhere
	\end{align*}
\end{proof}

The proof of Theorem~\ref{th:lambdaconvergence} requires Propositions~\ref{prop:levymetric} and~\ref{prop:convergenceprob} below. The first of them allows to bound the distance between $\lambda_{\gamma}$ and $\hat{\lambda}_{\gamma , n}$ in a special type of metric called Lévy metric (check \citealp{Cadre2013} for more details about this metric).

\begin{prop}
	\label{prop:levymetric}
	Let $M$ be a Riemannian manifold and $f$ a density function on $M$. Let $X_1, \ldots, X_n$ be a random sample from $f$. Let $f_n$ be an estimator of $f$.
	
	Consider the random variables:
	\begin{equation*}
		D_n = \sup_{x \in M} \abs{f(x) - f_n(x)}, \qquad
		S_n = \sup_{\lambda \in \R} \abs{\Prob \big[ L (\lambda) \big] - \Prob_n \big[ L (\lambda) \big] }.
	\end{equation*}
	
	Hence:
	\begin{equation*}
		\lambda_{\gamma - M_n} - M_n \leq \hat{\lambda}_{\gamma , n} \leq \lambda_{\gamma + M_n} + M_n,
	\end{equation*}
	for all $\gamma \in (0, 1)$, where $M_n = \max \{D_n, S_n\}$.
\end{prop}

\begin{proof}
	The definition of $D_n$ directly yields
	\begin{equation*}
		L ( \lambda + D_n ) 
		\subset f_n^{-1} \big( [\lambda, + \infty) \big) \subset
		L ( \lambda - D_n ),
	\end{equation*}
	for every $\lambda \in \R$.
	So,
	\begin{equation}
		\label{eq:contentPn}
		\Prob_n \big[ L ( \lambda + D_n ) \big] \leq \Prob_n \big[ f_n^{-1} \big( [\lambda, + \infty) \big) \big] \leq \Prob_n \big[ L ( \lambda - D_n ) \big], \quad \forall \lambda \in \R.
	\end{equation}
	On the other hand, from the definition of $S_n$, it follows that
	\begin{equation}
		\label{eq:approxSn}
		\Prob \big[ L ( \lambda ) \big] - S_n \leq \Prob_n \big[ L ( \lambda ) \big] \leq \Prob \big[ L ( \lambda )  \big] + S_n, \quad \forall \lambda \in \R.
	\end{equation}
	Consequently, \eqref{eq:contentPn} and \eqref{eq:approxSn} guarantee that
	\begin{equation*}
		\Prob \big[ L ( \lambda + D_n ) \big] - S_n \leq \Prob_n \big[ f_n^{-1} \big( [\lambda, + \infty) \big) \big] \leq \Prob \big[ L ( \lambda - D_n ) \big] + S_n , \quad \forall \lambda \in \R.
	\end{equation*}
	Taking into account that $\Prob \big[ L ( \lambda ) \big]$ is non-increasing in $\lambda$, this yields
	\begin{equation*}
		\Prob \big[ L ( \lambda + M_n ) \big] - M_n \leq \Prob_n \big[ f_n^{-1} \big( [\lambda, + \infty) \big) \big] \leq \Prob \big[ L ( \lambda - M_n ) \big] + M_n , \quad \forall \lambda \in \R.
	\end{equation*}
	
	From the previous inequality one deduces that the following two contents
	\begin{align*}
		\left\lbrace \lambda \in \R \colon \Prob_n \big[ f_n^{-1} \big( [\lambda, + \infty) \big) \big] \geq 1 - \gamma  \right\rbrace \subset &
		\left\lbrace \lambda \in \R \colon \Prob \big[ L ( \lambda - M_n ) \big] \geq 1 - \gamma - M_n \right\rbrace, \\
		\left\lbrace \lambda \in \R \colon \Prob_n \big[ f_n^{-1} \big( [\lambda, + \infty) \big) \big] \geq 1 - \gamma  \right\rbrace \supset &
		\left\lbrace \lambda \in \R \colon \Prob \big[ L ( \lambda + M_n ) \big] \geq 1 - \gamma + M_n \right\rbrace ;
	\end{align*}
	hold for all $\gamma \in (0,1)$. Hence,
	\begin{align*}
		\hat{\lambda}_{\gamma , n} =& \sup \left\lbrace \lambda \in \R \colon \Prob_n \big[ f_n^{-1} \big( [\lambda, + \infty) \big) \big] \geq 1 - \gamma  \right\rbrace \\
		\leq& \sup \left\lbrace \lambda \in \R \colon \Prob \big[ L ( \lambda - M_n ) \big] \geq 1 - \gamma - M_n \right\rbrace \\
		=& \sup \left\lbrace \lambda \in \R \colon \Prob \big[ L ( \lambda ) \big] \geq 1 - \gamma - M_n \right\rbrace + M_n 
		= \lambda_{\gamma + M_n} + M_n, \\
		\hat{\lambda}_{\gamma , n} =& \sup \left\lbrace \lambda \in \R \colon \Prob_n \big[ f_n^{-1} \big( [\lambda, + \infty) \big) \big] \geq 1 - \gamma  \right\rbrace \\
		\geq& \sup \left\lbrace \lambda \in \R \colon \Prob \big[ L ( \lambda + M_n ) \big] \geq 1 - \gamma + M_n \right\rbrace \\
		=& \sup \left\lbrace \lambda \in \R \colon \Prob \big[ L ( \lambda ) \big] \geq 1 - \gamma + M_n \right\rbrace - M_n
		= \lambda_{\gamma - M_n} - M_n,
	\end{align*}
	also hold for all $\gamma \in (0,1)$, completing the proof
\end{proof}

Once proven that the distance between $\lambda_{\gamma}$ and $\hat{\lambda}_{\gamma , n}$ is bounded by the maximum of $D_n$ and $S_n$, one needs to show how fast these quantities converge to zero. Proposition~\ref{prop:convergenceprob} ensures that $S_n$ is smaller than $(\log n / n)^{1/2}$ almost surely.

\begin{prop}
	\label{prop:convergenceprob}
	Let $M$ be a Riemannian manifold and $f$ a density function on $M$. Let $X_1, \ldots, X_n$ be a random sample from $f$. Then
	\begin{equation*}
		\sup_{\lambda \in \R} \abs{\Prob \big[ L (\lambda) \big] - \Prob_n \big[ L (\lambda) \big] } \leq \sqrt{\frac{\log n}{n}}
	\end{equation*}
	eventually almost surely, where
	\begin{equation*}
		\Prob_n \big[ L (\lambda) \big] = \frac{1}{n} \sum_{i = 1}^{n} \mathbbm{1} \big[ f (X_i) \geq \lambda \big].
	\end{equation*}
\end{prop}

\begin{proof}
	Consider the random variable $Y = -f(X)$ and its cumulative distribution function defined by
	\begin{equation*}
		F_{Y} (t) = \Prob [Y \leq t] = \Prob[-f (X) \leq t] = \Prob [f(X) \geq -t] = \Prob [L(-t)].
	\end{equation*}
	One can estimate $F_Y$ from the random sample $-f(X_1), \ldots, -f(X_n)$ with the empirical distribution function:
	\begin{equation*}
		F_{Y, n} (t) = \frac{1}{n} \sum_{i = 1}^{n} \mathbbm{1} \big[ -f (X_i) \leq t \big] = \frac{1}{n} \sum_{i = 1}^{n} \mathbbm{1} \big[ f (X_i) \geq -t \big] = \Prob_n \big[ L (-t) \big].
	\end{equation*}
	Therefore
	\begin{equation*}
		\sup_{\lambda \in \R} \abs{\Prob \big[ L (\lambda) \big] - \Prob_n \big[ L (\lambda) \big] } = \sup_{t \in \R} \abs{ F_Y ( t ) - F_{Y, n} (t) }.
	\end{equation*}
	
	The Dvoretzky–Kiefer–Wolfowitz inequality \citep{Massart1990} yields
	\begin{equation*}
		\Prob \left( \sup_{t \in \R} \abs{ F_Y ( t ) - F_{Y, n} (t) } > \sqrt{\frac{\log n}{n}} \right) \leq 2 \exp (-2 \log n) = 2 n^{-2}.
	\end{equation*}
	Consequently
	\begin{equation*}
		\sum_{n = 1}^{+ \infty} \Prob \left( \sup_{\lambda \in \R} \abs{\Prob \big[ L (\lambda) \big] - \Prob_n \big[ L (\lambda) \big] } > \sqrt{\frac{\log n}{n}} \right) \leq 2 \sum_{n = 1}^{+ \infty} n^{-2} < + \infty.
	\end{equation*}
	And the Borel-Cantelli Lemma ensures that
	\begin{multline*}
		\Prob \left[ \limsup_{n \in \N} \left( \sup_{\lambda \in \R} \abs{\Prob \big[ L (\lambda) \big] - \Prob_n \big[ L (\lambda) \big] } > \sqrt{\frac{\log n}{n}} \right) \right] = 0 \Rightarrow \\
		\Rightarrow \Prob \left[ \liminf_{n \in \N} \left( \sup_{\lambda \in \R} \abs{\Prob \big[ L (\lambda) \big] - \Prob_n \big[ L (\lambda) \big] } \leq \sqrt{\frac{\log n}{n}} \right) \right] = 1;
	\end{multline*}
	showing that $\sup_{\lambda \in \R} \abs{\Prob \big[ L (\lambda) \big] - \Prob_n \big[ L (\lambda) \big] } \leq \sqrt{\frac{\log n}{n}}$ holds eventually almost surely.
\end{proof}

Theorem~\ref{th:lambdaconvergence} then follows from Propositions~\ref{prop:levymetric} and~\ref{prop:convergenceprob}.

\begin{proof}[Proof of Theorem~\ref{th:lambdaconvergence}]
	Let
	\begin{equation*}
		S_n = \sup_{\lambda \geq 0} \abs{\Prob \big[ L (\lambda) \big] - \Prob_n \big[ L (\lambda) \big] }
	\end{equation*}
	and $M_n = \max \{D_n, S_n\}$. Proposition \ref{prop:levymetric} guarantees that
	\begin{equation*}
		\lambda_{\gamma - M_n} - M_n \leq \hat{\lambda}_{\gamma, n} \leq \lambda_{\gamma + M_n} + M_n, \quad \forall \gamma \in [\underline{\gamma}, \overline{\gamma}].
	\end{equation*}
	If $M_n < \delta$, then Assumption~P yields
	\begin{equation*}
		\sup_{\gamma \in [\underline{\gamma}, \overline{\gamma} ] } \abs{\lambda_{\gamma} - \lambda_{\gamma \pm M_n}} \leq kM_n.
	\end{equation*}
	Therefore, if $M_n < \delta$, it follows that
	\begin{equation*}
		\lambda_{\gamma } - (1 + k) M_n \leq \hat{\lambda}_{\gamma, n} \leq \lambda_{\gamma} + (1 + k) M_n,  \quad \forall \gamma \in [\underline{\gamma}, \overline{\gamma}];
	\end{equation*}
	and, consequently,
	\begin{equation*}
		\sup_{\gamma \in [\underline{\gamma}, \overline{\gamma} ] } \abs{\lambda_{\gamma} - \hat{\lambda}_{\gamma, n}} \leq (1 + k) M_n.
	\end{equation*}
	
	Proposition \ref{prop:convergenceprob} ensures that $S_n$ converges to zero almost surely. Since $D_n$ is assumed to converge to zero almost surely, $M_n \rightarrow 0$ almost surely. Hence $M_n < \delta$ happens eventually almost surely, and then
	\begin{equation*}
		\sup_{\gamma \in [\underline{\gamma}, \overline{\gamma} ] } \abs{\lambda_{\gamma} - \hat{\lambda}_{\gamma, n}} \leq (1 + k) M_n,
	\end{equation*}
	holds eventually almost surely too. In addition, Proposition \ref{prop:convergenceprob} also shows that $S_n \leq \sqrt{\frac{\log n}{n}}$ eventually almost surely. Therefore,
	\begin{equation*}
		\sup_{\gamma \in [\underline{\gamma}, \overline{\gamma} ] } \abs{\lambda_{\gamma} - \hat{\lambda}_{\gamma, n}} \leq (1 + k) \max \left\lbrace D_n, \sqrt{\frac{\log n}{n}} \right\rbrace
	\end{equation*}
	eventually almost surely, concluding the proof.
\end{proof}

\subsection{Proof of Theorem~6.1}
\label{sec:rnproof}

This section contains the proof of Theorem~\ref{th:rconsistent}, which is long and requires several auxiliary results. The key to showing the consistency of $r_n (\lambda)$ is to control the error made when one approximates any set $S$ with the discretization $\mathcal{X}_n \cap S$, that is, to control how fast the Hausdorff distance $\mathcal{d}_H (\mathcal{X}_n \cap S , S)$ goes to zero. The next result ensures that the probability of $\mathcal{d}_H (\mathcal{X}_n \cap S , S) > \varepsilon$ goes to zero exponentially with~$n$.

\begin{lemma}
	\label{lemma:thesampleisdense}
	Let $M$ be a Riemannian manifold under Assumptions M and $f:M \rightarrow \R$ a density function. Let $\mathcal{X}_n = \{ X_1, \ldots, X_n \}$ be an i.i.d.~sample from the density function $f$.
	
	Let $A \subset M$ a bounded set such that $f(x) \geq b > 0$ for all $x \in A$. Let $\delta, \varepsilon > 0$ such that $2\varepsilon < \delta$. Then, the following inequality holds:
	\begin{equation*}
		\Prob \Big[ \sup_{S \in \mathcal{S}_{\delta} (A)} \mathcal{d}_H (\mathcal{X}_n \cap S , S) > \varepsilon \Big] \leq D \left( \frac{\varepsilon}{12}, A \right)^2 \exp \left[ -nab \min \left\lbrace \frac{\varepsilon}{12} , \rho \right\rbrace^{d} \right];
	\end{equation*}
	where $\mathcal{S}_{\delta} (A) = \{ S \subset A \colon \Psi_{\delta} (S) = S \}$, and $\rho$ and $a$ are the constants in Assumption (M2) and $D (\cdot, \cdot)$ is the functional defined in Proposition~\ref{prop:M2impliesM3}.
\end{lemma}

\begin{proof}
	It is a consequence of Lemma~\ref{lemma:covering}b. Let $S \in \mathcal{S}_{\delta} (A)$ and take
	\begin{equation*}
		S' = (S \ominus \delta B) \oplus (\delta - 2\varepsilon/3) B \quad \text{ and } \quad A' = (A \ominus \delta B) \oplus (\delta - 2\varepsilon/3) B.
	\end{equation*}
	Then,
	\begin{multline*}
		S' \oplus (2\varepsilon/3 - \varepsilon/3) B \subset (\mathcal{X}_n \cap S') \oplus 2\varepsilon/3 B
		\Rightarrow
		\\
		\Rightarrow
		S' \oplus 2\varepsilon/3 B = \Psi_{\delta} ( S ) = S  \subset (\mathcal{X}_n \cap S') \oplus \varepsilon B \subset (\mathcal{X}_n \cap S) \oplus \varepsilon B
		\Rightarrow
		\\
		\Rightarrow
		\mathcal{d}_H (\mathcal{X}_n \cap S , S) \leq \varepsilon,
	\end{multline*}
	where we have used Proposition~\ref{prop:randsmin}a. Furthermore, $S' \in \mathcal{S}_{2\varepsilon/3} (A')$ by Propositions~\ref{prop:trivial}d and~\ref{prop:randspsi}d. Consequently,
	\begin{align*}
		\Prob \Big[ \sup_{S \in \mathcal{S}_{\delta} (A)} & \mathcal{d}_H (\mathcal{X}_n \cap S , S) > \varepsilon \Big]
		\\
		=&
		\Prob \Big[ \mathcal{d}_H (\mathcal{X}_n \cap S , S) > \varepsilon \text{ for some } S \in \mathcal{S}_{\delta} (A) \Big]
		\\
		\leq&
		\Prob \Big[ S' \oplus (2\varepsilon/3 - \varepsilon/3) B \not\subset (\mathcal{X}_n \cap S') \oplus 2\varepsilon/3 B \text{ for some } S' \in \mathcal{S}_{2\varepsilon/3} (A') \Big].
	\end{align*}
	So, applying Lemma~\ref{lemma:covering}b with $r = \delta = 2\varepsilon/3$ and $\varepsilon = \varepsilon/3$ yields
	\begin{align*}
		\Prob \Big[ \sup_{S \in \mathcal{S}_{\delta} (A)} \mathcal{d}_H & (\mathcal{X}_n \cap  S , S) > \varepsilon \Big]
		\\
		\leq &
		D \left( \frac{\varepsilon}{12}, A' \right)^2 \exp \left[ -nab \min \left\lbrace \frac{\varepsilon}{3} , \rho \right\rbrace^{(d-1)/2} \min \left\lbrace \frac{\varepsilon}{12} , \rho \right\rbrace^{(d+1)/2} \right]
		\\
		\leq &
		D \left( \frac{\varepsilon}{12}, A \right)^2 \exp \left[ -nab \min \left\lbrace \frac{\varepsilon}{12} , \rho \right\rbrace^{d} \right]. \qedhere
	\end{align*}
	
\end{proof}

Lemma~\ref{lemma:thesampleisdense} and the Borel--Cantelli Lemma allow to show that $\mathcal{d}_H (\mathcal{X}_n \cap S , S)$ converges to zero uniformly on $S$, and how fast it does so.

\begin{cor}
	\label{cor:unifrate}
	Let $M$ be a Riemannian manifold under Assumptions M and $f:M \rightarrow \R$ a density function. Let $\mathcal{X}_n = \{ X_1, \ldots, X_n \}$ be an i.i.d.~sample from the density function $f$.
	
	Let $A \subset M$ a bounded set such that $f(x) \geq b > 0$ for all $x \in A$, and let $\delta > 0$. Then,
	\begin{equation*}
		\sup_{S \in \mathcal{S}_{\delta} (A)} \mathcal{d}_H (\mathcal{X}_n \cap S, S)
		=
		O_{\textrm{\rm a.s.}} \Bigg( \bigg( \frac{\log n}{n} \bigg)^{1/d} \Bigg)
	\end{equation*}
	where $\mathcal{S}_{\delta} (A) = \{ S \subset A \colon \Psi_{\delta} (S) = S \}$.
\end{cor}

\begin{proof}
	Let
	\begin{equation*}
		\varepsilon_n = \bigg( \frac{c \log n}{n} \bigg)^{1/d}
	\end{equation*}
	where $c > 0$ is a constant that will be specified later on. Since $\varepsilon_n \to 0$, $\varepsilon_n < \min \{ \delta/2, \rho \}$ for large $n$, where $\rho > 0$ is the constant from Assumption (M2). If so, Lemma~\ref{lemma:thesampleisdense}b ensures that
	\begin{align*}
		\Prob \Big[ \sup_{S \in \mathcal{S}_{\delta} (A)} \mathcal{d}_H (\mathcal{X}_n \cap S , S) > \varepsilon_n \Big]
		&\leq
		D \left( \frac{\varepsilon_n}{12}, A \right)^2 \exp \left[ - \bigg(\frac{ab}{12^d} \bigg) n \varepsilon_n^{d} \right]
		\\
		&=
		O
		\bigg(
		\varepsilon_n^{-2d} \exp \left[ - \bigg(\frac{ab}{12^d} \bigg) n \varepsilon_n^{d} \right]
		\bigg)
	\end{align*}
	where we have used Proposition~\ref{prop:M2impliesM3}. In addition,
	\begin{equation*}
		\varepsilon_n^{-2d} \exp \left[ - \bigg(\frac{ab}{12^d} \bigg) n \varepsilon_n^{d} \right]
		=
		o
		\big(
		n^{2 - abc/12^d}
		\big)
	\end{equation*}
	Taking $c$ large enough so that $$2 - abc/12^d \leq -2$$ guarantees that
	\begin{equation*}
		\Prob \Big[ \sup_{S \in \mathcal{S}_{\delta} (A)} \mathcal{d}_H (\mathcal{X}_n \cap S , S) > \varepsilon_n \Big]
		=
		o
		\big(
		n^{-2}
		\big).
	\end{equation*}
	Therefore,
	\begin{equation*}
		\sum_{n = 1}^{+ \infty}\Prob \Big[ \sup_{S \in \mathcal{S}_{\delta} (A)} \mathcal{d}_H (\mathcal{X}_n \cap S , S) > \varepsilon_n \Big]
		< + \infty
	\end{equation*}
	and the result holds by the Borel-Cantelli Lemma.
\end{proof}

Lastly, the proof of Theorem~\ref{th:rconsistent} requires two further technical results. The first one connects the distances $\mathcal{d}_H ( S \ominus rB , S \big)$ and $\mathcal{d}_H \big( \Psi_{r} (S), S \big)$ through a useful equation.

\begin{lemma}
	\label{lemma:dists-rs}
	Let $M$ be a manifold under Assumptions M, and let $S \subset M$ a non-empty compact subset such that $\Psi_{r} (S) \neq S$. Then,\footnote{For the sake of convenience, the statement of Lemma~\ref{lemma:dists-rs} assumes that $\mathcal{d}_H (\emptyset, S) = + \infty$ for any non--empty compact set $S$, and $+\infty + r = +\infty$. This allows to avoid treating the case $S \ominus r B = \emptyset$ separately.}
	$$\mathcal{d}_H ( S \ominus rB , S \big) =  \mathcal{d}_H \big( \Psi_{r} (S), S \big) + r.$$
\end{lemma}

\begin{proof}
	Let $\varepsilon > 0$. From Proposition~\ref{prop:randsmin}a, one deduces
	\begin{multline*}
		\mathcal{d}_H \big( \Psi_{r} (S), S \big) = \mathcal{d}_H \big( (S \ominus rB) \oplus rB, S \big) \leq \varepsilon
		\Rightarrow
		\\
		\Rightarrow
		S \subset \Big[ (S \ominus rB) \oplus rB \Big] \oplus \varepsilon B = (S \ominus rB) \oplus (r + \varepsilon) B  
		\Rightarrow
		\\
		\Rightarrow
		\mathcal{d}_H ( S \ominus rB , S \big) \leq r + \varepsilon.
	\end{multline*}
	Consequently, $\mathcal{d}_H ( S \ominus rB , S \big) \leq \mathcal{d}_H \big( \Psi_{r} (S), S \big) + r$.
	
	On the other hand, $\Psi_{r} (S) \neq S$ implies that $S \not\subset (S \ominus rB) \oplus rB$. Therefore, if $\mathcal{d}_H ( S \ominus rB , S ) \leq \varepsilon$, then $\varepsilon \geq r$. This and Proposition~\ref{prop:randsmin}a again guarantee that
	\begin{multline*}
		\mathcal{d}_H \big( S \ominus rB , S \big) \leq \varepsilon
		\Rightarrow
		\\
		\Rightarrow
		S \subset (S \ominus rB) \oplus \varepsilon B = \Big[ (S \ominus rB) \oplus rB \Big] \oplus (\varepsilon - r) B
		\Rightarrow
		\\
		\Rightarrow
		\mathcal{d}_H \big( \Psi_{r} (S), S \big) = \mathcal{d}_H \big( (S \ominus rB) \oplus rB, S \big) \leq \varepsilon - r.
	\end{multline*}
	Hence, $\mathcal{d}_H \big( \Psi_{r} (S), S \big) \leq \mathcal{d}_H ( S \ominus rB , S \big) - r$, concluding the proof.
\end{proof}

The second result ensures that some special quantity is not zero.

\begin{lemma}
	\label{lemma:positivedistance}
	Under the hypotheses of Theorem~\ref{th:rconsistent}, it holds that
	\begin{equation}
		\inf_{\lambda \in [l, u] }
		\mathcal{d}_H \big( \Psi_{r_0(\lambda)  + \varepsilon } \big( L(\lambda) \big), L (\lambda) \big) > 0,
	\end{equation}
	for any $\varepsilon > 0$, where $l$ and $u$ are the constants in Assumptions~L.
\end{lemma}

\begin{proof}
	We will show the result by contradiction. Assume that 
	\begin{equation*}
		\inf_{\lambda \in [l, u] }
		\mathcal{d}_H \big( \Psi_{r_0(\lambda)  + \varepsilon} \big( L(\lambda) \big), L (\lambda) \big) = 0.
	\end{equation*}
	Then, there exists a convergent sequence $\lambda_n \to \lambda_0 \in [l, u]$ such that
	\begin{equation*}
		\mathcal{d}_H \big( \Psi_{r_0(\lambda_n)  + \varepsilon } \big( L(\lambda_n) \big), L (\lambda_n) \big) \to 0.
	\end{equation*}
	Taking into account that $\Psi_{r_0(\lambda)  + \varepsilon } \big( L(\lambda) \big) \subset L (\lambda)$, one has
	\begin{multline*}
		\mathcal{d}_H \big( \Psi_{r_0(\lambda_0)  + \varepsilon } \big( L(\lambda_0) \big), L (\lambda_0) \big)
		=
		\overrightarrow{\mathcal{d}_H} \big( L (\lambda_0) , \Psi_{r_0(\lambda_0)  + \varepsilon } \big( L(\lambda_0) \big) \big)
		\\
		\leq
		\overrightarrow{\mathcal{d}_H} \big( L (\lambda_0) , L (\lambda_n) \big)
		+
		\overrightarrow{\mathcal{d}_H} \big( L (\lambda_n) , \Psi_{r_0(\lambda_n)  + \varepsilon } \big( L(\lambda_n) \big) \big)
		\\
		+
		\overrightarrow{\mathcal{d}_H} \big( \Psi_{r_0(\lambda_n)  + \varepsilon } \big( L(\lambda_n) \big) , \Psi_{r_0(\lambda_0)  + \varepsilon } \big( L(\lambda_0) \big) \big).
	\end{multline*}
	The right-hand side of the inequality tends to zero by Assumption (L4), the continuity of $r_0(\lambda)$, and Corollary~\ref{cor:continuouspsi}. Then, $\mathcal{d}_H \big( \Psi_{r_0(\lambda_0)  + \varepsilon } \big( L(\lambda_0) \big), L (\lambda_0) \big) = 0$, which is a contradiction with the definition of $r_0 (\lambda)$ (see Equation~\eqref{eq:setest}). Hence, the result holds.
\end{proof}

We can finally proof Theorem~\ref{th:rconsistent}. It should be noticed that, from now on, $l, u, \delta$ and $k$ are the constants in Assumptions~L.

\begin{proof}[Proof of Theorem~\ref{th:rconsistent}]
	Let $\varepsilon > 0$ such that $\varepsilon < \min_{\lambda \in [l, u]} r_0(\lambda)$. The main argument of the proof consists in bounding the random variables
	\begin{equation}
		\label{eq:auxsup}
		\sup_{\lambda \in [l, u] }
		\bigg(
		\sup_{r \in [0, r_0(\lambda) - \varepsilon)} \Big[ \mathcal{d}_H \Big( \mathcal{X}_n^{+} (\lambda) \cap \big( \mathcal{X}_n^{-} (\lambda) \oplus r B \big)^c , \mathcal{X}_n^{+} (\lambda) \Big) - r \Big]
		\bigg)
	\end{equation}
	and
	\begin{equation}
		\label{eq:auxinf}
		\inf_{\lambda \in [l, u] }
		\bigg(
		\inf_{r \geq r_0(\lambda) + \varepsilon} \Big[ \mathcal{d}_H \Big( \mathcal{X}_n^{+} (\lambda) \cap \big( \mathcal{X}_n^{-} (\lambda) \oplus r B \big)^c , \mathcal{X}_n^{+} (\lambda) \Big) - r \Big]
		\bigg)
		.
	\end{equation}
	If \eqref{eq:auxsup} is smaller than $h_n$ (almost surely), then, for any $\lambda \in [l,u]$, every $r$ smaller than $r_0(\lambda) - \varepsilon$ will fulfill
	\begin{equation*}
		\mathcal{d}_H \Big( \mathcal{X}_n^{+} (\lambda) \cap \big( \mathcal{X}_n^{-} (\lambda) \oplus r B \big)^c , \mathcal{X}_n^{+} (\lambda) \Big) \leq r + h_n,
	\end{equation*}
	and, therefore, $\inf_{\lambda \in [l, u] } (r_n (\lambda) - r_0(\lambda))$ will be larger than $- \varepsilon$. Reciprocally, \eqref{eq:auxinf} being larger than $h_n$ will imply that $\sup_{\lambda \in [l, u] } (r_n (\lambda) - r_0(\lambda))$ is also smaller than $\varepsilon$, proving the result.
	
	Let us start by bounding \eqref{eq:auxsup}. Take $\lambda \in [l, u]$, and fix $r \in [0, r_0(\lambda) - \varepsilon)$. Lemma~\ref{lemma:X+andX-} yields
	\begin{align*}
		\mathcal{X}_n \cap \big[ L(\lambda) \ominus (r + kD_n) B \big]
		\subset &
		\mathcal{X}_n^{+} (\lambda) \cap \big( \mathcal{X}_n^{-} (\lambda) \oplus r B \big)^c
		\\
		\subset&
		\mathcal{X}_n^{+} (\lambda)
		\subset
		L (\lambda) \oplus kD_n B.
	\end{align*}
	So,
	\begin{align}
		\nonumber
		\mathcal{d}_H \Big( \mathcal{X}_n^{+} (\lambda) \cap& \big( \mathcal{X}_n^{-} (\lambda) \oplus r B \big)^c , \mathcal{X}_n^{+} (\lambda) \Big)
		\\
		\nonumber
		\leq&
		\mathcal{d}_H \Big( \mathcal{X}_n \cap \big[ L(\lambda) \ominus (r + kD_n) B \big] , L (\lambda) \oplus kD_n B \Big)
		\\
		\nonumber
		\leq&
		\mathcal{d}_H \Big( \mathcal{X}_n \cap \big[ L(\lambda) \ominus (r + kD_n) B \big] , L(\lambda) \ominus (r + kD_n) B \Big)
		\\
		&+
		\label{eq:distL-rL}
		\mathcal{d}_H \Big( L(\lambda) \ominus (r + kD_n) B ,  L (\lambda) \oplus kD_n B \Big)
	\end{align}
	If $kD_n < \varepsilon/2$, Equation~\eqref{eq:defr} and Propositions~\ref{prop:randspsi}d and~\ref{prop:randsmin}a guarantee that
	$$ L (\lambda) \oplus kD_n B = \big( L(\lambda) \ominus (r + kD_n) B \big) \oplus (r + 2kD_n) B.$$
	Therefore, the second term in \eqref{eq:distL-rL} is smaller that $r + 2kD_n$. This in turn implies
	\begin{multline*}
		\mathcal{d}_H \Big( \mathcal{X}_n^{+} (\lambda) \cap \big( \mathcal{X}_n^{-} (\lambda) \oplus r B \big)^c , \mathcal{X}_n^{+} (\lambda) \Big) - r
		\\
		\leq
		\mathcal{d}_H \Big( \mathcal{X}_n \cap \big[ L(\lambda) \ominus (r + kD_n) B \big] , L(\lambda) \ominus (r + kD_n) B \Big)
		+ 2 kD_n.
	\end{multline*}
	If $kD_n < \varepsilon/2$, Equation~\eqref{eq:defr}, Assumption (L2) and Lemma~\ref{lemma:psidifference} ensure that $\Psi_{\varepsilon/2} \big( L(\lambda) \ominus (r + kD_n) B \big) =  L(\lambda) \ominus (r + kD_n) B$. Hence,
	\begin{equation*}
		\mathcal{d}_H \Big( \mathcal{X}_n^{+} (\lambda) \cap \big( \mathcal{X}_n^{-} (\lambda) \oplus r B \big)^c , \mathcal{X}_n^{+} (\lambda) \Big) - r
		\leq
		\sup_{S \in \mathcal{S}_{\varepsilon/2} ( L(l) )}
		\mathcal{d}_H ( \mathcal{X}_n \cap S , S ) + 2kD_n.
	\end{equation*}
	The right-hand side does not depend on $r$ or $\lambda$, so
	\begin{multline*}
		\sup_{\lambda \in [l, u] }
		\bigg(
		\sup_{r \in [0, r_0(\lambda) - \varepsilon)} \Big[ \mathcal{d}_H \Big( \mathcal{X}_n^{+} (\lambda) \cap \big( \mathcal{X}_n^{-} (\lambda) \oplus r B \big)^c , \mathcal{X}_n^{+} (\lambda) \Big) - r \Big]
		\bigg)
		\\
		\leq
		\sup_{S \in \mathcal{S}_{\varepsilon/2} ( L(l) )}
		\mathcal{d}_H ( \mathcal{X}_n \cap S , S ) + 2kD_n
	\end{multline*}
	also holds if $kD_n < \varepsilon/2$. By Corollary~\ref{cor:unifrate} and Equation~\eqref{eq:hn}, 
	$$
	h_n^{-1} 
	\bigg(
	\sup_{S \in \mathcal{S}_{\varepsilon/2} ( L(l) )}
	\mathcal{d}_H ( \mathcal{X}_n \cap S , S ) + 2kD_n
	\bigg)
	\to 
	0
	$$
	almost surely. Since $D_n \to 0$ almost surely, it follows that
	\begin{multline*}
		\Prob \Bigg[ \liminf_{n \to + \infty} \bigg( \sup_{\lambda \in [l, u] }
		\bigg(
		\sup_{r \in [0, r_0(\lambda) - \varepsilon)} \Big[ \mathcal{d}_H \Big( \mathcal{X}_n^{+} (\lambda) \cap \big( \mathcal{X}_n^{-} (\lambda) \oplus r B \big)^c , \mathcal{X}_n^{+} (\lambda) \Big) - r \Big]
		\bigg) \leq h_n \bigg)
		\Bigg]
		\\
		\geq
		\Prob \Bigg[ \liminf_{n \to + \infty} \bigg( \sup_{S \in \mathcal{S}_{\varepsilon/2} ( L(l) )}
		\mathcal{d}_H ( \mathcal{X}_n \cap S , S ) + 2kD_n \leq h_n \bigg)
		\Bigg]
		=
		1.
	\end{multline*}
	This last result and Equation~\eqref{eq:defrn} imply
	\begin{equation}
		\label{eq:rnlowerbound}
		\Prob \bigg[ \liminf_{n \to + \infty} \Big( \inf_{\lambda \in [l, u] } \big( r_n (\lambda) - r_0(\lambda) \big) \geq - \varepsilon \Big) \bigg]
		=
		1.
	\end{equation}
	
	Now we will try to bound \eqref{eq:auxinf}. Take $\xi  = \min_{x \in L(l) } \mathcal{d} (x, L(l/2)^c)$ for all $\lambda \in [l, u]$. We know from Lemma~\ref{lemma:transU} that $\xi > 0$ and $L(\lambda) \oplus \xi B \subset L(l/2)$. Define,
	\begin{equation*}
		\Delta_n =  \sup_{\lambda \in [l, u] }\mathcal{d}_H \Big( \mathcal{X}_n \cap \big( [L(\lambda) \oplus \xi B] \setminus [L(\lambda) \oplus kD_n B] \big) , [L(\lambda) \oplus \xi B] \setminus [L(\lambda) \oplus kD_n B] \Big).
	\end{equation*}
	If $k D_n < \min \{\xi/2, \delta/2\}$, Assumption (L2) and Lemma~\ref{lemma:psidifference} yield that
	\begin{align*}
		\Psi_{\delta/2} \big( [L(\lambda) \oplus kD_n B]^c \big)
		=&
		\Psi_{\delta/2} \big( L(\lambda)^c \ominus kD_n B \big)
		\\
		=& L(\lambda)^c \ominus kD_n B = [L(\lambda) \oplus kD_n B]^c.
	\end{align*}
	Then, taking $c = \min \{\xi/6, \delta/2\}$, Lemma~\ref{lemma:psiminusset} ensures that
	\begin{equation*}
		\Psi_{c} \Big( [L(\lambda) \oplus \xi B] \setminus [L(\lambda) \oplus kD_n B] \Big) = [L(\lambda) \oplus \xi B] \setminus [L(\lambda) \oplus kD_n B].
	\end{equation*}
	Hence, if $k D_n < \min \{\xi/2, \delta/2\}$,
	\begin{equation*}
		\Delta_n \leq \sup_{S \in \mathcal{S}_{c} ( L(l/2) )}
		\mathcal{d}_H ( \mathcal{X}_n \cap S , S ).
	\end{equation*}
	Since $D_n \to 0$ almost surely, $\Prob \big[ \liminf_{n \to + \infty } \big( k D_n < \min \{\xi/2, \delta/2\} \big) \big] = 1$. This result jointly with Corollary~\ref{cor:unifrate} and Equation~\eqref{eq:hn} guarantee
	\begin{equation}
		\label{eq:Deltan}
		h_n^{-1} \Delta_n \to 0, \quad \text{ almost surely.}
	\end{equation}
	
	Take now $\lambda \in [l, u]$ and fix $r \in [r_0(\lambda) + \varepsilon, + \infty)$. If $k D_n < \min \{ \xi, \varepsilon/2, \delta \}$  and $\Delta_n \leq \varepsilon/4$, then
	\begin{multline*}
		[L(\lambda) \oplus \xi B] \setminus [L(\lambda) \oplus kD_n B]
		\subset
		\Big( \mathcal{X}_n \cap \big( [L(\lambda) \oplus \xi B] \setminus [L(\lambda) \oplus kD_n B] \big) \Big) \oplus \Delta_n B
		\Rightarrow
		\\
		\Rightarrow
		[L(\lambda) \oplus \xi B] \setminus [L(\lambda) \oplus kD_n B]
		\subset
		\mathcal{X}_n^{-} (\lambda) \oplus \Delta_n B
		\Rightarrow
		\\
		\Rightarrow
		[L(\lambda) \oplus (\xi + r - \Delta_n) B] \setminus [ L(\lambda) \ominus (r - \Delta_n - kD_n) B]
		\subset
		\mathcal{X}_n^{-} (\lambda) \oplus r B
		\Rightarrow
		\\
		\Rightarrow
		\mathcal{X}_n^{+} (\lambda) \cap \big( \mathcal{X}_n^{-} (\lambda) \oplus r B \big)^c
		\subset
		L(\lambda) \ominus (r - \Delta_n - kD_n) B;
	\end{multline*}
	where we have used Assumption (L2), Proposition~\ref{prop:randsmin}b and Lemmas~\ref{lemma:psidifference}, \ref{lemma:difsum} and~\ref{lemma:X+andX-}. Hence, if $k D_n < \min \{ \xi, \varepsilon/2, \delta \}$  and $\Delta_n \leq \varepsilon/4$,
	\begin{align*}
		\mathcal{d}_H \big( \mathcal{X}_n^{+} (\lambda) & \cap \big( \mathcal{X}_n^{-} (\lambda) \oplus r B \big)^c , \mathcal{X}_n^{+} (\lambda) \big)
		\\
		\geq&
		\mathcal{d}_H \big( L(\lambda) \ominus (r - \Delta_n - kD_n) B , \mathcal{X}_n^{+} (\lambda) \big)
		\\
		\geq&
		\mathcal{d}_H \big(L(\lambda) \ominus (r - \Delta_n - kD_n) B , L (\lambda) \big)
		\\
		&-
		\mathcal{d}_H \big(L (\lambda), L (\lambda) \oplus kD_n B \big)
		-
		\mathcal{d}_H \big( \mathcal{X}_n^{+} (\lambda)  , L (\lambda) \oplus kD_n B \big)
		\\
		\geq&
		\mathcal{d}_H \big( \Psi_{r - \Delta_n  - kD_n } \big( L(\lambda) \big) , L (\lambda) \big) + r
		-
		\Delta_n - 2kD_n
		\\
		&-
		\mathcal{d}_H \big( \mathcal{X}_n \cap [L (\lambda) \oplus kD_n B]  , L (\lambda) \oplus kD_n B \big)
	\end{align*}
	where we have used the triangular inequality, that $r - \Delta_n  - kD_n > r_0(\lambda) + \varepsilon/4$, and Lemmas~\ref{lemma:dists-rs} and~\ref{lemma:X+andX-}. Moreover, $\Psi_{r - \Delta_n  - kD_n} \big( L(\lambda) \big) \subset \Psi_{r_0(\lambda) + \varepsilon/4 } \big( L(\lambda) \big)$ by Proposition~\ref{prop:randspsi}c. So, if $k D_n < \min \{ \xi, \varepsilon/2, \delta \}$ and $\Delta_n \leq \varepsilon/4$,
	\begin{align*}
		\mathcal{d}_H \big( \mathcal{X}_n^{+} (\lambda) \cap \big( \mathcal{X}_n^{-} (\lambda) &\oplus r B \big)^c , \mathcal{X}_n^{+} (\lambda) \big)
		- r
		\\
		\geq&
		\mathcal{d}_H \big( \Psi_{r_0(\lambda)  + \varepsilon/4 } \big( L(\lambda) \big) , L (\lambda) \big)
		- \Delta_n - 2kD_n
		\\
		&- \mathcal{d}_H \big( \mathcal{X}_n \cap [L (\lambda) \oplus kD_n B]  , L (\lambda) \oplus kD_n B \big).
	\end{align*}
	Furthermore, since $\varepsilon < \inf_{\lambda \in [l, u] } r_0(\lambda)$, Propositions~\ref{prop:trivial}d and~\ref{prop:randspsi}d ensure that $\Psi_{\varepsilon} \big( L (\lambda) \oplus kD_n B \big) = L (\lambda) \oplus kD_n B$. Hence,
	\begin{align*}
		\mathcal{d}_H \big( \mathcal{X}_n^{+} (\lambda) \cap \big( \mathcal{X}_n^{-} (\lambda) &\oplus r B \big)^c , \mathcal{X}_n^{+} (\lambda) \big)
		- r
		\\
		\geq&
		\mathcal{d}_H \big( \Psi_{r_0(\lambda)  + \varepsilon/4 } \big( L(\lambda) \big) , L (\lambda) \big)
		- \Delta_n - 2kD_n
		\\
		&- \sup_{S \in \mathcal{S}_{\varepsilon} ( L(l/2) )}\mathcal{d}_H \big( \mathcal{X}_n \cap S  , S \big).
	\end{align*}
	This bound holds for any $r > r_0(\lambda) + \varepsilon$ if $k D_n < \min \{ \xi, \varepsilon/2, \delta \}$ and $\Delta_n \leq \varepsilon/4$. So,
	\begin{align}
		\nonumber
		\inf_{\lambda \in [l, u] }
		\bigg(
		\inf_{r \geq r_0(\lambda) + \varepsilon}
		\Big[
		\mathcal{d}_H& \big( \mathcal{X}_n^{+} (\lambda) \cap \big( \mathcal{X}_n^{-} (\lambda) \oplus r B \big)^c , \mathcal{X}_n^{+} (\lambda) \big)
		- r
		\Big]
		\bigg)
		\\
		\nonumber
		\geq&
		\inf_{\lambda \in [l, u] }
		\mathcal{d}_H \big( \Psi_{r_0(\lambda)  + \varepsilon/4 } \big( L(\lambda) \big) , L (\lambda) \big)
		- \Delta_n - 2kD_n
		\\
		\label{eq:inflambda}
		&- \sup_{S \in \mathcal{S}_{\varepsilon} ( L(l/2) )}\mathcal{d}_H \big( \mathcal{X}_n \cap S  , S \big).
	\end{align}
	also holds.
	
	Since both $D_n$ and $\Delta_n$ converge to zero almost surely, $\Prob [ \liminf_{n \to + \infty } ( \Delta_n \leq \varepsilon/4) ] = 1$ and $\Prob [ \liminf_{n \to + \infty } ( k D_n < \min \{ \xi, \varepsilon/2, \delta \} ) ] = 1$. Lemma~\ref{lemma:positivedistance}, Equations~\eqref{eq:hn}, \eqref{eq:Deltan} and \eqref{eq:inflambda} and Corollary~\ref{cor:unifrate} ensure that
	\begin{equation*}
		h_n^{-1}
		\inf_{\lambda \in [l, u] }
		\bigg(
		\inf_{r \geq r_0(\lambda) + \varepsilon} \Big[ \mathcal{d}_H \big( \mathcal{X}_n^{+} (\lambda) \cap \big( \mathcal{X}_n^{-} (\lambda) \oplus r B \big)^c , \mathcal{X}_n^{+} (\lambda) \big)
		- r \Big] 
		\bigg)
		\to + \infty
	\end{equation*}
	almost surely. Hence,
	\begin{equation*}
		\Prob
		\Bigg[
		\liminf_{n \to + \infty}
		\bigg(
		\inf_{\lambda \in [l, u] }
		\bigg(
		\inf_{r \geq r_0(\lambda) + \varepsilon} \Big[ \mathcal{d}_H \big( \mathcal{X}_n^{+} (\lambda) \cap \big( \mathcal{X}_n^{-} (\lambda) \oplus r B \big)^c , \mathcal{X}_n^{+} (\lambda) \big)
		- r \Big] 
		\bigg) > h_n
		\bigg)
		\Bigg]
		=
		1.
	\end{equation*}
	This equality jointly with Equation~\eqref{eq:defrn} guarantee that
	\begin{equation}
		\label{eq:rnupperbound}
		\Prob \bigg[ \liminf_{n \to + \infty} \Big( \sup_{\lambda \in [l, u] } \big( r_n (\lambda) - r_0(\lambda) \big) < \varepsilon \Big) \bigg]
		=
		1.
	\end{equation}
	Since both \eqref{eq:rnlowerbound} and \eqref{eq:rnupperbound} hold for any $\varepsilon > 0$ such that $\varepsilon < \inf_{\lambda \in [l, u] } r_0(\lambda)$, the result holds.
\end{proof}

\section{Relation between Assumptions M, L and P and the smoothness of $f$ and $M$}
\label{sec:appC}

As explained in Sections~\ref{sec:assM} and~\ref{sec:assA}, some assumptions must be made on the manifold $M$ that supports the data and the HDRs of the population for the consistency of the estimator $L_n (\lambda)$. These assumptions, that we refer as Assumptions~M and Assumptions~L, are the following:

\begin{assM}
	
	$M$ is a $d$-dimensional Riemannian manifold verifying:
	
	\begin{enumerate}[label = {(M\arabic*)}]
		\item $M$ is complete and connected.
		\item There exist two constants, $a, \rho > 0$, such that
		\begin{equation*}
			\mathrm{Vol}_{M}  \big( B_{r_1} [x_1] \cap B_{r_2} [x_2] \big) \geq a \min \{ r_1, r_2, \rho \}^{(d-1)/2} \min \{ \varepsilon, \rho \}^{(d+1)/2}.
		\end{equation*}
		for every pair of points $x_1, x_2 \in M$ and radii $r_1, r_2 > 0$ such that $\mathcal{d} (x_1, x_2) < r_1 + r_2$ and every $\varepsilon \leq r_1 + r_2 - \mathcal{d} (x_1, x_2)$.
	\end{enumerate}
	
\end{assM}

\begin{assL}
	Let $M$ be a connected Riemannian manifold, \hbox{$f: M \rightarrow \R$} a continuous density function, and $l$ and $u$ two constants such that $0 < l \leq u < \sup f$.
	There exist a radius, $r > 0$, and two positive constants $k$ and $\delta$, with $k > 0$ and $0 < \delta < r/2$, verifying
	\begin{enumerate}[label = {(L\arabic*)}]
		\item $\delta < r_0 (\lambda) < + \infty$ for all $\lambda \in [l - \delta , u + \delta]$.
		\item $L (\lambda)^c = \Psi_{\delta} \big[ L(\lambda)^c \big]$ for all $\lambda \in [l - \delta, u + \delta]$.
		\item $L (\lambda)$ is a compact set for all $\lambda \in [l - \delta , u + \delta]$.
		\item $\mathcal{d}_H \big( L(\lambda), L(\lambda + \eta) \big) \leq k \vert \eta \vert$ for all $\lambda \in [l, u]$ when $\eta \in [-\delta, \delta]$.
	\end{enumerate}
\end{assL}

In addition, proving the uniform consistency of the level estimator $\lambda_{\gamma, n}$ (see Equation~\eqref{eq:lambdan}) requires a further assumption.

\begin{assP}
	Define
	$\lambda_{\gamma} = \sup \Big\lbrace \lambda \in \R \colon \Prob \big[ L ( \lambda ) \big] \geq 1 - \gamma \Big\rbrace.$
	There exist $0 < \underline{\gamma} \leq \overline{\gamma}$ such that
	\begin{equation*}
		\sup_{\gamma \in [\underline{\gamma}, \overline{\gamma} ] } \abs{\lambda_{\gamma} - \lambda_{\gamma \pm \eta}} \leq k \eta
	\end{equation*}
	for all $\eta \in [0, \delta]$, where $\delta$ and $k$ are the same constants as Assumptions~L.
\end{assP}

In this appendix, these Assumptions M, L and P are carefully investigated, and it is shown that they are related to the differentiability of the density function $f$ and the curvature of $M$. The contents of this appendix are divided in three different sections, each one devoted to a group of assumptions. Specifically, Section~\ref{sec:proofsassM} contains the proofs of Proposition~\ref{prop:M2impliesM3} and Theorem~\ref{th:manifold}, Section~\ref{sec:proofsassA} includes the proofs of Propositions~\ref{prop:assA3} and~\ref{prop:assA4}, and Section~\ref{sec:proofassP} is entirely dedicated to the proof of Proposition~\ref{prop:assA5}.

Since this supplement involves more technical results about Riemannian geometry than the main paper, some additional notation must be introduced before proceeding.
Let $M$ be a connected and complete $d$-dimensional Riemannian manifold, and $\langle \cdot , \cdot \rangle_M$ its Riemannian metric.
Given a point $x \in M$ and a tangent vector $v \in T_x M$, the norm of $v$ is denoted as $\| v \|_M = \sqrt{\langle v , v \rangle_M}$ (the notation $\norm{v}$ will be also used when there is no possibility of confusion). Denote by $\exp_x: T_x M \rightarrow M$ the exponential map restricted at $x$ (see \citealp{Lee2018}, Ch.~5, pp.~126--131). Denote by $i_M (x)$ the injectivity radius of $M$ at $x$, that is, the supremum of all $\delta > 0$ such that $\exp_x$ is a diffeomorphism between $B_{\delta} (0) \subset T_x M$ and its image. Then, denote the injectivity radius of $M$ by $i_M = \sup_{x \in M} i_M (x)$. Given a two Riemannian manifolds, $M$ and $N$, and a smooth map $F: M \rightarrow N$, denote by $dF_x: T_x M \rightarrow T_{f(x)} N$ the differential of $F$ at $x$. Given a smooth function $f: M \rightarrow \R$, the gradient of $f$ is denoted by $\nabla f$. 

\subsection{About Assumptions M}
\label{sec:proofsassM}

This section is dedicated to the proofs of Proposition~\ref{prop:M2impliesM3} and Theorem~\ref{th:manifold} in Section~\ref{sec:assM}. For Proposition~\ref{prop:M2impliesM3}, an auxiliary lemma is required first.

\begin{lemma}
	\label{lemma:balls}
	Let $M$ be a Riemannian manifold verifying Assumption~\ref{ass:M2}. Then,
	\begin{equation*}
		\mathrm{Vol}_{M} \big( B_{r} [x] \big) \geq a \min \{r, \rho \}^d;
	\end{equation*}
	for any $x \in M$, where $a$ and $\rho$ are the constants in Assumption \ref{ass:M2}.
\end{lemma}

\begin{proof}
	This result is a direct consequence of Assumption~\ref{ass:M2}, taking $x_1 = x_2 = x$ and $r_1 = r_2 = \varepsilon = r$.
\end{proof}

\begin{proof}[Proof of Proposition~\ref{prop:M2impliesM3}]
	Let $A$ be a bounded subset of $M$. Hence, there exist a point $z \in M$ and a radius $r > 0$ such that $A \subset B_r (z)$. Let $\varepsilon \leq \rho$, and $T$ a subset of $A$ verifying:
	\begin{equation}
		\label{eq:packcond}
		\mathcal{d} (x, y) > \varepsilon \text{ for all } x, y \in T, x \neq y.
	\end{equation}
	Then:
	\begin{equation*}
		\bigcup_{ x \in T } B_{\varepsilon/2} (x) \subset B_{r+\rho} (z).
	\end{equation*}
	
	Since $B_{\varepsilon/2} (x) \cap B_{\varepsilon/2} (y) = \emptyset$ for all $x, y \in T, x \neq y$, one deduces that
	\begin{equation*}
		\mathrm{Vol}_M \big(  B_{r+\rho} (z) \big) \geq \mathrm{Vol}_M \left(  \bigcup_{ x \in T } B_{\varepsilon/2} (x) \right) = \sum_{x \in T}  \mathrm{Vol}_M \big(  B_{\varepsilon/2} (x) \big).
	\end{equation*}
	Then, Lemma~\ref{lemma:balls} yields
	\begin{equation*}
		\mathrm{Vol}_M \big(  B_{r+\rho} (z) \big) 
		\geq
		\sum_{x \in T}  \mathrm{Vol}_M \big(  B_{\varepsilon/2} (x) \big)
		\geq
		\mathrm{card} (T) \frac{a \varepsilon^d}{2^d}.
	\end{equation*}
	Consequently,
	\begin{equation*}
		\mathrm{card} (T) \leq  K \varepsilon^{-d},
		\qquad
		\text{with }
		K = \frac{2^d \mathrm{Vol}_{M} \big( B_{r+\rho} (z) \big)}{a}.
	\end{equation*}
	The previous inequality holds for every $T \subset A$ verifying \eqref{eq:packcond}, so:
	\begin{equation*}
		D(\varepsilon, A) \leq K \varepsilon^{-d}. \qedhere
	\end{equation*}
\end{proof}

The proof of Theorem~\ref{th:manifold} relies on the Rauch Theorem, which states that bounding the sectional curvatures of a manifold implies in turn bounding the norm of the differential of the exponential map. The version of the Rauch Theorem included here is taken from \cite{Dyer2015} (Lemma 8, p. 105). In the statement of the theorem the identification between the tangent spaces of a tangent space and the tangent space itself is implicitly used.

\begin{theorem}[Rauch Theorem, \citealp{Dyer2015}]
	Assume that all sectional curvatures of $M$ are bounded by $\kappa^{-} \leq \kappa_M \leq \kappa^{+}$. Taking $v \in T_x M$ with $\norm{v} = 1$, one has for any $w \in T_x M, w \perp v$ that
	\[\frac{S_{\kappa^{+}} (r)}{r} \norm{w} \leq \norm{ (d \exp_x)_{rv} w} \leq \frac{S_{\kappa^{-}} (r)}{r} \norm{w}\]
	for all $x \in M$ and all $r < \min \left\lbrace i_M, \dfrac{\pi}{2 \sqrt{\max \{ \abs{\kappa^{-}}, \abs{\kappa^{+}} \}}}\right\rbrace$; where
	\begin{equation*}
		S_{\kappa} (r) =
		\left\lbrace
		\begin{array}{ll}
			\frac{1}{\sqrt{\kappa}} \sin (\sqrt{\kappa} r), & \kappa > 0 \\
			r, & \kappa = 0 \\
			\frac{1}{\sqrt{\kappa}} \sinh (\sqrt{-\kappa} r), & \kappa < 0.
		\end{array}
		\right.
	\end{equation*}
\end{theorem}

Since $d \exp_x$ preserves the length of any radial vector $v \in T_x M$, a bound on the distortion of the length of a vector orthogonal to $v$ means a bound for all tangent vectors. Furthermore, given any $\kappa > 0$, one can show using Taylor's theorem that:
\begin{equation*}S_{\kappa} (r) \geq r - \frac{\kappa r^3}{6},
\end{equation*}
for all $r \geq 0$. Therefore, the next result holds:

\begin{cor}
	\label{cor:raunch}
	Suppose that all sectional curvatures of $M$ are bounded by $0 \leq \kappa_M < \kappa$. If $v \in T_x M$ is such that $\norm{v} = r < \min \left\lbrace i_M, \frac{\pi}{2 \sqrt{ \kappa } }\right\rbrace$, then for all $w \in T_v (T_x M) \simeq T_x M$ one has
	\begin{equation*}
		\left( 1 - \frac{\kappa r^2}{6} \right) \norm{w} \leq \norm{(d \exp_x)_{v} w} \leq \norm{w}
	\end{equation*} 
	for any $x \in M$. 
\end{cor}

Corollary \ref{cor:raunch} clarifies how the exponential map distorts distance and volume. The following two lemmas concretize how this distortion actually works.

\begin{lemma}
	\label{lemma:contraction}
	Suppose that $M$ is under the conditions of Theorem 4.1. Then, given any $x \in M$, the exponential mapping:
	\begin{equation*}
		\exp_x: B_{\rho} (0) \subset T_x M \longrightarrow M
	\end{equation*}
	is a contraction. That is:
	\begin{equation*}
		\mathcal{d} \big( \exp_x (v_1), \exp_x (v_2) \big) \leq \norm{v_1 - v_2}, \quad \forall v_1, v_2 \in B_{\rho} (0);
	\end{equation*}
	where $\rho = \min \left\lbrace i_M, \frac{\pi}{2 \sqrt{\kappa} }\right\rbrace$.
\end{lemma}

\begin{proof}
	Given any $v_1, v_2 \in B_{\rho} (0) \subset T_x M$, consider the curve:
	\begin{equation*}
		\begin{array}{rcl}
			\alpha: [0,1] & \longrightarrow & B_{\rho} (0);\\
			t & \longmapsto & (1 - t) v_1 + t v_2.
		\end{array}
	\end{equation*}
	
	Then $\exp_x \circ \alpha$ is a curve from $\exp_x (v_1)$ to $\exp_x (v_2)$. Since the distance between two points on $M$ is smaller that the length of any curve joining them, one deduces that
	\begin{equation*}
		\mathcal{d} \big( \exp_x (v_1), \exp_x (v_2) \big) \leq \int_{0}^{1} \norm{ (\exp_x \circ \alpha)' (t)} dt = \int_{0}^{1} \norm{ (d \exp_x)_{\alpha (t)} (\alpha' (t) )} dt
	\end{equation*}
	Corollary \ref{cor:raunch} guarantees that $\norm{ (d \exp_x)_{\alpha (t)} (\alpha' (t) )} \leq \norm{ \alpha' (t) }$ for all $t \in [0, 1]$. So
	\begin{equation*}
		\mathcal{d} \big( \exp_x (v_1), \exp_x (v_2) \big) \leq \int_{0}^{1} \norm{ \alpha' (t) } dt = \norm{v_1 - v_2}. \qedhere
	\end{equation*}
\end{proof}

\begin{lemma}
	\label{lemma:area}
	Suppose that $M$ is a $d$-dimensional manifold under the conditions of Theorem 4.1. Let $x$ be an arbitrary point on $M$ and $\exp_x$ the exponential mapping.
	
	Any measurable set $A \subset B_{\rho} (0) \subset T_x M \simeq \R^d$ satisfies that
	\begin{equation*}
		\mathrm{Vol}_M \big( \exp_x (A) \big) \geq \left( 1 - \frac{\kappa \rho^2}{6}\right)^d \mathrm{Vol}_{\R^d} (A);
	\end{equation*}
	where $\rho = \min \left\lbrace i_M, \frac{\pi}{2 \sqrt{ \kappa } }\right\rbrace$.
\end{lemma}

\begin{proof}
	The mapping
	\begin{equation*}
		\begin{array}{rcl}
			\exp_x: A \subset \R^d & \longrightarrow & M \\
			y & \longmapsto & \exp_x (y)
		\end{array}
	\end{equation*}
	is a parametrization of the set $\exp_x (A) \subset M$. So:
	\begin{equation*}
		\mathrm{Vol}_M \big( \exp_x (A) \big) = \int_{A} \sqrt{ \det \big( G (y) \big) } dy_1 \ldots d y_d;
	\end{equation*}
	where $G(y)$ is a $(d \times d)$ matrix with coefficients
	\begin{equation*}
		G_{ij} (y) = \big\langle (d \exp_x)_y e_i, (d \exp_x)_y e_j \big\rangle_M,
	\end{equation*}
	$\langle \cdot, \cdot \rangle_M$ is the Riemannian metric of $M$, and $e_i \in \R^d$ denotes the $i$-th vector of the standard basis, i.e., $e_i$ is the vector with a $1$ in the $i$-th coordinate and $0$'s elsewhere.
	
	Corollary \ref{cor:raunch} ensures that
	\begin{multline*}
		\left( 1 - \frac{\kappa \norm{y}^2}{6} \right) \norm{w} \leq \norm{(d \exp_x)_{y} w} \Rightarrow \left( 1 - \frac{\kappa \norm{y}^2}{6} \right)^2 \norm{w}^2 \leq \norm{(d \exp_x)_{y} w}^2 \Rightarrow \\
		\Rightarrow \left( 1 - \frac{\kappa \rho^2}{6} \right)^2 \leq \left( 1 - \frac{\kappa \norm{y}^2}{6} \right)^2 \leq \frac{\norm{(d \exp_x)_{y} w}^2}{\norm{w}^2} = \frac{w' G (y) w}{w' w}
	\end{multline*}
	holds for all $y \in B_{\rho} (0)$ and all $w \in \R^d \simeq T_x M$. Hence, all the eigenvalues of $G (y)$ are larger than $\left( 1 - \kappa \rho^2/6 \right)^2$, and, therefore,
	\begin{equation*}
		\det (G) \geq \left( 1 - \frac{\kappa \rho^2}{6} \right)^{2d}.
	\end{equation*}
	Finally, one deduces that 
	\begin{equation*}
		\mathrm{Vol}_M \big( \exp_x (A) \big) 
		\geq \left( 1 - \frac{\kappa \rho^2}{6} \right)^{d} \int_{A} d y_1 \ldots d y_d
		= \left( 1 - \frac{\kappa \rho^2}{6} \right)^{d} \mathrm{Vol}_{\R^d} (A). \qedhere
	\end{equation*}
\end{proof}

Lemmas~\ref{lemma:contraction} and \ref{lemma:area} imply Theorem~\ref{th:manifold}. First, one proves that Assumption~\ref{ass:M2} holds when $r_1$ and $r_2$ are small enough, and one then slowly extends the result for any radii. These results are collected in Lemmas~\ref{lemma:M3.1} and~\ref{lemma:condMrad} below.

\begin{lemma}
	\label{lemma:M3.1}
	Suppose that $M$ is a $d$-dimensional manifold under the conditions of Theorem 4.1.
	
	Let $0 < r_1, r_2 \leq \rho/2$ with $\rho = \min \left\lbrace i_M, \frac{\pi}{2 \sqrt{ \kappa } }\right\rbrace$. Let $x_1, x_2$ be two points on $M$ such that $\mathfrak{e} = r_1+r_2-\mathcal{d}(x_1,x_2) > 0$.
	
	Then:
	\begin{equation*}
		\mathrm{Vol}_{M}  \big( B_{r_1} [x_1] \cap B_{r_2} [x_2] \big) \geq a r^{(d-1)/2} \mathfrak{e}^{(d+1)/2}
	\end{equation*}
	where $r = \min \{ r_1, r_2 \}$ and $a > 0$ is a constant.
\end{lemma}

\begin{proof}
	Consider the exponential map centered in $x_1$:
	\begin{equation*}
		\begin{array}{rcl}
			\exp_{x_1}: B_{\rho} (0) \subset \R^d & \longrightarrow & M \\
			v & \longmapsto & \exp_{x_1} (v).
		\end{array}
	\end{equation*}
	
	Let $v_2 \in B_{\rho} (0)$ a preimage of $x_2$. Lemma \ref{lemma:contraction} guarantees that $\exp_{x_1}$ is a contraction in $B_{\rho} (0)$. Consequently,
	\begin{equation*}
		\exp_{x_1} \big( B_{r_1} [0] \cap B_{r_2} [v_2] \big) \subset B_{r_1} [x_1] \cap B_{r_2} [x_2].
	\end{equation*}
	
	Hence, Lemma \ref{lemma:area} guarantees that
	\begin{align*}
		\mathrm{Vol}_{M}  \big( B_{r_1} [x_1] \cap B_{r_2} [x_2] \big) & \geq \mathrm{Vol}_{M} \Big( \exp_{x_1} \big( B_{r_1} [0] \cap B_{r_2} [v_2] \big) \Big) \\
		& \geq \left( 1 - \frac{\kappa \rho^2}{6}\right)^d \mathrm{Vol}_{\R^d} \big( B_{r_1} [0] \cap B_{r_2} [v_2] \big).
	\end{align*}
	And Lemma 1 from \cite{Walther1997} yields
	\begin{equation*}
		\mathrm{Vol}_{M}  \big( B_{r_1} [x_1] \cap B_{r_2} [x_2] \big) 
		\geq \left( 1 - \frac{\kappa \rho^2}{6}\right)^d c_d r^{(d-1)/2} \mathfrak{e}^{(d+1)/2};
	\end{equation*}
	where
	\begin{equation*}
		c_d = \frac{\pi^{(d-1)/2}}{ 2^{(d-3)/2} (d+1) \Gamma \big( 1 + (d-1)/2 \big)}.
	\end{equation*}
	Then, taking $a = \left( 1 - \frac{\kappa \rho^2}{6}\right)^d c_d$ ends the proof.
\end{proof} 

\begin{lemma}
	\label{lemma:condMrad}
	Suppose that $M$ is a $d$-dimensional manifold under the conditions of Theorem 4.1.
	
	Let $r_1, r_2 > 0$, and let $x_1, x_2 \in M$ such that $\mathfrak{e} = r_1+r_2-\mathcal{d}(x_1,x_2) > 0$. Assume that $\mathfrak{e} \leq \rho/2$ with $\rho = \min \left\lbrace i_M, \frac{\pi}{2 \sqrt{ \kappa } }\right\rbrace$.
	
	Then:
	\begin{equation*}
		\mathrm{Vol}_{M}  \big( B_{r_1} [x_1] \cap B_{r_2} [x_2] \big) \geq a r^{(d-1)/2} \mathfrak{e}^{(d+1)/2}
	\end{equation*}
	where $r = \min \{ r_1, r_2, \rho/2 \}$ and $a > 0$ is a constant.
\end{lemma}

\begin{proof}
	Since $\mathfrak{e} \leq r_1 + r_2$ and $\mathfrak{e} \leq \rho / 2$, it follows that
	\begin{equation}
		\label{eq:Delta}
		\mathfrak{e} \leq \min \{ r_1 + r_2, \rho / 2 \} \leq 
		\tilde{r_1} + \tilde {r_2};
	\end{equation}
	where $\tilde{r_1} = \min \{ r_1 , \rho / 2 \}$ and $\tilde{r_2} = \min \{ r_2 , \rho / 2 \}$.
	
	Let $\alpha$ be a geodesic segment from $x_1$ to $x_2$, $\alpha: \big[ 0, \mathcal{d} (x_1, x_2) \big] \rightarrow M$. From \eqref{eq:Delta} one easily deduces that
	$$
	r_1 - \tilde{r}_1 \leq 
	\mathcal{d} (x_1, x_2)
	\quad
	\text{and}
	\quad
	\mathcal{d}(x_1, x_2) - r_2 + \tilde{r}_2
	\geq 0.
	$$
	
	\begin{figure}[htbp]
		\centering
		\def\scale{1}
\begingroup%
  \makeatletter%
  \providecommand\color[2][]{%
    \errmessage{(Inkscape) Color is used for the text in Inkscape, but the package 'color.sty' is not loaded}%
    \renewcommand\color[2][]{}%
  }%
  \providecommand\transparent[1]{%
    \errmessage{(Inkscape) Transparency is used (non-zero) for the text in Inkscape, but the package 'transparent.sty' is not loaded}%
    \renewcommand\transparent[1]{}%
  }%
  \providecommand\rotatebox[2]{#2}%
  \newcommand*\fsize{\dimexpr\f@size pt\relax}%
  \newcommand*\lineheight[1]{\fontsize{\fsize}{#1\fsize}\selectfont}%
  \ifx\svgwidth\undefined%
    \setlength{\unitlength}{300.79965186bp}%
    \ifx\svgscale\undefined%
      \relax%
    \else%
      \setlength{\unitlength}{\unitlength * \real{\svgscale}}%
    \fi%
  \else%
    \setlength{\unitlength}{\svgwidth}%
  \fi%
  \global\let\svgwidth\undefined%
  \global\let\svgscale\undefined%
  \makeatother%
  \begin{picture}(1,0.12264616)%
    \lineheight{1}%
    \setlength\tabcolsep{0pt}%
    \put(0,0){\includegraphics[width=\unitlength,page=1]{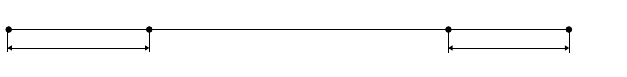}}%
    \put(0.06758698,0.00519859){\color[rgb]{0,0,0}\makebox(0,0)[lt]{\lineheight{1.25}\smash{\begin{tabular}[t]{l}$r_1 - \tilde{r}_1$\end{tabular}}}}%
    \put(0.76088482,0.00519859){\color[rgb]{0,0,0}\makebox(0,0)[lt]{\lineheight{1.25}\smash{\begin{tabular}[t]{l}$r_2 - \tilde{r}_2$\end{tabular}}}}%
    \put(-0.00138387,0.09946551){\color[rgb]{0,0,0}\makebox(0,0)[lt]{\lineheight{1.25}\smash{\begin{tabular}[t]{l}$x_1$\end{tabular}}}}%
    \put(0.89020493,0.09946551){\color[rgb]{0,0,0}\makebox(0,0)[lt]{\lineheight{1.25}\smash{\begin{tabular}[t]{l}$x_2$\end{tabular}}}}%
    \put(0.69694542,0.09943559){\color[rgb]{0,0,0}\makebox(0,0)[lt]{\lineheight{1.25}\smash{\begin{tabular}[t]{l}$\tilde{x}_1$\end{tabular}}}}%
    \put(0.22205436,0.09943559){\color[rgb]{0,0,0}\makebox(0,0)[lt]{\lineheight{1.25}\smash{\begin{tabular}[t]{l}$\tilde{x}_2$\end{tabular}}}}%
  \end{picture}%
\endgroup%

		\caption{\label{fig:lemma10}The points $x_1, \tilde{x}_1, x_2$ and $\tilde{x}_2$ and the distance between them.}
	\end{figure}
	
	Therefore, one can define $\tilde{x}_1 = \alpha \big( r_1 - \tilde{r}_1 \big)$ and $\tilde{x}_2 = \alpha \big( \mathcal{d} (x_1, x_2) -( r_2 - \tilde{r}_2 )\big)$ (check Figure \ref{fig:lemma10}). Then:
	\begin{equation*}
		\mathcal{d} (\tilde{x}_1, \tilde{x}_2) = \mathcal{d} (x_1, x_2) - (r_2 - \tilde{r}_2) - (r_1 - \tilde{r}_1) = \tilde{r}_1 + \tilde{r}_2 - \mathfrak{e}.
	\end{equation*}
	
	Moreover, given any $x \in B_{\tilde{r}_1} ( \tilde{x}_1 ) \cap B_{\tilde{r}_2} ( \tilde{x}_2 )$ one deduces
	\begin{itemize}
		\item $\mathcal{d} (x, x_1) \leq \mathcal{d} (x, \tilde{x}_1) + \mathcal{d} (\tilde{x}_1, x_1) \leq \tilde{r}_1 + r_1 - \tilde{r}_1 = r_1 \Rightarrow x \in B_{r_1} (x_1)$.
		\item $\mathcal{d} (x, x_2) \leq \mathcal{d} (x, \tilde{x}_2) + \mathcal{d} (\tilde{x}_2, x_2) \leq \tilde{r}_2 + r_2 - \tilde{r}_2 = r_2 \Rightarrow x \in B_{r_2} (x_2)$.
	\end{itemize}
	So, $B_{\tilde{r}_1} ( \tilde{x}_1 ) \cap B_{\tilde{r}_2} ( \tilde{x}_2 ) \subset B_{{r}_1} ( {x}_1 ) \cap B_{{r}_2} ( {x}_2 )$ and
	\begin{equation*}
		\mathrm{Vol}_M \big(  B_{{r}_1} ( {x}_1 ) \cap B_{{r}_2} ( {x}_2 ) \big) \geq \mathrm{Vol}_M \big(  B_{\tilde{r}_1} ( \tilde{x}_1 ) \cap B_{\tilde{r}_2} ( \tilde{x}_2 ) \big).
	\end{equation*}
	Finally, Lemma \ref{lemma:M3.1} guarantees that there exists a constant $a > 0$ such that
	\begin{align*}
		\mathrm{Vol}_M \big(  B_{{r}_1} ( {x}_1 ) \cap B_{{r}_2} ( {x}_2 ) \big) & \geq \mathrm{Vol}_M \big(  B_{\tilde{r}_1} ( \tilde{x}_1 ) \cap B_{\tilde{r}_2} ( \tilde{x}_2 ) \big) \\
		& \geq a \min \{ r_1, r_2, \rho / 2 \}^{(d-1)/2} \mathfrak{e}^{(d+1)/2},
	\end{align*}
	concluding the proof.
\end{proof}

\begin{proof}[Proof of Theorem~\ref{th:manifold}]
	Take $\rho = \min \left\lbrace i_M, \frac{\pi}{2 \sqrt{ \kappa } }\right\rbrace$. If $r_1 + r_2 - \mathcal{d} (x_1, x_2) \leq \rho /2$, the result is a direct consequence of Lemma \ref{lemma:condMrad}.
	Suppose then that $r_1 + r_2 - \mathcal{d} (x_1, x_2) > \rho /2$. We distinguish two cases:
	\begin{enumerate}[label = \textit{Case \arabic*:}, wide = \parindent, leftmargin = 0pt]
		\item If $r_1 - \rho/4 \geq \mathcal{d} (x_1, x_2)$, take
		\begin{equation*}
			\tilde{r}_1 = \mathcal{d} (x_1, x_2) + \rho / 4, \quad \tilde{r}_2 = \min \{r_2, \rho / 4\}, \quad \mathfrak{e} = \tilde{r}_1 + \tilde{r}_2 - \mathcal{d} (x_1, x_2).
		\end{equation*}
		Hence:
		\begin{equation*}
			\mathfrak{e} = \mathcal{d} (x_1, x_2) + \rho/4 + \min \{r_2, \rho / 4\} - \mathcal{d} (x_1, x_2) = \rho/4 + \min \{r_2, \rho / 4\} \leq \rho/2.
		\end{equation*}
		So, Lemma \ref{lemma:condMrad} guarantees that there exists a constant $a > 0$ such that
		\begin{equation*}
			\mathrm{Vol}_{M}  \big( B_{\tilde{r}_1} [x_1] \cap B_{\tilde{r}_2} [x_2] \big) \geq a \min \{ \tilde{r}_1, \tilde{r}_2, \rho/2 \}^{(d-1)/2} \mathfrak{e}^{(d+1)/2}.
		\end{equation*}
		Furthermore, $\rho / 4 \leq \tilde{r}_1 \leq r_1$ and $\rho/4 \leq \tilde{r}_2 \leq r_2$. Therefore:
		\begin{align*}
			\mathrm{Vol}_{M}  \big( B_{r_1} [x_1] \cap B_{r_2} [x_2] \big) \geq& a \min \{ \tilde{r}_1, \tilde{r}_2, \rho/2 \}^{(d-1)/2} \mathfrak{e}^{(d+1)/2} \\
			\geq& a \min \{ r_1, r_2, \rho/4 \}^{(d-1)/2} \mathfrak{e}^{(d+1)/2} \\
			\geq& a \min \{ r_1, r_2, \rho/4 \}^{(d-1)/2} \min \{ \varepsilon, \rho / 4 \} ^{(d+1)/2};
		\end{align*}
		where the last inequality follows from $\mathfrak{e} > \rho / 4$.
		
		\item If $r_1 - \rho/4 < \mathcal{d} (x_1, x_2)$, take 
		\begin{equation*}
			\tilde{r}_2 = \rho/2 + \mathcal{d} (p_1, p_2) - r_1, \quad \mathfrak{e} = r_1 + \tilde{r}_2 - d(x_1, x_2).
		\end{equation*}
		Hence:
		\begin{equation*}
			\mathfrak{e} = r_1 + \rho/2 + \mathcal{d}(x_1, x_2) - r_1 - \mathcal{d}(x_1, x_2) = \rho/2.
		\end{equation*}
		Then, Lemma \ref{lemma:condMrad} ensures that there exists a constant $a > 0$ such that
		\begin{equation*}
			\mathrm{Vol}_{M}  \big( B_{r_1} [x_1] \cap B_{\tilde{r}_2} [x_2] \big) \geq a \min \{ r_1, \tilde{r}_2, \rho/2 \}^{(d-1)/2} (\rho / 2) ^{(d+1)/2}.
		\end{equation*}
		From $r_1 + r_2 - \mathcal{d} (x_1, x_2) > \rho /2$ one deduces
		\begin{equation*}
			\tilde{r}_2 =  \rho/2 + \mathcal{d} (x_1, x_2) - r_1 < r_2.
		\end{equation*}
		So,
		\begin{equation*}
			\mathrm{Vol}_{M}  \big( B_{r_1} [x_1] \cap B_{r_2} [x_2] \big) \geq a \min \{ r_1, \tilde{r}_2, \rho/2 \}^{(d-1)/2} (\rho / 2 )^{(d+1)/2}.
		\end{equation*}
		In addition, it follows from $r_1 - \rho/4 < \mathcal{d}(x_1, x_2)$ that:
		\begin{equation*}
			\tilde{r}_2 = \rho/2 + \mathcal{d} (x_1, x_2) - r_1 = \mathcal{d}(x_1, x_2) - (r_1 - \rho/4) + \rho/4 > \rho/4.
		\end{equation*}
		Consequently,
		\begin{align*}
			\mathrm{Vol}_{M}  \big( B_{r_1} [x_1] \cap B_{r_2} [x_2] \big) 
			\geq& a \min \{ r_1, \tilde{r}_2, \rho/2 \}^{(d-1)/2} (\rho / 2 )^{(d+1)/2} \\
			\geq& a \min \{ r_1, r_2, \rho/4 \}^{(d-1)/2} (\rho / 2 )^{(d+1)/2} \\
			\geq& a \min \{ r_1, r_2, \rho/4 \}^{(d-1)/2} \min \{ \varepsilon, \rho / 4 \}^{(d+1)/2},
		\end{align*}
		which ends the proof. \qedhere
	\end{enumerate}
\end{proof}

\subsection{About Assumptions L}
\label{sec:proofsassA}

This section contains the proofs of Propositions~\ref{prop:assA3} and~\ref{prop:assA4} in Section~\ref{sec:assA}. We start with Proposition~\ref{prop:assA3}.

\begin{proof}[Proof of Proposition~\ref{prop:assA3}]
	$M$ is a complete Riemannian manifold by assumption. So, the Hopf-Rinow Theorem (see \citealp{Gallot2004}, Cor.~2.105) ensures that every bounded and closed subset of $M$ is a compact set. $L(\lambda)$ is a closed set since $f$ is continuous. We will prove that is bounded by contradiction.
	
	If $L(\lambda) = \emptyset$, $L(\lambda)$ is compact. Suppose that $L (\lambda) \neq \emptyset$ is not bounded. Let $x_1 \in L(\lambda)$. Since $L(\lambda)$ is not bounded, there exists a point $x_2 \in L(\lambda)$ such that $\mathcal{d} (x_1, x_2) > 4 \delta$. Iterating this argument, one gets a sequence $\{x_n\}_{n \in \N} \subset L(\lambda)$ verifying that 
	\begin{equation*}
		\min_{1 \leq i < n}  \mathcal{d} (x_n, x_i)  > 4 \delta, \quad \forall n \geq 2. 
	\end{equation*}

	By hypothesis, $L(\lambda) = \Psi_{\delta} \big[ L (\lambda) \big]$. Then, for each $x_n \in L(\lambda)$ there exists a point $y_n \in L(\lambda)$ such that $x_n \in B_{\delta} [y_n] \subset L (\lambda)$. Thus, given $i, j \in \N, i \neq j$:
	\begin{equation*}
		4 \delta < \mathcal{d} (x_i, x_j) \leq \mathcal{d} (x_i, y_i) + \mathcal{d} (y_i, y_j) + \mathcal{d} (y_j, x_j) \leq \mathcal{d} (y_i, y_j) + 2 \delta.
	\end{equation*}
	Consequently, $\mathcal{d} (y_i, y_j) > 2 \delta$ and $B_{\delta} [y_i] \cap B_{\delta} [y_j] = \emptyset$. Therefore, Lemma~\ref{lemma:balls} yields
	\begin{equation*}
		\mathrm{Vol}_{M} \big[ L (\lambda) \big] \geq \mathrm{Vol}_{M} \bigg( \bigcup_{n \in N} B_{\delta} [y_i] \bigg) = \sum_{n \in \N} \mathrm{Vol}_{M} \big( B_{\delta} [y_i] \big) \geq \sum_{n \in \N} a \min \{r, \rho \}^d = + \infty;
	\end{equation*}
	where $a, \rho > 0$ are the constants present in Assumption~\ref{ass:M2}. But then $$\Prob \big[ L (\lambda) \big] \geq \lambda \mathrm{Vol}_{M} \big[ L (\lambda) \big] = + \infty,$$ and that is a contradiction.
\end{proof}

The proof of Proposition~\ref{prop:assA4} makes use of the integral flow generated by $\nabla f$ (see \citealp{Lee2012}, Ch. 9, for an introduction to flows) to establish an explicit relation between the level sets of $f$ and the geodesic distance of $M$.

\begin{proof}[Proof of Proposition~\ref{prop:assA4}]
	Let $\eta \leq \delta$ and $\lambda \in [l, u]$. By definition $L(\lambda + \eta) \subset L (\lambda)$. So, in order to prove $\mathcal{d}_H \big( L (\lambda), L(\lambda + \eta) \big) \leq \eta/c$, we only need to show that
	\begin{equation*}
		L (\lambda) \subset L (\lambda + \eta) \oplus (\eta/c) B.
	\end{equation*}
	
	Let $p \in L (\lambda)$. Consider the following vector field in $U$:
	$$
	X = \frac{\nabla f}{\norm{\nabla f}}.
	$$
	Let $\alpha$ be the maximal integral curve of $X$ such that $\alpha (0) = p$: 
	\begin{equation*}
		\alpha: I \rightarrow U; \qquad \alpha' (t) = \frac{\nabla f \big( \alpha (t) \big)}{\norm{\nabla f \big( \alpha (t) \big)}}, \quad \forall t \in I.
	\end{equation*}
	Notice that $\norm{\alpha ' (t)} = 1$ for all $t \in I$. Then,
	\begin{equation}
		\label{eq:alphaanddist}
		\mathcal{d} \big( p, \alpha (t) \big)
		=
		\mathcal{d} \big( \alpha (0), \alpha (t) \big)
		\leq
		\abs{t},
		\quad
		\forall t \in I.
	\end{equation}
	Furthermore, given $t \in I$, we deduce:
	\begin{align*}
		f \big( \alpha (t) \big) - \lambda
		&=
		f \big( \alpha (t) \big) - f \big( \alpha (0) \big)
		\\
		&= 
		\int_{0}^{t} \langle \nabla f \big( \alpha (s) \big), \alpha' (s) \rangle_M ds
		\\
		&=
		\int_{0}^{t} \norm{ \nabla f \big( \alpha (s) \big) } ds \geq \int_{0}^{t} c ds = ct.
	\end{align*}
	Hence,
	\begin{equation}
		\label{eq:fandalpha}
		f \big( \alpha (t) \big) \geq \lambda + ct, \quad \forall t \in I.
	\end{equation}
	
	We distinguish two cases.
	\begin{itemize}
		\item If $\eta/c \in I$, then \eqref{eq:fandalpha} ensures that $\alpha (\eta/c) \in L( \lambda + \eta )$. Moreover, \eqref{eq:alphaanddist} guarantees that $\mathcal{d} \big( p, \alpha (\eta/c) \big) \leq \eta/c$. Hence, $p \in L( \lambda + \eta ) \oplus (\eta/c) B$.
		\item If $\eta/c \notin I$, then $\eta/c$ is an upper bound of $I$. The Escape Lemma (Theorem 9.20, \citealp{Lee2012}) guarantees that $\alpha \big( I \cap [0, \eta/c] \big) \not\subset f^{-1} \big( [\lambda, \lambda + \eta] \big)$. I.e., there exists a $\tilde{t} \in I \cap [0, \eta/c]$ such that $\alpha(\tilde{t}) \notin f^{-1} \big( [\lambda, \lambda + \eta] \big)$. Since $\tilde{t} \geq 0$, \eqref{eq:fandalpha} ensures that $\alpha (\tilde{t}) \in L (\lambda + \eta)$. In addition, $\mathcal{d} \big( p, \alpha (\tilde{t}) \big) \leq \eta/c$ by \eqref{eq:alphaanddist}. Consequently, $p \in L( \lambda + \eta ) \oplus (\eta/c) B$.
	\end{itemize}
	
	The proof of $\mathcal{d}_H \big( L (\lambda), L(\lambda - \eta) \big) \leq \eta/c$ is totally analogous.
\end{proof}

\subsection{About Assumption P \label{sec:assA5}}
\label{sec:proofassP}

This section is fully devoted to the proof of Proposition~\ref{prop:assA5} in Section~\ref{sec:lambdan}. The basic idea is to make use of the flow generated by $\nabla f$ again, this time to connect the level sets of $f$ and the volume measure of $M$. This allows to prove that the function $\lambda \mapsto \Prob \big[ L (\lambda) \big]$ is derivable, and its derivative is continuous and bounded away from zero. Then, one can translate that bound to the derivative of the function $\gamma \mapsto \lambda_{\gamma}$ through the Inverse Function Theorem.

First, Proposition~\ref{prop:difprobs} below ensures the derivability of $\Prob \big[ L (\lambda) \big]$.

\begin{prop}
	\label{prop:difprobs}
	Let $M$ be a Riemannian manifold under Assumptions M, and let \mbox{$f: M \rightarrow \R$} be a density function. Let $0 < l \leq u < \sup f$, and let $U$ be an open set such that $f^{-1} \big( [l, u] \big)\subset U$. Suppose that $f$ is a $\mathcal{C}^{1}$ function in $U$ and there exists a positive constant $c > 0$ satisfying such that $\norm{ \nabla f } \geq c$ in $U$.
	
	Therefore, the function $ H(\lambda) = P \big[ L (\lambda) \big] $ is derivable in $(l, u)$, its derivative $H'(\lambda)$ is continuous in $(l, u)$, and $H'(\lambda) < 0$ for all $\lambda \in (l, u)$.
\end{prop}

\begin{proof}
	Let $\lambda \in (l, u)$. Then,
	\begin{equation*}
		H(l) - H(\lambda) = \Prob \big[ L (l) \big] - \Prob \big[ L (\lambda ) \big]= \int_{L (l) \setminus L (\lambda)} f d\mathrm{Vol}_{M}.
	\end{equation*}
	The basis of the proof is to rewrite the set $L (l) \setminus L (\lambda)$ in a more convenient way. Consider the following vector field in $U$
	\begin{equation*}
		X = \frac{\nabla f}{\norm{\nabla f}^2}.
	\end{equation*}
	Let $\Upsilon$ be the smooth maximal flow over $U$ with $X$ as infinitesimal generator:
	\begin{equation*}
		\Upsilon: \mathcal{D} \subset U \times \R \rightarrow U.
	\end{equation*}
	Take $\mathcal{O} = \big( f^{-1} (l) \times \R \big) \cap \mathcal{D}$. The next lemma (that will be proven later on) guarantees that the restriction of the flow $\Upsilon$ to this set has nice properties.
	\begin{lemma}
		\label{lemma:thetadiff}
		Consider the set $\mathcal{O} = \big( f^{-1} (l) \times \R \big) \cap \mathcal{D}$. Then:
		\begin{enumerate}[label = (\alph*)]
			\item The map $\Upsilon \vert_{\mathcal{O}}$ is a diffeomorphism between $\mathcal{O}$ and an open submanifold of $M$.
			\item $f^{-1} (l) \times [0, u - l] \subset \mathcal{O}$.
			\item The map $\Upsilon \vert_{\mathcal{O}}$ takes horizontal fibers to level sets of $f$. That is,
			$$
			\Upsilon \vert_{\mathcal{O}}
			\big(
			f^{-1} (l) \times \{\eta\}
			\big)
			=
			f^{-1} (l + \eta),
			$$
			for all $\eta \in [0, u - l]$.
		\end{enumerate} 
	\end{lemma}

	In fact, $\Upsilon \vert_{\mathcal{O}}$ is the required reparametrization of $L (l) \setminus L (\lambda)$. For the sake of simplicity, denote $\Upsilon = \Upsilon \vert_{\mathcal{O}}$. Lemma~\ref{lemma:thetadiff} guarantees that
	\begin{align*}
		\int_{L (l) \setminus L (\lambda)} f d \mathrm{Vol}_{M}
		&=
		\int_{f^{-1} ( [l, \lambda ) )} f d \mathrm{Vol}_{M}
		\\
		&=
		\int_{f^{-1} (l) \times [0, \lambda - l)} (f \circ \Upsilon) \Big( \sqrt{ \abs{\det G} } \Big) d \mathrm{Vol}_{f^{-1} (l) \times \R}
		,
	\end{align*}
	where $G \equiv G(p, t)$ is the $m \times m$ matrix with coefficients
	\begin{equation*}
		G_{ij}: \mathcal{O} \longrightarrow M, \quad G_{ij} = \langle d \Upsilon E_i, d \Upsilon E_j \rangle_{M}
	\end{equation*}
	and $E_1, \ldots, E_n$ is an orthonormal frame defined in a neighborhood of $(p, t) \in \mathcal{O}$ (the determinant of $G$ does not depend on the choice of this frame). Keep in mind that, since $d \Upsilon$ and $E_i$ are continuous, the functions $G_{ij}$ are also continuous in $\mathcal{O}$ for all $i, j \in \{1, \ldots, n\}$, and then $\sqrt{ \abs{\det G} }$ is also continuous in $\mathcal{O}$.
	
	Lemma~\ref{lemma:thetadiff}c ensures that $f \big( \Upsilon (p, t) \big) = l + t$ for all $(p, t) \in f^{-1} (l) \times [0, \lambda - l)$. So,
	\begin{align*}
		\int_{L (l) \setminus L (\lambda)} f d \mathrm{Vol}_{M}
		&=
		\int_{0}^{\lambda - l} (l + t) \bigg( \int_{f^{-1} (l)} \Big( \sqrt{ \abs{\det G (p, t)} } \Big) d \mathrm{Vol}_{f^{-1} (l)} \bigg) dt
		\\
		&=
		\int_{l}^{\lambda} t \bigg( \int_{f^{-1} (l)} \Big( \sqrt{ \abs{\det G (p, t - l)} } \Big) d \mathrm{Vol}_{f^{-1} (l)} \bigg) dt.
	\end{align*}
	
	Therefore,
	\begin{multline*}
		H(\lambda)
		=
		H (l) - \Big( \Prob \big[ L (l) \big] - \Prob \big[ L (\lambda ) \big] \Big) =
		\\
		=
		H (l) - \int_{l}^{\lambda} t \bigg( \int_{f^{-1} (l)} \Big( \sqrt{ \abs{\det G (p, t - l)} } \Big) d \mathrm{Vol}_{f^{-1} (l)} \bigg) dt,
	\end{multline*}
	for all $\lambda \in (l, u)$. So, $H$ is derivable at $\lambda$ and
	\begin{equation*}
		H'(\lambda)
		=
		- \lambda \bigg( \int_{f^{-1} (l)} \Big( \sqrt{ \abs{\det G (p, \lambda - l)} } \Big) d \mathrm{Vol}_{f^{-1} (l)} \bigg).
	\end{equation*}
	Since $\sqrt{ \abs{\det G} }$ is continuous in $\mathcal{O}$ and $f^{-1} (l) \times \{\lambda - l \} \subset \mathcal{O}$ for all $\lambda \in (l, u)$ by Lemma~\ref{lemma:thetadiff}b, $H'$ is continuous in $(l, u)$. Moreover, $\Upsilon$ is a diffeomorphism, so $\abs{\det G} > 0$ in $\mathcal{O}$. Hence, $H' (\lambda) < 0$ for all $\lambda \in (l, u)$.
	
	The only thing left is to prove Lemma~\ref{lemma:thetadiff}.
	
	\begin{proof}[Proof of Lemma~\ref{lemma:thetadiff}]
		By assumption, $\nabla f \neq 0$ in $f^{-1} (l)$. So, $f^{-1} (l)$ is a regular submanifold of $M$ with codimension $1$. Furthermore, $X$ is not tangent to $f^{-1} (l)$ anywhere since the gradient is orthogonal to all the level sets. Hence, the Flowout Theorem \cite[Th. 9.20]{Lee2012} guarantees that $\Upsilon \vert_{\mathcal{O}}$ is a diffeomorphism between $\mathcal{O}$ and an open submanifold of $M$.
		
		The proof of (b) and (c) requires a deeper understanding of the relation between $f$ and the flow $\Upsilon$. Consider the set $\mathcal{D}^{(p)} = \left\lbrace t \in \R \colon (p, t) \in \mathcal{D} \right\rbrace$ and the map $\Upsilon^{(p)} : \mathcal{D}^{(p)} \rightarrow U $ defined by $\Upsilon^{(p)} (t) = \Upsilon (p, t)$. Given any $(p, t) \in \mathcal{D}$, we deduce
		\begin{align*}
			f \big( \Upsilon (p, t) \big) - f(p)
			=& f \big(\Upsilon^{(p)} (t) \big) - f \big(\Upsilon^{(p)} (0) \big) \\
			=& \int_0^t \langle \nabla f \big( \Upsilon^{(p)} (s) \big), \Upsilon^{(p) \prime}  (s) \rangle ds \\
			=& \int_0^t \langle \nabla f \big( \Upsilon^{(p)} (s) \big), X_{\Upsilon^{(p)}  (s)} \rangle ds \\
			=& \int_0^t \frac{\langle \nabla f \big( \Upsilon^{(p)} (s) \big), \nabla f \big( \Upsilon^{(p)} (s) \big) \rangle}{\norm{\nabla f \big( \Upsilon^{(p)} (s) \big)}^2} ds \\
			=& \int_0^t ds = t.
		\end{align*}
		So,
		\begin{equation}
			\label{eq:fandtheta}
			f \big( \Upsilon (p, t) \big) = f(p) + t, \quad \forall (p, t) \in \mathcal{D}.
		\end{equation}
		Furthermore, it holds that:
		\begin{equation}
			\label{eq:Dsuperp}
			\big[ l - f(p), u - f(p) \big] \subset \mathcal{D}^{(p)},
			\quad
			\forall p \in f^{-1} \big( [l , u ] \big).
		\end{equation}
		The latter result is proven by contradiction. Assume that $\big[ l - f(p), u - f(p) \big] \not\subset \mathcal{D}^{(p)}$. Then, either $l - f(p)$ is a lower bound of $\mathcal{D}^{(p)}$, or $u - f(p)$ is an upper bound of $\mathcal{D}^{(p)}$. The reasoning is a little bit different (but analogous) in each case.
		
		\begin{itemize}
			\item If $l - f(p)$ is a lower bound of $\mathcal{D}^{(p)}$, then the Escape Lemma \cite[Lemma 9.19]{Lee2012} guarantees that $\Upsilon^{(p)} \big(\mathcal{D}^{(p)} \cap (l - f(p), 0] \big)$ is not contained in any compact subset of $U$. But, at the same time, \eqref{eq:fandtheta} ensures that $$\Upsilon^{(p)} \big(\mathcal{D}^{(p)} \cap (l - f(p), 0] \big) \subset f^{-1} \big( [l , u ]\big),$$ which is clearly a contradiction.
			\item If $u - f(p)$ is an upper bound of $\mathcal{D}^{(p)}$, then the Escape Lemma again ensures that $\Upsilon^{(p)} \big(\mathcal{D}^{(p)} \cap [0, u - f(p)) \big)$ is not contained in any compact subset of $U$. However, it follows from \eqref{eq:fandtheta} that $$\Upsilon^{(p)} \big(\mathcal{D}^{(p)} \cap [0, u - f(p)) \big) \subset f^{-1} \big( [l, u ]\big),$$ and that is a contradiction.
		\end{itemize}
		
		Once \eqref{eq:fandtheta} and \eqref{eq:Dsuperp} are known, the proof of (b) and (c) is straightforward. In fact, (b) is a direct consequence of \eqref{eq:Dsuperp}. Moreover, \eqref{eq:fandtheta} yields
		$$\Upsilon \vert_{\mathcal{O}}
		\big(
		f^{-1} (l) \times \{\eta\}
		\big)
		=
		\Upsilon
		\big(
		f^{-1} (l) \times \{\eta\}
		\big)
		\subset
		f^{-1} (l + \eta).$$
		Take $p \in f^{-1} (l + \eta)$. From~\eqref{eq:Dsuperp}, it follows that $-\eta \in \mathcal{D}^{(p)}$. Hence, $\Upsilon (p, -\eta)$ exists and \eqref{eq:fandtheta} ensures that
		\begin{equation*}
			f \big( \Upsilon (p, -\eta ) \big) = f(p) - \eta = l + \eta - \eta = l.
		\end{equation*}
		Then, $\Upsilon (p, -\eta ) \in f^{-1} (l)$. Furthermore,
		\begin{equation*}
			\Upsilon \big( \Upsilon ( p, -\eta ), \eta \big) = \Upsilon ( p, 0 ) = p,
		\end{equation*}
		which shows that $p \in \Upsilon
		\big(
		f^{-1} (l) \times \{\eta\}
		\big)$, completing the proof.
	\end{proof}
	\renewcommand{\qedsymbol}{}
\end{proof}

Finally, the proof of Proposition~\ref{prop:assA5} uses the Inverse Function Theorem to translate this bound of the derivative of $\Prob \big[ L (\lambda) \big]$ into a bound of the derivative of the function $\gamma \mapsto \lambda_{\gamma}$.

\begin{proof}[Proof of Proposition~\ref{prop:assA5}]
	Proposition~\ref{prop:difprobs} and Inverse Function Theorem guarantee that $H(\lambda) = \Prob \big[ L(\lambda) \big]$ has a derivable inverse function, $H^{-1}: \big( H(u), H(l) \big) \rightarrow (l, u)$. We will show that this inverse function of $H$ is closely related to $\lambda_{\gamma}$.
	
	Let $\gamma \in \big[ \underline{\gamma} - \delta, \overline{\gamma} + \delta \big]$. From the definition of $\lambda_{\gamma}$ (see (5.5)) it follows that \mbox{$H(\lambda_{\gamma}) \geq 1 - \gamma$}. Assume by contradiction that $H(\lambda_{\gamma}) > 1 - \gamma$. Since $\lambda_{\gamma} \in [\lambda_{\underline{\gamma} - \delta}, \lambda_{\overline{\gamma} + \delta} ] \subset (l, u)$, $H$ is continuous in $\lambda_{\gamma}$. Then, there exists a constant $\epsilon > 0$ such that
	\begin{equation*}
		H( \lambda_{\gamma} + \eta) > 1 - \gamma, \quad \forall \eta \in [0, \epsilon]
	\end{equation*}
	which contradicts the definition of $\lambda_{\gamma}$. Therefore, $H (\lambda_{\gamma}) = 1 - \gamma$ and consequently \hbox{$\lambda_{\gamma} = H^{-1} (1 - \gamma)$} for all $\gamma \in \big[ \underline{\gamma} - \delta, \overline{\gamma} + \delta \big]$. Hence, $\lambda_{\gamma}$ is derivable with respect to $\gamma$ and its derivative is
	\begin{equation*}
		\frac{d \lambda_{\gamma}}{d \gamma} = (H^{-1})' (1 - \gamma) = \frac{-1}{H' (\lambda_{\gamma})},
	\end{equation*}
	which is a continuous function in $\big[ \underline{\gamma} - \delta, \overline{\gamma} + \delta \big]$.
	
	Now consider the function:
	\begin{equation*}
		F(\gamma, \eta)
		= 
		\left\lbrace
		\begin{array}{ll}
			\dfrac{\lambda_{\gamma + \eta} - \lambda_{\gamma}}{\eta}, & \text{if } \eta \in [- \delta, \delta] \setminus \{ 0 \};
			\\[5mm]
			\dfrac{-1}{H' (\lambda_{\gamma})}, & \text{if } \eta = 0.
		\end{array}
		\right.
	\end{equation*}
	From the continuity of $\lambda_{\gamma}$ and $H' (\lambda_{\gamma})$, and the differentiability of $\lambda_{\gamma}$, it follows that $F$ is a continuous function in $[\underline{\gamma}, \overline{\gamma}] \times [-\delta, \delta]$. Then, by compactness, there exists a constant such that $\sup_{(\gamma, \eta) \in [\underline{\gamma}, \overline{\gamma}] \times [-\delta, \delta]} \abs{F(\gamma, \eta)} \leq k$, which implies
	\begin{equation*}
		\sup_{\gamma \in [\underline{\gamma}, \overline{\gamma}]} \abs{\lambda_{\gamma + \eta} - \lambda_{\gamma}} \leq k \abs{\eta}
	\end{equation*}
	for all $\eta \in [-\delta, \delta]$.
\end{proof}

\end{document}